\documentclass[12pt]{amsart}
\usepackage[toc,page]{appendix}
\usepackage[T1]{fontenc}
\usepackage[cp1250]{inputenc}

\usepackage[arrow, matrix, curve]{xy}
\usepackage{amsmath,amsthm,amssymb}
\usepackage{amscd,graphics}
\usepackage{mathabx}
\usepackage{graphicx}

\usepackage[top=2.5cm, bottom=1.5cm, outer=5cm, inner=2cm, heightrounded, marginparwidth=2.25cm, marginparsep=0.2cm]{geometry}
\DeclareMathOperator{\clsp}{\overline{span}}

\DeclareMathOperator{\dashind}{-Ind}
\DeclareMathOperator{\Aut}{Aut}

\DeclareMathOperator{\End}{End}

\DeclareMathOperator{\coker}{coker}
\DeclareMathOperator{\Prim}{Prim}

\newcommand{\Int}{\textrm{Int}\,}

\newcommand{\I}{\mathcal I}
\newcommand{\B}{\mathcal B}

\newcommand{\TT}{\mathcal T}

\newcommand{\II}{\mathcal I}

\newcommand{\K}{\mathcal K}

\newcommand{\x}{\widetilde x}

\renewcommand{\b}{\widetilde b}
\newcommand{\OO}{\mathcal{O}}
\newcommand{\y}{\widetilde y}

\newcommand{\E}{\widetilde{E}}
\newcommand{\X}{\widetilde{X}}

\newcommand{\TDelta}{\widetilde \Delta}

\newcommand{\al}{\alpha}
\newcommand{\p}{\varphi}

\newcommand{\tp}{\widetilde{\varphi}}

\newcommand{\hal}{\widehat \alpha}

\newcommand{\tphi}{\widetilde{\varphi}}

\renewcommand{\L}{\mathcal L}

\newcommand{\A}{\mathcal A}

\newcommand{\C}{\mathbb C}
\newcommand{\R}{\mathbb R}

\newcommand{\Z}{\mathbb Z}
\newcommand{\N}{\mathbb N}

\newtheorem{thm}{Theorem}[section]
\newtheorem{lem}[thm]{Lemma}
\newtheorem{prop}[thm]{Proposition}
\newtheorem{cor}[thm]{Corollary}
\theoremstyle{definition}

\newtheorem{defn}[thm]{Definition}
\newtheorem{ex}[thm]{Example}
\newtheorem{rem}[thm]{Remark}
\newtheorem{convention}[thm]{Convention}

\begin{document}


   \title[Crossed products by endomorphisms of $C_0(X)$-algebras]{Crossed products by endomorphisms of $C_0(X)$-algebras}

\author{B. K.  Kwa\'sniewski}

\address{Department of Mathematics and Computer Science, The University of Southern Denmark, 
Campusvej 55, DK--5230 Odense M, Denmark // Institute of Mathematics,  University  of Bialystok\\
ul. K. Ciolkowskiego 1M, 15-245 Bialystok,   Poland}
\email{bartoszk@math.uwb.edu.pl}



\begin{abstract} 
In the first part of the paper, we develop a theory of crossed products of a $C^*$-algebra $A$ by an arbitrary (not necessarily extendible) endomorphism $\alpha:A\to A$. We consider  relative crossed products $C^*(A,\alpha;J)$ where $J$ is an ideal in $A$, and  describe up to Morita-Rieffel equivalence all gauge-invariant ideals in $C^*(A,\alpha;J)$ and give six term exact sequences determining their $K$-theory. We also obtain certain criteria implying that all ideals in $C^*(A,\alpha;J)$ are gauge-invariant, and that $C^*(A,\alpha;J)$ is purely infinite.

In the second part, we consider a situation where  $A$ is a $C_0(X)$-algebra and  $\alpha$ is such that $\alpha(f a)=\Phi(f)\alpha(a)$, $a\in A$, $f\in C_0(X)$ where $\Phi$ is an endomorphism of $C_0(X)$. Pictorially speaking,  $\alpha$  is a mixture of a topological dynamical system $(X,\varphi)$ dual to $(C_0(X),\Phi)$ and a continuous field of homomorphisms $\alpha_x$ between the fibers $A(x)$, $x\in X$, of the corresponding $C^*$-bundle.

For systems described above, we establish efficient conditions for the uniqueness property, gauge-invariance of all ideals, and  pure infiniteness of  $C^*(A,\alpha;J)$. We apply these results to the case when $X=\Prim(A)$ is a Hausdorff space. In particular, if the associated $C^*$-bundle is trivial, we obtain formulas for $K$-groups of all ideals in  $C^*(A,\alpha;J)$. In this way, we constitute   a large class of crossed products   whose ideal structure and $K$-theory is completely described in terms  of $(X,\varphi,\{\alpha_{x}\}_{x\in X};Y)$ where $Y$ is a closed subset of $X$. 

\end{abstract}
\maketitle


\section*{Introduction.}
 Crossed products by  endomorphisms proved to be one of the major  model examples  in classification of simple $C^*$-algebras. The first instances of such crossed products, informally introduced  in \cite{cuntz}, were Cuntz algebras $\OO_n$.  R{\o}rdam \cite{Rordam1} and R{\o}rdam and Elliot \cite{ER} established the range of $K$-theoretical invariant for all Kirchberg algebras by showing that crossed products by endomorphisms of $A\mathbb{T}$-algebras of real rank zero contain  classifiable Kirchberg algebras with arbitrary $K$-theory. In particular, by Kirchberg-Phillips classification, every Kirchberg algebra is isomorphic to such a crossed product. Significantly, Elliott's classification of (not necessarily simple) $A\mathbb{T}$-algebras of real rank zero \cite{Elliot2} implies that all unital simple $A\mathbb{T}$-algebras of real rank zero with $K_1$ equal to integers are modeled by crossed products associated to  Cantor systems, studied by Putnam \cite{Putnam}, see also \cite{HPS}.  Another milestone  in the classification of non-simple $C^*$-algebras is Kirchberg's classification  of strongly purely infinite,  nuclear, separable $C^*$-algebras via ideal related $KK$-theory \cite{kirchberg}. Nevertheless, this invariant is fairly complicated and there is still a lot of effort put into classifying  certain  non-simple purely infinite  $C^*$-algebras  by means of  apparently less elaborated  invariants, cf. \cite{meyer_nest}, \cite{bonkat}, \cite{Rordam2}.   Accordingly, it is of interest to establish  non-trivial but still accessible examples of $C^*$-algebras whose  ideal structure and $K$-theory of all ideals and quotients can be controlled.  An overall aim of the present paper is to develop tools to construct and analyze a large class of crossed products by endomorphisms  that fulfill these requirements. Another source of motivation comes from potential applications to spectral analysis of certain  non-local operators \cite{Anton_Lebed}, \cite{karlovich}, \cite{BFK}.

In section \ref{morphisms section}, we introduce and study $C_0(X)$-dynamical systems. A  \emph{$C_0(X)$-dynamical system} is a pair $(A,\alpha)$ where $A$ is a $C_0(X)$-algebra and $\alpha:A\to A$ is an endomorphism compatible with the $C_0(X)$-structure (we give several characterization of this notion). Such a system  can be viewed as  a convenient combination of topological  and noncommutative dynamics; encoded in a pair $(\p, \{\alpha_x\}_{x\in \Delta})$ where $\p:\Delta \to X$ is a continuous proper mapping defined on an open set $\Delta\subseteq X$, and $\alpha_x:A(\p(x))\to A(x)$, $x\in \Delta$, is a homomorphism between the corresponding fibers of the $C_0(X)$-algebra $A$, so that 
$$
\alpha(a)(x)=\alpha_x(a(\p(x))), \qquad a\in A, x\in \Delta. 
$$
We refer to the pair $(\p, \{\alpha_x\}_{x\in \Delta})$ as to a \emph{morphism} of the corresponding $C^*$-bundle $\A:=\bigsqcup\limits_{x\in X} A(x)$ (Definition \ref{morphism definition}).
In particular, if every fiber $A(x)$ is trivial (equal to $\C$) we get a topological dynamical system. In the case when  $X$ is trivial (a singleton), $\alpha:A\to A$ is just  an endomorphism of $A$, and  we call $(A,\alpha)$ simply a \emph{$C^*$-dynamical system}. An important non-trivial example arises when $A$ is a unital $C^*$-algebra, $C\subseteq Z(A)$ is a non-degenerate $C^*$-subalgebra of the center of $A$ and  $\alpha$ `almost preserves'  $C$, that is  $\alpha(C)\subseteq C\alpha(1)$. Then $(A,\alpha)$ is naturally a $C_0(X)$-dynamical system with $C_0(X)\cong C$.
Analysis of crossed products associated  to such $C_0(X)$-dynamical systems, in the case  $\alpha$ is an  automorphism,  played an important role in the study of non-local operators, cf. \cite{karlovich}, such as   (abstract) weighted shift operators \cite{Anton_Lebed},  or singular integral operators with shifts \cite{BFK}.  
If $C=Z(A)$ and the primitive ideal space $\Prim(A)$ of $A$ is a Hausdorff space, then $X\cong \Prim(A)$, and this will be our model example.

 We associate to  any $C^*$-dynamical system $(A,\alpha)$ and an ideal $J$ in $A$ the \emph{relative crossed product} $C^*(A,\alpha;J)$ introduced (for an arbitrary completely positive map) in  \cite{kwa-exel}. As explained in detail in \cite[Section 3.4]{kwa-exel}, these crossed products include as special cases  those  studied in \cite{Paschke2}, \cite{Stacey},  \cite{LR}, \cite{kwa-leb}, \cite{kwa-rever}.
In the present paper, we consider only the  case when $A$ embeds into $C^*(A,\alpha;J)$, equivalently  when $J$ is contained in the annihilator $(\ker\alpha)^\bot$ of the kernel of $\alpha$ (the general case may be covered by passing to a quotient  $C^*$-dynamical system, cf. Remark \ref{remark to mentioned in the introduction} below). 
The \emph{unrelative crossed product} is $C^*(A,\alpha):=C^*(A,\alpha;(\ker\alpha)^\bot)$.  If $(A,\alpha)$ is a $C_0(X)$-dynamical system  with the related morphism $(\p, \{\alpha_x\}_{x\in \Delta})$,  then among our main results we list the following: 

$\bullet$  \emph{Isomorphism theorem}. We show that for certain continuous $C_0(X)$-algebras, if the map $\p$ is topologically free outside a set $Y$ related to the ideal $J$, then every injective representation of $(A,\alpha)$ whose ideal of covariance is maximal possible, give rise to a faithful representation of $C^*(A,\alpha;J)$ (see Theorem \ref{uniqueness theorem}).

$\bullet$ \emph{Description of the ideal structure}. We prove  that,  if $\p$ is free, then we have a bijective correspondence between ideals $\I$ in $C^*(A,\alpha;J)$ and certain pairs $(I,I')$ of ideals in $A$, called \emph{$J$-pairs} for $(A,\alpha)$ (see  Theorem  \ref{pure infiniteness theorem} and Definition \ref{J-pairs definition}).  Moreover, the quotient of  $C^*(A,\alpha;J)$ by $\I$ is naturally isomorphic a crossed product associated to the quotient of $(A,\alpha)$, and the ideal $\I$  is Morita-Rieffel (strongly Morita) equivalent either to the crossed product associated to the restricted endomorphism $\alpha|_{I}$ or to an endomorphism constructed from $(I,I')$ and $\alpha$ (see Theorem \ref{gauge-invariant ideals thm} and Proposition \ref{general gauge-invariant ideals description}).
In the case $A$ has a Hausdorff primitive ideal space and $X=\Prim(A)$, we  describe ideals in  $C^*(A,\alpha;J)$ in terms of pairs $(V,V')$ of closed subsets of $X$,  called \emph{$Y$-pairs} for $(X,\p)$  (see Proposition \ref{proposition on ideals in commutative}  and Definition \ref{invariants definition}).  Hence  the ideal structure of $C^*(A,\alpha;J)$ is completely described in terms of the topological dynamical system $(X,\p)$.  In particular, in  this case we characterize  simplicity of crossed products  (Proposition \ref{simplicity result}).

$\bullet$ \emph{Pure infinitneness}. It seems that  amongst the existing technics of showing pure infiniteness of crossed products there are two types of approaches. In the first one the corresponding crossed product is simple \cite{Laca-Spiel}, \cite{Jeong}, \cite{Joli-Robert}. In the second one the initial algebra $A$ is assumed to have the ideal property \cite{rordam_sier}, \cite{gs}, \cite{Ortega-Pardo}, \cite{pp}, \cite{KS}. We cover these two lines of research in our context by showing that if $\p$ is free, then
$$
\begin{array}{l} A  \textrm{ is purely infinite and } \\
\textrm{has the ideal property }
\end{array} \Longrightarrow \textrm{ the same is true for }C^*(A,\alpha,J).
$$
and 
$$
\begin{array}{l} A  \textrm{ is purely infinite and there are } \\
\textrm{finitely many }J\textrm{-pairs for } (A,\alpha)
\end{array} \Longrightarrow  \begin{array}{l} C^*(A,\alpha,J)  \textrm{ is purely infinite and } \\
\textrm{has  finitely many ideals }
\end{array}
$$
(see Theorem \ref{pure infiniteness theorem}).  If $X=\Prim(A)$ and $A$ is purely infinite, this leads us to necessary and sufficient conditions for $C^*(A,\alpha)$ to be a Kirchberg algebra (Corollary \ref{kirchberg classification}).  We recall that in the presence of the ideal property, pure infiniteness is equivalent to strong pure infiniteness. We also point out that using conditions introduced recently in \cite{KS},  the aforementioned results could be potentially generalized to the case when $A$ is not necessarily purely infinite. Moreover,   in  recent papers \cite{ks1}, \cite{ks2} Sierakowski and Kirchberg introduced a new machinery that gives strong pure infiniteness criteria for crossed products by discrete group actions, without passing (explicitly) through pure infiniteness.  It seems plausible that  combining their technics with  tools of the present paper one could also obtain permanence results for strong pure infiniteness of  $C^*(A,\alpha,J)$. Nevertheless, we do not pursue these issues here.

$\bullet$ \emph{$K$-theory}. In the case when the corresponding $C^*$-bundle is trivial, that is when $A=C_0(X,D)$ for a $C^*$-algebra $D$,  and under the assumptions that $X$ is totally disconnected, $K_0(D)$ is torsion free and $K_1(D)=0$,  we give  formulas for $K$-groups of $C^*(A,\alpha,J)$ formulated in terms of  $(X,\p,\{\al_{x}\}_{x\in \Delta};Y)$ (Proposition \ref{K-theory computation}). These  formulas  can be viewed as a far reaching generalization of those  given by Putnam in \cite{Putnam}. If additionally $D$ is simple and $\p$ is free we get formulas for $K$-groups of all ideals in $C^*(A,\alpha,J)$ (Theorem \ref{corollary complicated}). We show by concrete examples that not only the dynamical system $(X,\p)$ (which determines the ideal structure of $C^*(A,\alpha,J)$) but also endomorphisms $\al_{x}$, $x\in \Delta$, contribute to $K$-theory, thus giving us a lot of flexibility in constructing interesting algebras.

The  aforementioned results are based on general facts for  crossed products $C^*(A,\alpha;J)$, which we develop in section \ref{endomorphism section}. One of the main tools  is a description of a reversible $J$-extension $(B,\beta)$ of $(A,\alpha)$, introduced    in \cite{kwa-rever}: we show that if $(A,\alpha)$ is a $C_0(X)$-dynamical system, then  $(B,\beta)$ is a $C_0(\X)$-dynamical system induced by a morphism $(\tp, \{\beta_x\}_{x\in \TDelta})$ where $(X,\tp)$ is a reversible $Y$-extension of $(X,\p)$ introduced in \cite{kwa-logist} (see Theorem \ref{proposition C-bundle B}).

\bigskip 

It has  to be emphasized that so far, cf.  \cite{Paschke2}, \cite{Stacey},  \cite{LR}, \cite{kwa-leb}, \cite{kwa-rever},   crossed products by endomorphisms where studied  either in the case $A$ is unital or under the assumption that  the endomorphism $\alpha:A\to A$ is  \emph{extendible} \cite{adji}, i.e.  that it extends to an endomorphism of the multiplier algebra $M(A)$ of $A$. However, these assumptions exclude a number of important applications. For instance, a restriction of an extendible endomorphism to an invariant ideal in general is  not  extendible. Thus,   we are forced  to develop a large part of theory of crossed products by \emph{not necessarily extendible endomorphisms}. We do it  in section \ref{endomorphism section}. The established  results are interesting in their own right.

More specifically, we generalize one of the main results of \cite{kwa-rever} and  describe  the gauge-invariant ideals $\I$ in $C^*(A,\alpha;J)$ by  $J$-pairs $(I, I')$ 
of ideals in $A$. Additionally, we show that $\I$ is Morita-Rieffel equivalent  either to the crossed product associated to the restricted endomorphism $\alpha|_{I}$ or to an endomorphism constructed from $(I,I')$ and $\alpha$ (Theorem \ref{gauge-invariant ideals thm} and Proposition \ref{general gauge-invariant ideals description}).  We generalize the classic Pimsner-Voiculescu sequence so that it applies to the  crossed product $C^*(A,\alpha;J)$ (Proposition \ref{Voicu-Pimsner for interacts}). As a consequence we get six-term exact  sequences for  $K$-groups of all gauge-invariant ideals in $C^*(A,\alpha;J)$ (Theorem \ref{K-theory theorem}). 
We extend the terminology of \cite{kwa-rever} and say that $(A,\alpha)$ is a \emph{reversible $C^*$-dynamical systems} if  $\alpha$ has a complemented kernel and
 a hereditary range. For an arbitrary $C^*$-dynamical system $(A,\alpha)$ we  generalize the construction of a \emph{reversible $J$-extension} $(B,\beta)$ introduced in  \cite{kwa-rever}, see also \cite{kwa-ext}. We show that 
$$
C^*(A,\alpha; J)\cong C^*(B,\beta)
$$
 (Theorem \ref{rozszerzenie a repr. kowariantne2}). This is a powerful tool  because for reversible systems  $(A,\alpha)$  the crossed product $C^*(A,\alpha)$ has an accessible structure, very similar to that of classical crossed product by an automorphism. In particular, for such systems  we have natural criteria for uniqueness property, gauge-invariance of all ideals,   and  pure infiniteness of $C^*(A,\alpha)$ (see Propositions \ref{interactions?} and \ref{pure infiniteness for reversible systems}). 
	
	We note that, in contrast to \cite{kwa-rever} where more direct methods where used,  in the present paper   we  base our more general analysis on certain  results for relative Cuntz-Pimsner algebras and an identification of $C^*(A,\alpha;J)$ as such an algebra. We present the relevant  facts in Appendix \ref{appendix section}.

 \bigskip

The content is organized as follows: We recall the relevant notions and facts concerning $C_0(X)$-algebras in section \ref{bundles preliminary section}.
 General crossed products are studied in section \ref{endomorphism section}. In section \ref{morphisms section} we introduce and analyze $C_0(X)$-dynamical systems. Section \ref{reversible section} contains  general main results for $C_0(X)$-dynamical systems. We apply them to $C_0(X)$-dynamical systems with 
$X=\Prim(A)$ in section \ref{applications section}, where our results attain a particularly nice form. We finish the paper with Appendix \ref{appendix section}, which  contains relevant facts from the theory of $C^*$-correspondences  and relative Cuntz-Pimsner algebras, as well as   a discussion of a particular case of the $C^*$-correspondence $E_\alpha$ associated to $(A,\alpha)$. 

\subsection{Notation and conventions} The set of natural numbers $\N$ starts from zero.
All ideals in  $C^*$-algebras  are assumed to be closed and two-sided. All homomorphisms between $C^*$-algebras are by definition $*$-preserving.  For  actions $\gamma\colon A\times B\to C$ 
such as  multiplications, inner products, etc., we use the notation:
$$
\gamma(A,B)=\clsp\{\gamma(a,b) :a\in A, b\in B\}.
$$ 
If $A$ is a $C^*$-algebra $1$   denotes  the unit in the multiplier $C^*$-algebra $M(A)$. The enveloping von Neumann algebra of $A$ is denoted by $A^{**}$. 
 We recall, see \cite[Theorem 4.16]{kr}, that  a $C^*$-algebra $A$ is \emph{purely infinite} if and only if every $a\in A^+\setminus\{0\}$ is \emph{properly infinite}, e.g. $a\oplus a \precsim a\oplus 0$ in $M_2(A)$, where $\precsim$  Cuntz comparison of positive elements. We recall that for $a,b\in A^+$,   $a\precsim b$ in $A$ if and only if for every $\varepsilon >0$ there is $x\in A^+$ such that $\|a-xbx\|<\varepsilon$.
A $C^*$-algebra $A$ has \emph{the ideal property} \cite{Pas-Ror}, \cite{Pasnicu}, if  every ideal in $A$ is generated (as an ideal) by its projections.

\section{Preliminaries on $C_0(X)$-algebras and  $C^*$-bundles}\label{bundles preliminary section}
In  this section, we gather  certain facts concerning $C_0(X)$-algebras. We find it beneficial to use  two equivalent pictures of such objects: as $C^*$-algebras with a  $C_0(X)$-module structure and as   $C^*$-algebras of sections of $C^*$-bundles. Thus  we  implement  both of these viewpoints. As a general reference we use \cite[Section C]{Williams}, but cf. also, for instance,   \cite{HRW}, \cite{Blan-Kirch}.
\subsection{$C^*$-bundles and section $C^*$-algebras}
Let $X$ be a locally compact Hausdorff space. An  \emph{upper semicontinuous $C^*$-bundle  over} $X$ is a topological space $\A=\bigsqcup\limits_{x\in X} A(x)$ such that the natural surjection $p:\A \to X$ is open continuous, each fiber $A(x)$ is a $C^*$-algebra, the mapping $\A\ni a\to \|a\| \in \R$ is upper semicontinuous, and the $^*$-algebraic operations in each of the fibers are continuous in $\A$, for details see  \cite[Definition C.16]{Williams}. If additionally, the mapping $\A\ni a\to \|a\| \in \R$ is continuous, $\A$ is called a \emph{continuous $C^*$-bundle over} $X$. For each $x\in X$, we denote by $0_x$  the zero element in the fiber $C^*$-algebra $A(x)$, and by $1_x$  the unit in the multiplier algebra $M(A(x))$ of $A(x)$. A $C^*$-bundle $\A$ is \emph{trivial} if  there is a $C^*$-algebra $D$ and homeomorphism from $\A=\bigsqcup\limits_{x\in X} A(x)$ onto $X\times D$ which intertwines $p$ with the  projection onto the first coordinate. 

We denote by $\Gamma(\A):=\{a\in C(X,\A): p(a(x))=x\}$ the set of  continuous sections of the  upper semicontinuous $C^*$-bundle $\A$. It is a $*$-algebra with respect to natural pointwise operations. Moreover, the set of  continuous sections that vanish at infinity
$$
\Gamma_0(\A):=\{a\in \Gamma(\A): \forall_{\varepsilon > 0} \quad \{x\in X: \|a(x)\| \geq \varepsilon \}\textrm{ is compact}\}
$$
is a $C^*$-algebra with the norm $
\|a\|:=\sup_{x\in X} \|a(x)\|
$. We   call  $\Gamma_0(\A)$ the \emph{section $C^*$-algebra} of $\A$. The section algebra $\Gamma_0(\A)$   determines the topology of the $C^*$-bundle $\A$. In particular, we have   the following lemma  (see, for instance, the proof of \cite[Theorem C.25]{Williams}). 
\begin{lem}\label{topology on bundles lemma}
A net $\{b_i\}$ converges to $b$ in the $C^*$-bundle $\A$ if and only if $p(b_i)\to p(b)$ and for each $\varepsilon>0$ there is $a\in \Gamma_0(\A)$ 
such that $\|a(p(b))-b\| <\varepsilon$ and we eventually have $\|a(p(b_i))-b_i\|< \varepsilon$. 
\end{lem}   

The  algebra $\Gamma_0(\A)$ is naturally    equipped with  the structure of $C_0(X)$-algebra given by
$(f\cdot a)(x):= f(x)a(x)$   for $f\in C_0(X)$ and  $a\in \Gamma_0(\A)$.

\subsection{$C_0(X)$-algebras} A \emph{$C_0(X)$-algebra} is a $C^*$-algebra $A$  endowed with a nondegenerate homomorphism $\mu_A$ from $C_0(X)$ into the center $Z(M(A))$ of the multiplier algebra $M(A)$ of $A$. When $X$ is compact  $A$ is  also called a \emph{$C(X)$-algebra}. The $C_0(X)$-algebra $A$ is  viewed as a $C_0(X)$-module where
$$
f\cdot a:=\mu_A(f)a, \qquad f\in C_0(X),\,\, a\in A.
$$
Accordingly,  the \emph{structure map} $\mu_A: C_0(X)\to Z(M(A))$ is often suppressed. Using the Dauns-Hofmann isomorphism we may identify  $Z(M(A))$ with $C_b(\Prim A)$, and then $\mu_A$ becomes the  operator of composition with  a continuous map $\sigma_A:\Prim A\to X$. This map, called the \emph{base map}, 
 is determined by the equivalence:
\begin{equation}\label{base map via primitive ideals}
C_0(X\setminus\{x\})\cdot A \subseteq P \,\,\, \Longleftrightarrow \,\,\, \sigma_A(P)=x, \qquad P\in \Prim(A).
\end{equation}
Let us fix a $C_0(X)$-algebra $A$ and consider a bundle  $\A:=\bigsqcup\limits_{x\in X} A(x)$ where  
$$
A(x):=A/ \Big(C_0(X\setminus\{x\})\cdot A\Big),\qquad x\in X.
$$ 
It can be shown that there is a unique topology on  $\A:=\bigsqcup\limits_{x\in X} A(x)$ such that $\A$ becomes an upper semicontinuous $C^*$-bundle and the $C_0(X)$-algebra $A$ can be identified with $\Gamma_0(\A)$ by writing 
$a(x)$ for the image of $a\in A$ in the quotient algebra $A(x)$. Moreover,  $\A$ is a continuous $C^*$-bundle if and only if  $\sigma_A:\Prim A\to X$ is an open map. In the latter case,  $A$ is called a \emph{continuous $C_0(X)$-algebra}. 
In other words, we have the following statement, see  \cite[Theorem C.26]{Williams}.
\begin{thm}
A $C^*$-algebra $A$ is  a $C_0(X)$-algebra if and only if $
A\cong \Gamma_0(\A)$ where $\A$  is an upper semicontinuous $C^*$-bundle. 
Moreover, $A$ is a continuous $C_0(X)$-algebra if and only if $\A$ is a continuous $C^*$-bundle.
\end{thm}
\begin{convention}
In the sequel we will  freely  pass (often without a warning) between the above equivalent descriptions. Thus  for any $C_0(X)$-algebra $A$ we will write  $A=\Gamma_0(\A)$ where $\A$ is the associated $C^*$-bundle. 
\end{convention}
\begin{rem}\label{surjectivity of structure map} Let $A$ be a $C_0(X)$-algebra. In view of \eqref{base map via primitive ideals}, we have  $\bigcap_{P\in \sigma_A^{-1}(x)}P=C_0(X\setminus\{x\})\cdot A$ whenever $x\in \sigma_A(\Prim A)$, and  $A(x)=\{0\}$ if and only if $x\notin \sigma_A(\Prim A)$.   Thus if $\sigma_A(\Prim(A))$ is locally compact (which is always the case  when $A$ is unital, or when $A$ is a continuous $C_0(X)$-algebra), then we may treat $A$ as a $C_0\big(\sigma_A(\Prim(A))\big)$-algebra; in other words, we may  assume that $\sigma_A$ is surjective, or equivalently that all fibers $A(x)$ are non-trivial.  
\end{rem}

\subsection{Multiplier algebra of a $C_0(X)$-algebra}
We say that a $C_0(X)$-algebra $A$  \emph{has local units} if all fibers $A(x)$, $x\in X$, are unital, and for any $x\in X$ there is $a\in A$ such that $a(y)=1_y$ is the unit in $A(y)$ for all $y$ in a neighborhood of $x$.  
\begin{lem}
A $C_0(X)$-algebra $A$ is unital  if and only if $A$ has local units and the range of $\sigma_A$ is compact.
\end{lem}
\begin{proof}
If $1$ is the unit in $A$ then $\sigma_A(\Prim(A))=\{x\in X: \|1(x)\| \geq 1/2\}$ is compact, because $1\in \Gamma_0(\A)$, and clearly the (global) unit $1$ is a  local unit for any point in $X$. Conversely, suppose that $\sigma_A(\Prim(A))$ is compact and $A$ has  local units. Consider  the function $1:X \to \A=\bigsqcup\limits_{x\in X} A(x)$ where for each $x\in A$ we let  $1(x):=1_x$ to be the unit in $A(x)$. Using Lemma \ref{topology on bundles lemma} and local units one readily sees that  $1$ is a continuous section of $\A$. For any $\varepsilon $ the set $\{x\in X: \|1(x)\| \geq \varepsilon \}$ is compact, as a closed subset of $\sigma_A(\Prim(A))$. Thus $1\in \Gamma_0(\A)=A$.
\end{proof}

We have the following natural description of the multiplier algebra $M(A)$ of a $C_0(X)$-algebra $A$ as sections of the set  $M(\A):=\bigsqcup\limits_{x\in X} M(A(x))$,  see \cite[Lemma C.11]{Williams}. We emphasize however, that  in general (even when $X$ is compact) $M(\A)$ can not be equipped with a topology making it an upper semicontinuous $C^*$-bundle such that $M(A)\subseteq \Gamma(M(\A))$, see  \cite[Example C.13]{Williams}.
\begin{prop}\label{multiplier description}
Suppose that $A$ is a $C_0(X)$-algebra. The multiplier algebra $M(A)$ can be naturally identified with  the set of all functions $m$ on $X$ such that $m(x)\in M(A(x))$, for all $x\in X$, and  the functions $x\mapsto m(x)a(x)$,  $x\mapsto a(x)m(x)$ are in $A=\Gamma_0(\A)$ for any $a\in A$. Then the $C^*$-algebraic structure of $M(A)$ is given by the pointwise operations and the supremum norm $\|m\|=\sup_{x\in X}\|m(x)\|$.
\end{prop}

\subsection{Ideals and quotients of a $C_0(X)$-algebra}
Fix a $C_0(X)$-algebra $A$ and let  $I$ be an ideal in  $A$. Assuming the standard identifications $\Prim I=\{P\in \Prim A:I\nsubseteq P\}$ and $\Prim(A/I)=\{P\in \Prim A:I\subseteq P\}$, we see that  both $I$ and  $A/I$ are $C_0(X)$-algebras with base maps $\sigma_A:\Prim(I)\to X$ and $\sigma_A:\Prim(A/I)\to X$ respectively.  Moreover, we have natural isomorphisms   $(A/I)(x)\cong A(x)/ I(x)$ where $I(x)=\{a(x): a\in I\subseteq A\}$, $x\in X$.

Suppose  that   $A$ is a continuous $C_0(X)$-algebra. Then the ideal $I$ is naturally a continuous $C_0(Y)$-algebra for any locally compact set $Y$  containing  the open set $\sigma_A(\Prim(I))$, because a restriction of an open map to an open set is  open (independently of the codomain). The situation is quite different when dealing with a restriction to a closed set, and thus the case of the quotient $A/I$ is  more delicate.  
Nevertheless,  the set 
$
Y=\sigma_A(\Prim(A/I))
$
is  locally compact,   and   the mapping  $\sigma_A:\Prim(A/I)\to Y$ is open for instance when  $I$ is complemented or $\sigma_A$ is injective. 
Translating this to  the language of $C^*$-bundles  we get the following lemma.
\begin{lem}\label{lemma on ideals in C_0(X)-algebras}
Suppose that $I$ is an ideal in  a $C^*$-algebra $A=\Gamma_0(\A)$ of continuous sections of an upper semicontinuous $C^*$-bundle $\A=\bigsqcup\limits_{x\in X}A(x)$. The ideal $I$   and the quotient algebra $A/I$ can be naturally treated as algebras of continuous sections of $\II=\bigsqcup\limits_{x\in X} I(x)$  and $\A/\II=\bigsqcup\limits_{x\in X}A(x)/I(x)$ (equipped with unique topologies), respectively. Moreover, we have
\begin{equation}\label{J-ideal fibers}
\{x\in X:  I(x)\neq \{0\}\}=\sigma_A(\Prim(I)),
\end{equation}
\begin{equation}\label{J-quotient fibers}
\{x\in X:  I(x)\neq A(x)\}=\sigma_A(\Prim(A/I)).
\end{equation}
If $\A$ is a continuous bundle, then $\II$ is continuous over the set \eqref{J-ideal fibers} and $\A/\II$ is continuous over the set \eqref{J-quotient fibers} whenever $I$ is complemented or $\sigma_A$ is injective.

\end{lem}
\begin{proof} In view of the above discussion we only need to show   \eqref{J-ideal fibers} and \eqref{J-quotient fibers}. The equivalences
$$
I(x)\neq \{0\}\, \Longleftrightarrow\, I  \nsubseteq \bigcap_{P\in \sigma_A^{-1}(x)}P\, \Longleftrightarrow\, \exists_{P\in \sigma_A^{-1}(x)} I  \nsubseteq P \, \Longleftrightarrow\, x \in \sigma_A(\Prim(I))
$$
prove \eqref{J-ideal fibers}. To see \eqref{J-quotient fibers} notice  that 
using \eqref{base map via primitive ideals} we get
\begin{align*}
I(x) \neq A(x)  \, &\Longleftrightarrow\,   \Big(I + \bigcap_{P\in \sigma_A^{-1}(x)}P\Big) \neq A  \, \Longleftrightarrow\,  \exists_{P_0\in \Prim(A)} \Big(I + \bigcap_{P\in \sigma_A^{-1}(x)}P\Big)  \subseteq P_0
\\
& \Longleftrightarrow\,  \exists_{P_0\in \Prim(A)}\,\,  I \subseteq P_0\, \textrm{ and }\, \bigcap_{P\in \sigma_A^{-1}(x)}P  \subseteq P_0 \,
\\
& \Longleftrightarrow\, x \in \sigma_A(\Prim(A/I)).
\end{align*}
\end{proof}

Let $I$ be an ideal  in the  $C_0(X)$-algebra $A$. The annihilator
$
I^\bot=\{a\in A: aI=0\}
$
of  $I$ is also a $C_0(X)$-algebra with the base map $\sigma_A:\Prim(I^\bot)\to X$. Moreover, since $I^\bot$ is the biggest ideal in $A$ with the property  that $I\cap I^\bot=\{0\}$  it follows that $\Prim(I^\bot)=\Int(\Prim(A/I))$. If $A$ is a continuous $C_0(X)$-algebra then $I^\bot$ is  a continuous $C_0(U)$-algebra where $U=\sigma_A(\Int(\Prim(A/I)))$. In terms of $C^*$-bundles, $I^\bot$ can be viewed as the algebra of continuous sections of the $C^*$-bundle $
\II^\bot:=\bigsqcup\limits_{x\in X} I(x)^\bot$  where $I(x)^\bot$ is contained in the annihilator of $I(x)$  in $A(x)$.
In particular, $I(x)^\bot=\{0\}$   if and only if $x \notin U$.

\section{General crossed products by endomorphisms}\label{endomorphism section}
 In this section, we define crossed products  $ C^*(A,\alpha;J)$ for  an arbitrary (not necessarily extendible) endomorphism $\alpha:A\to A$. We establish basic results concerning the structure of these $C^*$-algebras, including description of all gauge-invariant ideals and `Pimsner Voiculescu sequences' determining their $K$-theory. We also construct a reversible $J$-extension $(B,\beta)$ of $(A,\alpha)$, discuss the notion of topological freeness for systems which are reversible or commutative, and give a general pure infiniteness criteria for reversible systems.
\subsection{$C^*$-dynamical systems and their crossed products}
A \emph{$C^*$-dynamical system} is a pair $(A,\alpha)$ where $A$ is a $C^*$-algebra and $\alpha:A\to A$ is an endomorphism.  We say that $\alpha$, or that the system $(A,\alpha)$, is \emph{extendible} \cite{adji} if $\alpha$ extends to a strictly continuous endomorphism $\overline{\alpha}:M(A)\to M(A)$. It is known to  hold exactly when for some (and  hence any) approximate unit $\{\mu_\lambda\}$  in $A$ the net $\{\alpha(\mu_\lambda)\}$ converges strictly in $M(A)$. 
In contrast to \cite{kwa-rever}, in the present paper in  general we do not  assume that $(A,\alpha)$ is extendible.
  
\begin{defn}[Definition 2.4 in \cite{kwa-rever}] A  $C^*$-dynamical system  $(A,\al)$ is called \emph{reversible} if $\ker\alpha$ is a complemented ideal in $A$ and $\alpha(A)$ is a hereditary subalgebra of $A$  (briefly,  $\alpha$ has a complemented kernel and a hereditary range).
\end{defn}
\begin{rem}
An  extendible endomorphism $\alpha:A\to A$   has  a hereditary range if and only if it is  a \emph{corner endomorphism}, that is if  $\alpha(A)$ is a corner in $A$ (we then necessarily have $\overline{\alpha}(1)A\overline{\alpha}(1)=\alpha(A)$). In particular, an extendible $C^*$-dynamical system $(A,\alpha)$ is  reversible  if and only if $\alpha$ is a corner endomorphism with complemented kernel.  
\end{rem}
Suppose that $(A,\al)$ is a reversible $C^*$-dynamical system. Then $\alpha:(\ker\alpha)^\bot\mapsto \alpha(A)A\alpha(A)$ is an isomorphism and we denote  its inverse by 
$\alpha^{-1}$. If  $(A,\al)$ is extendible, then $\alpha(A)A\alpha(A)=\overline{\alpha}(1)A\overline{\alpha}(1)$ and  $\alpha^{-1}$ extends to a completely positive  map  $\alpha_*:A\to A$ given by   the formula 
\begin{equation}\label{complete transfer operator form}
\alpha_*(a)=\alpha^{-1}(\overline{\alpha}(1)a\overline{\alpha}(1)), \qquad a\in A. 
\end{equation}
The map $\alpha_*$ is a \emph{transfer operator} for $(A,\alpha)$ in the sense of Exel \cite{exel2}, that is we have
$
\alpha_*(\alpha(a)b) =a\alpha_*(b)$,  for all $a,b\in A$. Moreover, $\alpha_*$ is \emph{regular}, which means that $\alpha \circ \alpha_*$ is a conditional expectation onto $\alpha(A)$.  In fact, $\alpha_*$ is a unique \emph{regular transfer operator} for $(A,\alpha)$, see \cite[Proposition 4.15]{kwa-exel}. Transfer operators satisfying \eqref{complete transfer operator form}  appear  in a natural way in a number of papers, see for instance  \cite{exel2},  \cite{Ant-Bakht-Leb},  \cite{kwa-interact}, \cite{kwa-rever}.
\begin{ex}\label{commutative example of reversible systems}
If $A=C_0(X)$ where $X$ is a locally compact  Hausdorff space, then every endomorphism $\alpha:A\to A$ is  of the form
$$
\al(a)(x)=\begin{cases}
a(\p(x)),  & x \in \Delta,
\\
0,        & x \notin \Delta,
\end{cases}
$$ 
where $\p:\Delta\to X$ is a continuous proper mapping defined on an open subset $\Delta \subseteq X$. Note that, properness of $\p$ implies that $\p(\Delta)$ is closed in $X$. We call the pair $(X,\p)$  a \emph{partial dynamical system} dual to $(A,\alpha)$. The endomorphism $\alpha$ is extendible if and only if $\Delta$ is closed. The kernel of $\alpha$ is a complemented ideal in $A$ if and only if $\p(\Delta)$ is open. The pair $(A,\alpha)$ is a reversible $C^*$-dynamical system if and only if both $\Delta$ and $\p(\Delta)$ are open in $X$ and $\p:\Delta \to \p(\Delta)$ is a homeomorphism. If the latter conditions  are satisfied we say that  $(X,\p)$ is a \emph{reversible partial dynamical system}. 
\end{ex}

Now, we turn to the definition of crossed products. For more details, in the case $A$ is unital or $\alpha$ is extendible, see \cite{kwa-leb} and \cite{kwa-rever}.
\begin{defn}A \emph{representation} $(\pi,U)$ of  a $C^*$-dynamical system $(A,\alpha)$ on  a Hilbert space $H$ consists of  a non-degenerate representation $\pi:A\to \B(H)$ and an operator $U\in \B(H)$  such that 
\begin{equation}\label{covariance rel1*}
U\pi(a)U^* =\pi(\alpha(a)),\qquad \textrm{ for all }a \in A.
\end{equation}
We will occasionally deal with representations of  $(A,\alpha)$ in a $C^*$-algebra $B$ by which mean a pair $(\pi,U)$ where $\pi:A\to B$ is a non-degenerate homomorphism and  $U\in B^{**}$ (an element of the enveloping von Neumann algebra of $B$) satisfies  \eqref{covariance rel1*}. If $\pi$ is injective then we say $(\pi,U)$ is \emph{injective}.
\end{defn}
Let $(\pi,U)$ be a representation of $(A,\alpha)$ in a $C^*$-algebra $B$. Then $U$ is necessarily a partial isometry. Indeed, if $\{\mu_\lambda\}$ is an approximate unit in $A$ then by non-degeneracy of $\pi$, $\{\pi(\mu_\lambda)\}$ converges $\sigma$-weakly to the unit in $B^{**}$ and therefore $\{\pi(\alpha(\mu_\lambda))\}=\{U\pi(\mu_\lambda)U^*\}$ converges $\sigma$-weakly to $UU^*$. Hence, using multiplicativity of $\alpha$, we get that
$$
(UU^*)^2=\sigma\textrm{-}\lim_\lambda \sigma\textrm{-}\lim_{\lambda'}\pi(\alpha(\mu_\lambda))\pi(\alpha(\mu_{\lambda'}))=\sigma\textrm{-}\lim_\lambda \pi(\alpha(\mu_\lambda))=U^*U
$$
is a projection, cf. \cite[Proposition 3.21]{kwa-exel}. Moreover, see \cite[Proposition 3.21]{kwa-exel} or the proof of \cite[Lemma 1.2]{kwa-leb}, multiplicativity of $\alpha$ implies that the initial projection $U^*U$ of $U$ commutes with the elements of $\pi(A)$. In particular, 
$$
I_{(\pi,U)}:=\{ a\in A: U^*U \pi(a)=\pi(a)\}
$$
 is an ideal in $A$. If   an ideal $J$ in $A$ is contained in $I_{(\pi,U)}$ we say that the representation $(\pi,U)$ is \emph{$J$-covariant}. If $(\pi,U)$ is \emph{$(\ker\alpha)^\bot$-covariant}, that is if 
 $$
a\in (\ker\alpha)^\bot \,\,\, \Longrightarrow\,\,\, \pi(a)=U^*U\pi(a)
 $$
  we say that $(\pi,U)$  is a \emph{covariant representation}.  Note that if $\alpha$ is injective, then the representation $(\pi,U)$ is covariant if and only if $U$ is an isometry. Thus if $\alpha$ is injective and non-degenerate  then the representation $(\pi,U)$ is covariant if and only if $U$ is a unitary. 
The special role of $(\ker\alpha)^\bot$ is also indicated in the following fact.
	\begin{lem}\label{lemma justifying annihilator}
	Suppose that  $(\pi,U)$ is an injective representation of $(A,\alpha)$. Then $I_{(\pi,U)}\subseteq (\ker\alpha)^\bot$.  In particular, 
	$
	I_{(\pi,U)}=(\ker\alpha)^\bot 
	$ if and only if  $(\pi,U)$ is a covariant representation.
	\end{lem}
	\begin{proof}
	Let $a \in I_{(\pi,U)}$ and $b\in \ker\alpha$. Then $ab \in I_{(\pi,U)}$ and 
	$$
	\pi(ab)=U^*U\pi(ab)U^*U=U^*\pi(\alpha(ab))U=U^*\pi(\alpha(a))\pi(\alpha(b))U=0.
	$$
	Hence $ab=0$ because $\pi$ is injective. Accordingly, $I_{(\pi,U)}\subseteq (\ker\alpha)^\bot$.
	\end{proof}
Combining the above lemma and the following proposition one can see that $(A,\alpha)$ admits an injective $J$-covariant representation if and only if $J\subseteq (\ker\alpha)^\bot$.
	\begin{prop}\label{existence of crossed products}
For any $C^*$-dynamical system $(A,\alpha)$ and  any ideal  $J$ in $(\ker\alpha)^\bot$ there exists a $C^*$-algebra $C^*(A,\alpha;J)$  containing $A$ as a non-degenerate $C^*$-algebra and an operator $u\in C^*(A,\alpha;J)^{**}$ such that 
\begin{itemize}
\item[a)]  $ C^*(A,\alpha;J)$ is generated (as a $C^*$-algebra) by $A\cup uA$,  
$$
\alpha(a)=u a u^*  \textrm{ for each } a\in A\quad\textrm{ and  }\quad J=\{ a\in A: u^*u  a= a\},
$$
\item[b)] for every $J$-covariant representation $(\pi,U)$ of $(A,\alpha)$ there is a representation $\pi\rtimes U$ of  $C^*(A,\alpha;J)$ determined by relations $(\pi\rtimes U)(a)= \pi(a)$, $a\in A$, and $(\pi\rtimes U)(u)=U$.
\end{itemize}
Moreover, if $\alpha$ is extendible, then $u\in M(C^*(A,\alpha;J))$ and 
$ C^*(A,\alpha;J)=C^*(A\cup Au)$. 
\end{prop}
\begin{proof}
Existence of $C^*(A,\alpha;J)$ with the prescribed properties can be deduced from Propositions \ref{identification of crossed products} and \ref{existence of relative Cuntz-Pismener algebras}. It also follows from \cite[Proposition 3.26]{kwa-exel} which in essence states that $C^*(A,\alpha;J)$ is a special case of the crossed product  defined in \cite[Definition 3.5]{kwa-exel} (note that the identification of the aforementioned algebras goes thorough the equality $s=u^*$). In particular,  \cite[Remark 3.11]{kwa-exel} implies that  $u\in C^*(A,\alpha;J)^{**}$  and when $\alpha$ is extendible then $u\in M(C^*(A,\alpha;J))$. If $\alpha$ is extendible then $C^*(A,\alpha;J)=C^*(A\cup Au)$ by \cite[Lemma 3.23]{kwa-exel}.
\end{proof}
Universal properties of the $C^*$-algebra $ C^*(A,\alpha;J)$ imply that, up to a natural isomorphism, it  is uniquely determined  by the triple $(A,\alpha, J)$.
\begin{defn}\label{crossed product defn}
We define the \emph{relative crossed product} associated to a $C^*$-dynamical system $(A,\alpha)$ and  an ideal  $J$ in $(\ker\alpha)^\bot$ to be the $C^*$-algebra described in Proposition \ref{existence of crossed products}. 
We also write $
C^*( A,\alpha):=C^*( A,\alpha,(\ker\alpha)^\bot)
$ and call it  the (unrelative) \emph{crossed product of $A$ by $\alpha$}.
\end{defn}
	\begin{rem}\label{remark to mentioned in the introduction} In the case $A$ is unital or $\alpha$ is extendible, the $C^*$-algebra $ C^*(A,\alpha;J)$ was studied  respectively in \cite{kwa-leb} and \cite{kwa-rever}. In general, the crossed product $ C^*(A,\alpha;J)$ is a special case of  the one defined in \cite[Definition 3.5]{kwa-exel} where  $\alpha$ is treated as a completely positive map, or the one introduced in  \cite[Definition 4.9]{kwa-doplicher} where  $\alpha$ is treated as a partial morphism of $A$.  In particular, one could  consider  the crossed product $C^*( A,\alpha;J)$ for an arbitrary  ideal $J$ in $A$, not necessarily contained in $(\ker\alpha)^\bot$, cf. \cite{kwa-exel}, or  \cite{kwa-doplicher}. 
	However, if $J\nsubseteq (\ker\alpha)^\bot$ the algebra $A$ does not embed into $C^*( A,\alpha;J)$. Moreover, as described in \cite[Section 5.3]{kwa-leb}, see also  \cite[Example 6.24]{kwa-doplicher}, or  \cite[Remark 4.4]{kwa-rever}, by passing to a quotient $C^*$-dynamical system,  one can always reduce this seemingly more general situation  to that  of Definition \ref{crossed product defn}. 
\end{rem}

By  universal property of the crossed product $C^*( A,\alpha;J)$, there is  a  circle action $\mathbb{T}=\{z\in \C: |z|=1\}\ni z \longmapsto \gamma_z \in \Aut(C^*( A,\alpha;J))$ determined by relations $\gamma_z(a)=a, \gamma_z(u)=z u$,  $a\in A$, $z\in \mathbb{T}$.
We  call $\gamma=\{\gamma_z\}_{z\in \mathbb{T}}$  the \emph{gauge action} on $C^*( A,\alpha;J)$. We say that a representation $(\pi,U)$ of $(A,\alpha)$ \emph{admits a gauge action} if the relations $\gamma_z(\pi(a))=\pi(a)$, $\gamma_z(U)=z U$,  $a\in A$, $z\in \mathbb{T}$, determine a circle action on the $C^*$-algebra generated by $\pi(A)\cup U\pi(A)$.
We have the following version of the gauge-uniqueness theorem.
\begin{prop}\label{gauge-uniqueness theorem} For any injective $J$-covariant representation $(\pi,U)$ the homomorphism $\pi\rtimes U$ of  $C^*(A,\alpha;J)$ is injective if and only if $I_{(\pi,U)}=J$ and  $(\pi,U)$  admits a gauge action.

In particular, if  $(\pi,U)$ is a covariant representation of $(A,\alpha)$ then the homomorphism $\pi\rtimes U$ of  $C^*(A,\alpha)$ is injective if and only if  $\pi$ is injective and  $(\pi,U)$  and  admits a gauge action.
\end{prop}
\begin{proof}
The first part of the  assertion follows from  Propositions \ref{gauge-invariance theorem} and \ref{identification of crossed products}. For the second part apply Lemma \ref{lemma justifying annihilator}.
\end{proof}
We list certain general permanence properties for the crossed products $C^*(A,\alpha;J)$.
\begin{prop}\label{permanence properties prop}
Let  $(A,\alpha)$ be a $C^*$-dynamical system  and  let ideal  $J$ be an ideal in $(\ker\alpha)^\bot$.
\begin{itemize}
\item[(i)] $A$ is exact  $\Longleftrightarrow$ $C^*(A,\alpha;J)$ is exact.
\item[(ii)] $A$  is nuclear  $\Longrightarrow$ $C^*(A,\alpha;J)$ is nuclear.
\item[(iii)] If $A$ is separable, nuclear, and both $A$   and   $J $  satisfy the UCT,  then  $C^*(A,\alpha;J)$ satisfies the UCT.
\end{itemize}
\end{prop}
\begin{proof} By Proposition \ref{identification of crossed products},  we have $C^*(A,\alpha;J)\cong \OO(J,E_\alpha)$. Since  $\OO(J,E_\alpha)$ is the quotient of the Toeplitz algebra $\TT_{E_\alpha}=C(\{0\},E_\alpha)$, item (i) follows from \cite[Theorem 7.1]{katsura}.  Similarly,  \cite[Theorem 7.2]{katsura} implies (ii). The argument leading to  \cite[Proposition 8.8]{katsura} gives (iii).
\end{proof}
\subsection{Algebraic structure of crossed products}
 The $*$-algebraic structure underlying the crossed product $C^*(A,\alpha;J)$, that could  actually be used to construct $C^*(A,\alpha;J)$, cf. \cite[Example 2.20 and Definition 4.9]{kwa-doplicher}, is described in the following proposition.

\begin{prop}\label{on the structure of crossed products}
For any   $C^*$-dynamical system $( A,\alpha)$ and any  ideal $J$ in $(\ker\alpha)^\bot$ the universal operator $u\in C^*(A,\alpha;J)^{**}$ is a power partial isometry, that is  $\{u^{n*}u^n\}_{n\in \N}$ and $\{u^{n}u^{n*}\}_{n\in \N}$ are decreasing sequences of mutually commuting projections. Moreover, 
\begin{equation}\label{commutation relation}
u^na=\alpha^n(a)u^n, \qquad \textrm{for all }a\in A, n \in \N,
\end{equation}
so the projections $u^{*n}u^n$ commute with elements of $A$. The elements 
\begin{equation}\label{general form of a in cross}
a=\sum_{n,m=1}^N u^{*n}a_{n,m} u^m, \qquad a_{n,m} \in \alpha^n(A)A\alpha^m(A),\,\, n,m=1,...,N,\,\, N\in \N,
\end{equation}
form a dense $*$-subalgebra of $C^*( A,\alpha;J)$ and their products are determined by the formula 
\begin{equation}\label{general form of multiplication in cross}
\left(u^{*n}a_{n,m} u^m\right)   \left(u^{*m+k}a_{m+k,l} u^l\right) = u^{*n+k}\alpha^k(a_{n,m}) a_{m+k,l} u^l,\qquad n,m,k,l\in \N.
\end{equation}
 \end{prop}
\begin{proof}
The first part of the assertion follows from  \cite[Proposition 3.21]{kwa-exel} and the fact that if $(\pi,U)$ is a representation of $(A,\alpha)$, then $(\pi,U^n)$ is a representation of $(A,\alpha^n)$, $n\in \N$. Let us  show \eqref{general form of multiplication in cross}.
Since  $a_{n,m} \in A\alpha^m(A)$ we have $a_{n,m} u^{m}u^{*m}=a_{n,m}$ and by  \eqref{commutation relation} we get $a_{n,m}u^{*k}=u^{*k}\alpha^k(a_{n,m})$. Thus
$$
\left(u^{*n}a_{n,m} u^m\, \right)   \left(u^{*m+k}a_{m+k,l}\, u^l\right) = u^{*n}a_{n,m}u^{*k}a_{m+k,l} u^l =u^{*n+k}\alpha^k(a_{n,m}) a_{m+k,l} u^l.
$$
Now, using \eqref{general form of multiplication in cross}, one readily sees that elements \eqref{general form of a in cross} form a 
$*$-algebra generated by $A$ and  $uA=\alpha(A)u$. 
\end{proof}
\begin{cor}\label{initial projection is a mutliplier} The initial projection $u^*u$ of the universal partial isometry $u\in C^*(A,\alpha;J)^{**}$ belongs to the multiplier algebra $M(C^*(A,\alpha;J))$ of $C^*(A,\alpha;J)$. 
\end{cor}
\begin{proof} Let  $a_{n,m} \in \alpha^n(A)A\alpha^m(A)$, $n,m\in \N$. If $n>0$, then $(u^*u) u^{*n}a_{n,m} u^m=u^{*n}a_{n,m} u^m\in C^*(A,\alpha;J)$. If $n=0$ then 
 $$
(u^*u) u^{*n}a_{n,m} u^m= (u^*u)a_{0,m} u^m=u^*\alpha(a_{0,m}) u^{m+1}\in C^*(A,\alpha;J).
$$
By Proposition \ref{on the structure of crossed products} we get $(u^*u) C^*(A,\alpha;J)\subseteq C^*(A,\alpha;J)$ and consequently (since $u^*u$ is self-adjoint) $u^*u\in M(C^*(A,\alpha;J))$.
\end{proof}

If the kernel of $\alpha$ is complemented then the  initial projection $u^*u$ of the universal partial isometry $u\in C^*(A,\alpha)^{**}$ can also be treated as a multiplier of $A$: 
\begin{lem}\label{lemma on systems with complemented kernel} Suppose that  $(A,\alpha)$ is a $C^*$-dynamical system such that the kernel of $\alpha$ is a complemented ideal in $A$. Let $u\in C^*(A,\alpha)^{**}$ be  the universal partial isometry. Then 
$$
u^*uA=(\ker\alpha)^\bot\quad\textrm{and}\quad  u^{*}\alpha(a)u= u^*u a, \qquad a\in A. 
$$
If  $(A,\alpha)$ is a reversible $C^*$-dynamical system then for every $n\in \N$ the system $(A,\alpha^n)$ is also  reversible and
\begin{equation}\label{equation to be added because I am stupid}
u^{n*}u^nA=(\ker\alpha^n)^\bot\quad\textrm{and}\quad u^{n*}\alpha^n(a)u^n= u^{n*}u^na, \qquad a\in A. 
\end{equation}
\end{lem}
\begin{proof}
We have  $(\ker\alpha)^\bot=\{ a\in A: u^*u  a= a\}\subseteq u^*uA$ and  $u(\cdot)u^*$ maps both of the $C^*$-algebras $(\ker\alpha)^\bot$ and $u^*uA$ isomorphically  onto $\alpha(A)=uAu^*$.  This implies that $(\ker\alpha)^\bot =u^*uA$. In particular, $u^*ua=u^*uau^*u=u^{*}\alpha(a)u$ for every $a\in A$. 

Now assume that   $(A,\alpha)$ is   reversible.  Note that a composition of two homomorphisms  $f:A\to B$ and $g:B\to C$ with hereditary ranges have a hereditary range. The latter holds because
\begin{align*}
g(f(A))C g(f(A))&=g(f(A))g(B)C g(B)g(f(A))=g(f(A))g(B)g(f(A))
\\
&=g(f(A)Bf(A))=g(f(A)).
\end{align*}
Thus for every $n\in\N$ the range of $\alpha^n$ is a hereditary subalgebra of $A$. We prove that $\ker\alpha^n$ is complemented and that \eqref{equation to be added because I am stupid} holds by induction on $n$. For $n=1$ we have already seen it.  Assume  that the assertion holds for some $n\in \N$. Let $\theta$  be the inverse to the isomorphism  $\alpha^n:(\ker\alpha^n)^\bot\to \alpha^n(A)$. Then clearly $\theta(\alpha^n(A)\cap (\ker\alpha)^\bot) \subseteq (\ker\alpha^{n+1})^\bot$. However, since $\alpha^n(A)$ is hereditary in $A$ we have $\alpha^n(A)\cap (\ker\alpha)^\bot =\alpha^n(A)(\ker\alpha)^\bot\alpha^n(A)$.  Hence $\alpha^{n+1}$ maps $\theta(\alpha^n(A)\cap (\ker\alpha)^\bot)$ onto $\alpha^{n+1}(A)$. Since  $\alpha^{n+1}(\ker\alpha^{n+1})^\bot \to \alpha^{n+1}(A)$ is isometric it follows that it is actually an isomorphism and we have $\theta(\alpha^n(A)\cap (\ker\alpha)^\bot)= (\ker\alpha^{n+1})^\bot$.  For any element $\alpha^n(a)$ in $(\ker\alpha)^\bot=u^*uA$, by the induction hypothesis, we have
$$
\theta(\alpha^n(a))=u^{*n} \alpha^n(a) u^{n}= u^{*n}  (u^*u) \alpha^n(a)u^{n}=u^{*n+1}u^{n+1}au^{*n} u^{n} =u^{*n+1}u^{n+1}a .
$$
Hence  $(\ker\alpha^{n+1})^\bot\subseteq u^{*n+1}u^{n+1}A$, and the argument we used to show that  $(\ker\alpha)^\bot=u^*uA$ implies that we actually have $(\ker\alpha^{n+1})^\bot= u^{*n+1}u^{n+1}A$. Thus $u^{*n+1}u^{n+1}\in M(A)$ is the projection onto  $(\ker\alpha^{n+1})^\bot$.
\end{proof}

  Integration over the Haar measure  on $\mathbb{T}$ gives  the (faithful) conditional expectation
\begin{equation}\label{conditional expectation formula}
\mathcal{E}(a)=\int_{\mathbb{T}} \gamma_z(a) d\mu, \qquad a\in C^*( A,\alpha;J),
\end{equation}
from $C^*( A,\alpha;J)$ onto  the fixed point $C^*$-algebra $B$ for the gauge action. 
We  refer to the $C^*$-algebra $B$ as the \emph{core $C^*$-subalgebra}  of $C^*( A,\alpha;J)$. In view of Proposition \ref{on the structure of crossed products}, we have
$$
B=\clsp\{u^{*n}a u^n: a \in \alpha^n(A)A\alpha^n(A),\,\, n\in \N\}.
$$
If the $C^*$-dynamical system  $(A,\alpha)$ is reversible, then the core  of $C^*(A,\alpha)$ coincides with $A$ and $C^*(A,\alpha)$ has a similar structure to that of classical crossed product by an automorphism. 

\begin{prop}\label{crossed product for reversible system}
Suppose that $(A,\alpha)$ is a reversible $C^*$-dynamical system. T The crossed product $C^*(A,\alpha)$  is the closure of a dense $^*$-algebra consisting of the elements of the form \begin{equation}\label{general form of a guy in cross}
a=\sum_{k=1}^n u^{*k}a_{-k}^* + a_0 + \sum_{k=1}^n a_ku^{k}, \qquad a_{k} \in A\alpha^k(A),\,\, k=0,\pm 1,...,\pm n.
\end{equation}
The coefficients  $a_{k}\in A\alpha^k(A)$ in  \eqref{general form of a guy in cross} are uniquely determined by $a$.
\end{prop}
\begin{proof} 

To see that any element \eqref{general form of a in cross} can be presented in the form \eqref{general form of a guy in cross} let us consider an element 
$u^{*n}a_{n,m} u^m$ where $a_{n,m} \in \alpha^n(A)A\alpha^m(A)$ and put $k:=m-n$. Suppose that $k \geq 0$. Then $\alpha^n(A)A\alpha^{n+k}(A)=\alpha^n(A)A\alpha^n(A)\alpha^{n+k}(A)=\alpha^n(A)\alpha^{n+k}(A)=\alpha^n(\alpha^k(A))$. Thus there is $a_k\in \alpha^k(A)\cap (\ker\alpha^k)^\bot$ such that $a_{n,m}=\alpha^n(a_k)$. Hence, by Lemma \ref{lemma on systems with complemented kernel}, we get
$$
u^{*n}a_{n,m} u^m =u^{*n}\alpha^n(a_k)u^{n} u^{k}=a_k u^{k}.
$$
If $k<0$ by passing to adjoints we get $u^{*n}a_{n,m} u^m =u^{-*k}a_k$. In view of Proposition \ref{on the structure of crossed products}, this proves the  first part of  the assertion.  For the last part notice that if $a$ is of the form \eqref{general form of a guy in cross}, then  for $k\geq 0$ we have 
 $\mathcal{E}(u^ka)=a_{-k}^*$ and $\mathcal{E}(au^{*k})=a_{k}$. Hence the  coefficients  $a_{\pm k}$ are uniquely determined by $a$.
\end{proof}

\subsection{Gauge-invariant ideals}
Let  $(A,\alpha)$ be a fixed  $C^*$-dynamical system and let $J$ be an ideal in  $(\ker\al)^\bot$. Ideals in $C^*( A,\alpha;J)$ that are invariant under the gauge action are  called \emph{gauge-invariant}. In this subsection we describe these ideals in terms of pairs of ideals in $A$.
\begin{defn}[Definitions 3.2 and 3.3 in \cite{kwa-rever}]\label{invariance definition}  We  say that an ideal 
  $I$  in $A$ is a \emph{positively invariant} ideal in $(A,\alpha)$ if $\al(I)\subseteq I$. We  say that
  $I$ is \emph{$J$-negatively invariant} ideal in $(A,\alpha)$ if $J\cap \al^{-1}(I)\subseteq I$.
If $I$ is both positively invariant and $J$-negatively invariant we   say that $I$ is \emph{$J$-invariant}, and if  $J=(\ker\alpha)^\bot$ we drop the prefix `$J$-'. 
		\end{defn}

	Let $I$ be a positively invariant ideal in $(A,\alpha)$. It induces two $C^*$-dynamical systems: the \emph{restricted 
	$C^*$-dynamical system} $(I,\al|_{I})$ and the \emph{quotient $C^*$-dynamical system} $(A/I,\al_I)$   where 
   $\al_I(a+I):=\al(a)+I$  for all  $a\in A$. 
	Note that if $(A,\alpha)$ is extendible then so is the quotient $(A/I,\al_I)$, but $(I,\al|_{I})$ in general fails to be extendible. For instance, if $A=C_0(\mathbb{R})$,  $\alpha(a)(x)=a(x-1)$ and $I=C_0(0,\infty)$ then $\alpha$ is extendible but $\alpha|_I$ is not, see also  \cite{adji}. 
	
	\begin{lem}\label{complemented kernel invariance lemma}
Let $I$ be an invariant ideal  in $(A,\alpha)$.
\begin{itemize}
\item[(i)] If the kernel of $\alpha$  is a complemented ideal in $A$ then $\al_I$ and $\al|_{I}$ have complemented kernels in $A/I$ and $I$ respectively, and 
$$
(\ker\alpha_I)^\bot=q_I((\ker\alpha)^\bot), \qquad (\ker\alpha|_I)^\bot=(\ker\alpha)^\bot \cap I.
$$ 

\item[(ii)] If $(A,\alpha)$  is reversible then $(A/I,\al_I)$ and $(I,\al|_{I})$ are reversible.
\end{itemize} 
\end{lem}
\begin{proof} (i).
Since  $(\ker\alpha)^\bot\cap \al^{-1}(I)\subseteq I$ we get $\ker\alpha_I = q_I(\alpha^{-1}(I))=q_I((\ker\alpha)^\bot\cap \alpha^{-1}(I)+\ker\alpha)  =  q_I(\ker\alpha)$. Thus  $q_I((\ker\alpha)^\bot)=(\ker\alpha_I)^\bot$ and $q_I(\ker\alpha)$ is a complemented ideal in $A/I$. Since $\ker\alpha|_I =\ker\alpha\cap I$ we see that   $(\ker\alpha|_I)^\bot=(\ker\alpha)^\bot \cap I$ and therefore these ideals are complementary in $I$.

(ii). In view of part (i) it suffices to note that both $\alpha_I$ and $\alpha|_I$ have hereditary ranges. The former is straightforward and the latter follows from the following relations
$$
\alpha(I)I\alpha(I)=\alpha(IA)I\alpha(AI)=\alpha(I)\alpha(A)I\alpha(A) \alpha(I)\subseteq \alpha(I)\alpha(A)\alpha(I)=\alpha(I).
$$
\end{proof}
\begin{defn}\label{J-pairs definition} 
Let $I, I', J$ be ideals in $A$ where $J\subseteq (\ker\al)^\bot$. We  say that $(I,I')$ is a \emph{$J$-pair} for a $C^*$-dynamical system $(A,\al)$  if 
$$ 
I  \textrm{ is positively invariant,}\quad  J\subseteq I' \quad \textrm{and}\quad  I'\cap \al^{-1}(I)=I.
$$
The set of $J$-pairs for $(A,\al)$ is equipped with a natural  partial order induced by inclusion: $(I_1, I_1') \subseteq (I_2,I_2')$ $\stackrel{def}\Longleftrightarrow$  $I_1\subseteq I_2$ and $I_1'\subseteq I_2'$.
	\end{defn}
	\begin{lem}\label{J-pairs from representations}
If $(\pi,U)$ is a $J$-covariant representation then $(\ker\pi, I_{(\pi,U)})$ is a $J$-pair.
\end{lem}
\begin{proof}
It is clear that $\ker\pi$ is positively invariant,  $J\subseteq I_{(\pi,U)}$ and $\ker\pi\subseteq I_{(\pi,U)}$. In particular,  $\ker\pi\subseteq I_{(\pi,U)}\cap \al^{-1}(\ker\pi)$. For the reverse inclusion, note that for any $a\in I_{(\pi,U)}\cap \al^{-1}(\ker\pi)$ we have $\pi(a)=U^*U\pi(a)=U^*U\pi(a)U^*U=U^*\pi(\alpha(a))U=0$. 
\end{proof}
	Clearly, if $(I,I')$ is a $J$-pair, then $I$ is $J$-invariant  and $I+J\subseteq I'$. Note that then $(I,I+J)$ is  also a  $J$-pair, but in general $I+J\neq I'$, cf. \cite[Remark 3.2 and Example 3.1]{kwa-rever}.
	 We have the following relationship between gauge-invariant ideals  and $J$-pairs for $(A,\alpha)$.

\begin{thm}\label{gauge-invariant ideals thm}
Let $(A,\alpha)$ be a $C^*$-dynamical system and let  $J$ an ideal in $(\ker\alpha)^\bot$. The relations
\begin{equation}\label{lattice isomorphisms relations}
I=A\cap \I, \qquad I'=\{a\in A: (1-u^*u)a \in \I\} 
\end{equation}
establish an order preserving  bijective correspondence between 
$J$-pairs $(I,I')$ for $(A,\alpha)$ and  gauge-invariant ideals $\I$ in $C^*(A,\alpha;J)$. Moreover, for objects satisfying \eqref{lattice isomorphisms relations}   we have  a natural isomorphism
$$
C^*( A,\alpha;J)/ \I\cong C^*(A/I,\al_{I};q_I(I'))
$$
and if $I'=I+J$ (equivalently $\I$ is generated by its intersection with $A$), then $\I$ is Morita-Rieffel equivalent to $C^*(I,\alpha|_I;I\cap J)$.

\end{thm}
\begin{proof}
We use   Proposition \ref{identification of crossed products} to identify $C^*(A,\alpha;J)$ with $\OO(J,E_\alpha)$. By Theorem \ref{thm for referee} and Proposition \ref{J-pairs and T-pairs}  the relations $I=A\cap \I$ and $I'=A\cap (\I+ E_\alpha E_\alpha^*)$  establish  a bijective correspondence between $J$-pairs $(I,I')$ for $(A,\alpha)$ and  gauge-invariant ideals $\I$ in $C^*(A,\alpha;J)$. Note that  $E_\alpha E_\alpha^*=u^*\alpha(A)A\alpha(A)u$ and recall that $(1-u^*u)$ is a multiplier of $C^*( A,\alpha;J)$ by Corollary  \ref{initial projection is a mutliplier}. Thus  $a\in I'$ implies that  $(1-u^*u)a\in \I$. Conversely, if $a\in A$ is such that $(1-u^*u)a \in \I$ then, since $(1-u^*u)a=a- u^*\alpha(a)u$, we have $a\in \I+E_\alpha E_\alpha^*$. Hence $I'=\{a\in A: (1-u^*u)a \in \I\}$.  Since $E_{\alpha_I}\cong E_I$ we get $
C^*( A,\alpha;J)/ \I\cong C^*(A/I,\al_{I};q_I(I'))
$, see Theorem \ref{thm for referee} and Proposition \ref{J-pairs and T-pairs}.

Clearly, the bijective correspondence $(I,I')\longleftrightarrow\I$ preserves order. Thus $\I$ is generated by $I$ if and only if $I'=I+J$. In this case  we see that $\I$ is Morita-Rieffel equivalent to $C^*(I,\alpha|_I;I\cap J)$ by Theorem \ref{thm for referee}, because $E_{\alpha|_{I}}\cong IE_\alpha$.
\end{proof}
\begin{rem}\label{remark on simplicity}
The pairs  $(\{0\}, J)$ and $(\{0\},(\ker\alpha)^\bot)$ are always  $J$-pairs. Thus Theorem \ref{gauge-invariant ideals thm} implies that $C^*(A,\alpha;J)$ is never simple unless $J=(\ker\alpha)^\bot$, that is unless $C^*(A,\alpha;J)=C^*(A,\alpha)$. More detailed necessary conditions and certain sufficient conditions for $C^*(A,\alpha)$ to be simple can be found in \cite[Theorem 4.2]{kwa-rever}. The only simplicity result we explicitly state in this paper is Proposition \ref{simplicity result} below.
\end{rem}
\begin{cor}\label{corollary for dolary}
If the kernel of $\alpha$ is a complemented ideal in $A$ then 
the relations
\begin{equation}\label{lattice isomorphisms relations2}
I=A\cap \I,  \qquad \II \textrm{ is generated by }  I 
\end{equation}
establish a  bijective correspondence between 
invariant ideals  $I$ for $(A,\alpha)$ and  gauge-invariant ideals $\I$ in $C^*(A,\alpha)$, under which we have  
$
C^*( A,\alpha)/ \I\cong C^*(A/I,\al_{I})
$
and $\I$ is Morita-Rieffel equivalent to $C^*(I,\alpha|_{I})$.  
\end{cor}
 \begin{proof}
 Let $(I,I')$ be a $(\ker\alpha)^\bot$-pair and let $\I$ be the corresponding gauge-invariant ideal in  $C^*( A,\alpha)$. 
Using \eqref{lattice isomorphisms relations} and Lemma \ref{lemma on systems with complemented kernel} for any such pair we get 
\begin{align*}
I'&=\{a\in A: (1-u^*u)a \in \I\}=\{a\oplus b \in \ker\alpha\oplus (\ker\alpha)^\bot: a \in I\}
\\
&= (\ker\alpha \cap I) \oplus (\ker\alpha)^\bot.
\end{align*}
Hence $I'=I+(\ker\alpha)^\bot$, that is $\I$ is generated by $I$. By Lemma \ref{complemented kernel invariance lemma},  $\ker\alpha_I = q_I(\ker\alpha)$ and $(\ker\alpha)^\bot \cap I=(\ker\alpha|_I)^\bot$. In particular, we  get $q_I(I')=q_I((\ker\alpha)^\bot)=(\ker\alpha_I)^\bot$. Now the assertion follows from  Theorem \ref{gauge-invariant ideals thm}.
\end{proof}
For crossed products of reversible $C^*$-dynamical systems we can actually identify gauge-invariant ideals up to isomorphism.
\begin{prop}\label{restrictions for reversible systems}
If $(A,\alpha)$ is a reversible $C^*$-dynamical system and $\I$ is a gauge-invariant ideal in $C^*(A,\alpha)$, then
$
\I\cong C^*(I,\alpha|_{I})$ where $I=A\cap \I$.
\end{prop}
\begin{proof} In view of Propositions \ref{identification of crossed products} and \ref{reversible correspondence prop} it suffices to apply the general result for Hilbert bimodules \cite[Theorem 10.6.]{ka3}. Alternatively, the assertion can be proved directly using  Lemma \ref{complemented kernel invariance lemma} and Propositions \ref{crossed product for reversible system} and  \ref{gauge-uniqueness theorem}.
\end{proof}
\subsection{Extensions with complemented kernel}
We can use Corollary \ref{corollary for dolary} to describe up to Morita-Rieffel equivalence all gauge-invariant ideals in an arbitrary crossed product $C^*(A,\alpha;J)$. More specifically, there is  a canonical construction of a $C^*$-dynamical system $(A_J,\alpha_J)$ such that   $C^*(A,\alpha;J)\cong C^*(A_J,\alpha_J)$ and the kernel of $\alpha_J$ is  complemented. The  system $(A_J,\alpha_J)$ was considered in \cite[Subsection 6.1]{kwa-leb}, in the case $A$ is unital, but the construction works also in our general context. 
\begin{defn}\label{complementing the kernel}
For every $C^*$-dynamical system $(A,\alpha)$  and 
an ideal $J$  in $(\ker\alpha)^\bot$ we put $$
A^J:=
\big(A/\ker\alpha\big) \oplus \big( A/J\big)
$$
and define an endomorphism $\alpha^J:A^J\to A^J$ by the formula 
$$
A^J\ni (a +\ker\alpha)\oplus (b + J)\stackrel{\alpha^J }{\longrightarrow}(\alpha(a) +\ker\alpha)\oplus
(\alpha(a) + J)\in A^J.
$$
\end{defn}
The system $(\mathcal{A}^J,\alpha^J)$ extends $(A,\alpha)$, in the sense that  the  map 
$$
A \ni a \stackrel{\iota^J }{\longmapsto} \big(a +\ker\alpha\big)\oplus \big(a + J\big)\in A^J
$$
is an injective homomorphism that intertwines $\alpha$ and $\alpha^J$. Moreover, the kernel of $\alpha^J$  coincides with the direct summand $A/J$ in $A^J$. Hence  $(\ker\alpha^J)^\bot$ corresponds to the $A/\ker\alpha$ summand. We also note that $\iota^J: A\to A^J$ is an isomorphism if and only if 
$\ker\alpha$ is a complemented ideal in $A$ and $J=(\ker\alpha)^\bot$.
\begin{lem}
If $(I,I')$ is a $J$-pair for $(A,\alpha)$ then 
\begin{equation}\label{I i i' ideal}
(I,I')^J:=q_{\ker\alpha }(I) \oplus  q_{J}(I') \triangleleft A^J
\end{equation}
is an invariant ideal in $(A^J,\alpha^J)$ such that 
\begin{equation}\label{intersection of I jej}
I=\iota_J^{-1}((I,I')^J).
\end{equation}
\end{lem}
\begin{proof}
Let us prove \eqref{intersection of I jej} first. Since $I\subseteq I'$ we have $I\subseteq \iota_J^{-1}((I,I')^J)$. If $a \in \iota_J^{-1}((I,I')^J)$, then  $a= i+ k$ for some $i\in I$, $k \in  \ker\alpha$, and $a\in I'$.  This implies that  $k\in I'\cap \ker\alpha \subseteq I'\cap \alpha^{-1}(I)=I$. Hence $a\in I$ and \eqref{intersection of I jej} holds.
\\
Now using \eqref{intersection of I jej} and the equality $\alpha^J(A^J)=\iota_J(\alpha(A))$  we get
$$
\alpha^J(A^J)\cap(I,I')^J\subseteq \iota_J(I)\subseteq \alpha^J((I,I')^J).
$$
On the other hand, since $\alpha(I)\subseteq I\subseteq I'$ we have  $\alpha^J((I,I')^J)\subseteq (I,I')^J$. Therefore $\alpha^J((I,I')^J)=\alpha^J(A^J)\cap(I,I')$.
It follows that $\alpha^J: (I,I')^J \cap (\ker\alpha^J)^{\bot}\to \alpha^J(A^J)\cap(I,I')$ is an isomorphism and this implies that $(I,I')^J$ is invariant in $C^*(A^J,\alpha^J)$.
\end{proof}
\begin{prop}\label{general gauge-invariant ideals description} Let $(A,\alpha)$  be a $C^*$-dynamical system  and let  $J$ be an ideal in $(\ker\alpha)^\bot$. The embedding $\iota^J$ extends to a gauge-invariant isomorphism 
$$
C^*(A,\alpha;J)\cong C^*(A^J,\alpha^J).
$$
If $\I$ is a  gauge-invariant ideal  in $C^*(A,\alpha;J)$ corresponding to a $J$-pair $(I,I')$ for $(A,\alpha)$,  then $\I$  is mapped by the above isomorphism onto a gauge-invariant ideal in $C^*(A^J,\alpha^J)$ which is generated by the  ideal $(I,I')^J$ given by \eqref{I i i' ideal}. 
In particular, $\I$ is Morita-Rieffel equivalent to $C^*((I,I')^J,\alpha^J|_{(I,I')^J})$. 
\end{prop}
\begin{proof}
Let us denote by $u$ and $v$ the universal partial isometries in $C^*(A,\alpha;J)$  and  $C^*(A_J,\alpha_J)$ respectively. It is clear that $(\iota_J, v)$ is an injective representation of $(A,\alpha)$ in $C^*(A_J,\alpha_J)$ that admits a gauge action. Using Lemma \ref{lemma on systems with complemented kernel} we get that
$
\{a\in A: (v^*v)\iota_J(a)=\iota_J(a)\}=J.
$
By virtue of Proposition \ref{gauge-uniqueness theorem} we see that $\iota_J\rtimes  v:C^*(A,\alpha;J)\to C^*(A_J,\alpha_J)$ is a gauge-invariant isomorphism. Note, again using Lemma \ref{lemma on systems with complemented kernel}, that for any $a, b\in A$ we have
$$
(1-v^*v) \big((a+\ker\alpha)\oplus (b +J)\big) = 0\oplus (b +J)= (1-v^*v) \iota_J(b)
$$
and 
$$
(v^*v) \big((a+\ker\alpha)\oplus (b +J)\big) = (a+\ker\alpha)\oplus  0 = (v^*v)\iota_J(a).
$$
 Let us now fix a $J$-pair $(I,I')$ in $(A,\alpha)$ and let $(I,I')^J$ be the corresponding invariant ideal in  $(A^J,\alpha^J)$ given by \eqref{I i i' ideal}.  In view of the above  equalities, we have 
\begin{equation}\label{kain i abel}
(I,I')^J=(v^*v)\iota_J(I) + (1-v^*v)\iota_J(I').
\end{equation}
Let $\I^{J}$ be the ideal in in $C^*(A^J,\alpha^J)$ generated by $(I,I')^J$ and let $\I:=(\iota_J\rtimes  v)^{-1}(\I^J)$. Then
\begin{align*}
\{a\in A: (1-u^*u)a \in \I\}&=\{a\in A: (1-v^*v)\iota_J(a) \in \I^{J}\} 
\\
&=\{a\in A: (1-v^*v)\iota_J(a) \in (I,I')^J\} = I'. 
\end{align*}
This, together with \eqref{intersection of I jej}, shows that $\I$ is the gauge-invariant ideal corresponding to the $J$-pair $(I,I')$. Hence $\I$ is Morita-Rieffel equivalent to $C^*((I,I')^J,\alpha^J|_{(I,I')^J})$, by Corollary \ref{corollary for dolary}.
\end{proof}

\subsection{$K$-theory of gauge-invariant ideals}
We have the following generalization of the classical Pimsner-Voiculescu  sequence.

\begin{prop}\label{Voicu-Pimsner for interacts} For an ideal $J$ in $(\ker\alpha)^\bot$ we have the following exact sequence  
$$
\begin{xy}
\xymatrix{
      K_0(J) \ar[rr]_{K_0(\iota)- K_0(\alpha|_{J})} & & K_0(A) \ar[rr]_{K_0(\iota)\,\,\,\,\,\,\,\,\,}  & &  \ar[d]  K_0(C^*(A,\alpha;J))
             \\
   K_1(C^*(A,\alpha;J)) \ar[u]  & & K_1(A)  \ar[ll]_{\,\,\,\,\,\,\,\,\,\,\,\,K_1(\iota)}  &  &  \ar[ll]_{K_1(\iota)- K_1(\alpha|_{J})}  K_1(J)
              } 
  \end{xy}, 
$$
where $\iota$ stands for inclusion.
\end{prop}
\begin{proof}
Using  Lemma \ref{probably non-standard proposition} we see that   in the sequence \eqref{Katsura Pimsner voiculescu sequence} we may replace the maps $K_i(\iota_{22})^{-1}\circ K_i(\iota_{11}\circ \phi|_{J})$ with $K_i(\alpha|_{J})$,  $i=0,1$. This results with  the desired sequence. 
\end{proof}
One can combine  results from previous subsections with Proposition \ref{Voicu-Pimsner for interacts} to get exact six-term sequences  for $K$-theory of all gauge-invariant ideals and relevant quotients in the crossed product $C^*(A,\alpha;J)$. We state explicitly only results for gauge-invariant ideals.
\begin{thm}\label{K-theory theorem}
Let $\I$ be a gauge-invariant ideal in $C^*(A,\alpha;J)$ where  $(A,\alpha)$ is a $C^*$-dynamical system  and   $J$ is an ideal in $(\ker\alpha)^\bot$.  Let  $(I,I')$ be the $J$-pair for $(A,\alpha)$ given by \eqref{lattice isomorphisms relations}.
We have  
$$ 
 K_*(\I)\cong K_*\left( C^*((I,I')^J,\alpha^J|_{(I,I')^J})\right),
$$
 where $(I,I')^J$ is given by \eqref{I i i' ideal}, and in particular if 
 $K_1((I,I')^J)=0$  then
\begin{equation}\label{K-theory for I and I's}
	K_0(\I)\cong \coker \left(K_0(\iota)- K_0(\alpha^J|_{q_{\ker\alpha }(I)})\right), \qquad
K_1(\I)\cong \ker \left(K_0(\iota)- K_0(\alpha^J|_{q_{\ker\alpha }(I)})\right),
\end{equation}
where $\alpha^J|_{q_{\ker\alpha }(I)}:q_{\ker\alpha }(I)\oplus\{0\} \to (I,I')^J$ is the restriction of $\alpha^J$ and $\iota:q_{\ker\alpha }(I) \to (I,I')^J$ is the inclusion.  If $\I$ is generated by $I$, that is if $I'=I+J$, then 
$$
K_*(\I)\cong K_*\left(C^*(I,\alpha|_I;I\cap J)\right)
,
$$ and if additionally $K_1(I)=K_1(I\cap J)=0$, then 
$$
	K_0(\I)=\coker (K_0(\iota)- K_0(\alpha|_{I\cap J})), \qquad
K_1(\I)=\ker (K_0(\iota)- K_0(\alpha|_{I\cap J}))
$$
where $\alpha|_{I\cap J}: I\cap J\to  I$ is the restriction of $\alpha$ and $\iota:I\cap J\to I$ is the   inclusion.  
\end{thm}
\begin{proof}
By the last part of Proposition \ref{general gauge-invariant ideals description},  $\I$ is Morita-Rieffel equivalent to the crossed product $C^*((I,I')^J,\alpha^J|_{(I,I')^J})$. Hence the corresponding $K$-groups are isomorphic by \cite[Proposition B.5]{katsura}, see also \cite[Remark B.6]{katsura}. If $K_1((I,I')^J)=0$ then also $(\alpha^J|_{(I,I')^J})^\bot=q_{\ker\alpha }(I)\oplus\{0\}$ has $K_1$-group equal to zero. Thus applying  Proposition \ref{Voicu-Pimsner for interacts} to the system $((I,I')^J, \alpha^J|_{(I,I')^J})$ and the ideal $(\ker\alpha^J|_{(I,I')^J})^\bot$ we get the second part of the  assertion and \eqref{K-theory for I and I's}. In view of the second part of Theorem \ref{gauge-invariant ideals thm}, the above argument proves also the first part of the assertion.
\end{proof}
\begin{cor} If $\ker\alpha$ is a complemented ideal in $A$ then for every gauge-invariant ideal $\I$ in $C^*(A,\alpha)$ we have  
$$
K_*(\I)\cong K_*\left(C^*(I,\alpha|_I)\right), \qquad \text{where } I:=\I\cap A.
$$ 
If additionally 	$K_1(I)=0$, then 
	$$
	K_0(\I)=\coker (K_0(\iota)- K_0(\alpha|_{I\cap (\ker\alpha)^{\bot}})), \qquad
K_1(\I)=\ker (K_1(\iota)- K_1(\alpha|_{I\cap (\ker\alpha)^{\bot}}))
$$
where $\alpha|_{I\cap (\ker\alpha)^{\bot}}: I\cap (\ker\alpha)^{\bot}\to  I$ is restriction of $\alpha$ and $\iota:I\cap (\ker\alpha)^{\bot}\to I$ is the   inclusion.  
	\end{cor}
\begin{proof}
It suffices to combine the second parts of Theorem  \ref{K-theory theorem}  and Corollary \ref{corollary for dolary}. 
\end{proof}
\subsection{Reversible extensions} 

We fix a $C^*$-dynamical system $(A,\al)$ and an ideal $J$ in $(\ker\al)^\bot$. We generalize a  construction of a reversible $C^*$-dynamical system $(B,\beta)$  associated to the triple $(A,\alpha;J)$ in \cite[Subsection 3.1]{kwa-rever}, see also \cite[Section 4]{kwa-ext},  to the case when $\alpha$ is not necessarily extendible.   The system $(B,\beta)$  can be viewed as a direct limit of approximating $C^*$-dynamical systems $(B_n,\beta_n)$, $n\in \N$.
We denote by $q:A\to A/J$ the quotient map and for each  $n\in \N$
we put
$$\label{algebry A_n 2}
A_n:=\al^n(A)A \al^n(A).
$$
The $C^*$-algebra $B_n$, $n\in \N$, is a direct sum of the form
 $$
B_n=q(A_{0})\oplus q(A_{1})\oplus ... \oplus q(A_{n-1}) \oplus A_{n}, 
$$
and  the endomorphism $\beta_n:B_n \to B_n$ is   given by the formula
$$
\beta_n(a_{0}\oplus a_{1}\oplus ... \oplus a_{n})=a_{1}\oplus a_{ 2}\oplus ... \oplus q(a_{n})\oplus \al(a_n),
$$
where $a_k\in q(A_k)$, $k=0,...,n-1$, and $a_n\in A_n$, $n>0$. Thus we get  a sequence $(B_n,\beta_n)$, $n\in \N$, of $C^*$-dynamical systems where $(B_0,\beta_0)= (A,\alpha)$. We consider bonding homomorphisms $\al_n:B_n \to B_{n+1}$, $n\in \N$, whose action is 
presented by the diagram
$$
\begin{xy}
\xymatrix@C=3pt{
      **[r]  B_n \ar[d]^{\al_n}& = &  q(A_{0}) \ar[d]^{id} &  \oplus & ... & \oplus &
      q(A_{n-1})\ar[d]^{id}& \oplus & A_{n} \ar[d]^{q}    \ar[rrd]^{\al}   \\
       B_{n+1} & = &  q(A_{0})& \oplus & ... &  \oplus& q(A_{n-1}) & \oplus & q(A_{n}) & \oplus  & A_{n+1}
        }
  \end{xy}.
$$
In other words, $\alpha_n$ is   given by  the formula
$$
\al_n (a_{0}\oplus ... \oplus a_{ n-1}\oplus a_{n})= a_{0}\oplus ... \oplus a_{ n-1}\oplus q(a_{n})
\oplus  \al(a_{n}),
$$
where $a_{k}\in q(A_{k})$, $k=0,...,n-1,$ and $a_{n}\in A_{n}$. Plainly, the homomorphisms $\al_n$ are  injective and we have 
$$
\alpha_n\circ \beta_n=\beta_{n+1}\circ \alpha_n, \qquad n \in \N.
$$
Accordingly,   we get the  direct  sequence  of $C^*$-dynamical systems: 
$$
(B_0,\beta_0) \stackrel{\al_0}{\longrightarrow} (B_1,\beta_1)
\stackrel{\al_1}{\longrightarrow} (B_2,\beta_2)  \stackrel{\al_2}{\longrightarrow}...\, .
$$
We denote by $(B,\beta)$ a direct limit of the above direct sequence.  More precisely, $B=\underrightarrow{\lim\,\,}\{B_{n},\alpha_{n}\}$ is the $C^*$-algebraic direct limit, and $\beta$ is determined by the formula $\beta(\phi_n(a))=\phi_n(\beta_n(a))$ where    $\phi_n:B_n \to B$ is the natural (injective) homomorphism, $a\in B_n$ and  $n\in \N$. That is we have
\begin{equation}\label{beta description}
\beta(\phi_{n}(a_{0}\oplus a_{1}\oplus ... \oplus a_{{n}}))=\phi_{n-1}(a_{1}\oplus a_{ 2}\oplus ... \oplus a_{n}).
\end{equation}
We now extend the main parts of  \cite[Theorem 3.1 and Proposition 4.7]{kwa-rever}, see also \cite[Remark 3.3]{kwa-rever}.
\begin{thm}\label{rozszerzenie a repr. kowariantne2}
The $C^*$-dynamical system $(B,\beta)$ described above is reversible and we may assume a natural identification 
$$
C^*(A,\alpha;J)=C^*(B,\beta) 
$$
under which we have 
$$
B=\clsp\{u^{*k}au^k: a\in \alpha^k(A)A\alpha^k(A), k\in \mathbb{N}\} \quad \textrm{ and  }\quad \beta(b)=ubu^*,\,\,  b\in B.
$$
In particular,  the relation 
$
 \widetilde{\pi}(\sum_{k=0}^n u^{*k}a_ku^k)=\sum_{k=0}^n U^{*k}\pi(a_{k})U^k$,  $a_k \in \alpha^k(A)A\alpha^k(A)$,
establishes  a one-to-one correspondence between  $J$-covariant representations $ (\pi, U)$ of $(A,\alpha)$ and
covariant representations $(\widetilde{\pi}, U)$ of  $(B,\beta)$.
\end{thm}
\begin{proof}
Let us prove first that $(B,\beta)$ is reversible. To this end, take  $a=a_{0}\oplus a_{1}\oplus ... \oplus a_{n}\in B_n$ and $b=b_{0}\oplus b_{1}\oplus ... \oplus b_{n}\in B_{n-1}$ for  $n>1$. Then, in view of \eqref{beta description}, we get
$$
\beta(\phi_n(a))\phi_{n-1}(b)\beta(\phi_n(a))=\beta(\phi_{n}(0\oplus a_1 b_{0}a_1 \oplus ... \oplus a_{n} b_{n-1} a_n)).
$$
This implies that $\beta(B)B\beta(B)=\beta(B)$. Hence $\beta(B)$ is a hereditary $C^*$-subalgebra in $B$. The ideal $\ker\beta$ is complemented in $B$ as 
$
B_n\cap \ker\beta=\{\phi_{n}(a_{0}\oplus 0\oplus ... \oplus 0): a_0\in q(A_0)\}
$ is complemented in $B_n$ for every $n >0$. Thus $(B,\beta)$ is reversible.

Now, for each $n\in \N$, we  define $C_n:=\{\sum_{k=0}^n u^{*k}a_ku^k: a_k\in \alpha^k(A)A\alpha^k(A), k=0,...,n\}\subseteq C^*(A,\alpha;J)$. We also put $C:=\overline{\bigcup_{n\in \N} C_n}$. It follows from \eqref{general form of multiplication in cross} that $C_n$, $n\in \N$, and $C$ are  $C^*$-algebras. Recall, see Proposition \ref{on the structure of crossed products}, that $\{u^{*k}u^k\}_{k\in \mathbb{N}}$  is a  decreasing sequence of orthogonal projections that commute with elements of $A$. Hence they commute with elements of $C$. Exactly as in the proof of  \cite[Statement 1]{kwa-ext}, one checks that if $a=\sum_{k=0}^n u^{*k}b_ku^k\in C_n$,  $b_k\in \alpha^k(A)A\alpha^k(A)$, $k=0,...,n$, then putting $a_k=\sum_{i=0}^{k}\alpha^{k-i}(b_i)$ we get that
\begin{equation}\label{general form of an element}
a=\sum_{k=0}^{n-1}u^{*k}(1-u^*u)a_ku^k + u^{*n}a_n u^n,
\end{equation}
where $1$ is the unit in $M(C^*(A,\alpha;J))$. 
Hence \eqref{general form of an element} is a general form of an element in $C_n$. In particular, since $u^{*k}(1-u^*u)u^k=(u^{*k}u^k- u^{*k+1}u^{k+1})$, $k=0,...,n-1$, and $u^{*n}u^{n}$ are mutually orthogonal projections commuting with elements of $C_n$, we see that $C_n$ admits the following direct sum decomposition
$$
C_n=\bigoplus_{k=0}^{n-1} (u^{*k}u^k- u^{*k+1}u^{k+1}) C_n \oplus   u^{*n}u^{n}C_n
$$
Since $u^k$ is a partial isometry  it follows that
$$
\alpha^k(A) C \alpha^{k}(A)=u^k A u^{*k} C u^k A u^{*k}\ni a \to u^{*k}au^{k} \in C
, \qquad k=1,...,n,
$$
is a $*$-homomorphic isometry. Since $J= u^*uA\cap A$ we also see, cf.  for instance \cite[Lemma 10.1.6]{kadison_ringrose}, that 
$$
(1-u^*u) A \ni a \to q(a)\in q(A) 
$$
is an isomorphism of $C^*$-algebras. Combining these facts we get  
that the formula 
 $$
\Phi_n\left(q(a_{0})\oplus q(a_{1})\oplus ... \oplus q(a_{n-1}) \oplus a_{n}\right)=\sum_{k=0}^{n-1}u^{*k}(1-u^*u)a_ku^k + u^{*n}a_n u^n,
$$
 defines an isomorphism $\Phi_n:B_n \to C_n$. If $a\in C_n$ is given by \eqref{general form of an element}, then using equality $u^{*n+1} \alpha(a_n)  u^{n+1}=u^{*n}  (u^*u) a_n  u^{n}$ we get
 $$
 a=\sum_{k=0}^{n}u^{*k}(1-u^*u)a_ku^k + u^{*n+1} \alpha(a_n)  u^{n+1}.
 $$ 
 Therefore  $\Phi_{n+1}\circ \alpha_n=\Phi_n$, $n\in \N$. Hence the isomorphisms $\Phi_n$ induce the isomorphism $\Phi:B\to C$ between  the  inductive limit $C^*$-algebras $B$ and $C$. 

We claim that $\Phi(\beta(b))=ubu^*$, for $b\in C$. Indeed, let $a\in C_n$ is given by \eqref{general form of an element}.  Notice that for $k>0$ we have $uu^{*k}=(uu^*)u^{*k-1}=u^{*k-1}u^{k-1} (uu^*)u^{*k-1}=u^{*k-1}u^{k} u^{*k}$. Therefore, since $u^{k} u^{*k} a_k=a_k$, we get
$$
  u \left(u^{*k}(1-u^*u)a_ku^k \right) u^* = u^{*k-1}(1-u^*u)a_k u^{k-1}.
$$   
Clearly, $u(1-u^*u)a_0 u^{*}=u a_0 u^*-u a_0 u^*=0$. Accordingly,  $\Phi(\beta(\phi_n(a)))=u \Phi(\phi_n(a))u^*$, which proves our claim. 

It readily follows from the definition of $\Phi_n$ that 
$$
u^*u\Phi_n(B_n)=\{\phi_{n}(0\oplus a_1 \oplus ... \oplus a_n): 0\oplus a_1 \oplus ... \oplus a_n\in  B_n\}
=B_n\cap (\ker\beta)^\bot.
$$
This implies that $(\ker\beta)^{\bot}=\{ b\in B: u^*u \Phi(b)= \Phi(b)\}$. 

Concluding the pair $(\Phi, u)$ is an injective covariant representation of $(B,\beta)$ in $C^*(A,\alpha;J)$ that admits gauge-action. 
Thus, by Proposition \ref{gauge-uniqueness theorem},  $\Phi\rtimes u:C^*(B,\beta)\to C^*(A,\alpha;J)$ is an isomorphism which we may use to assumed the described identification.  
 The last part of the assertion follows from the universal properties of the crossed products. 
\end{proof}

\begin{defn}[Definition 3.1 in \cite{kwa-rever}]\label{reversible extension definition} Suppose that  $(A,\alpha)$ is a $C^*$-dynamical system and $J$ is an ideal in $(\ker\alpha)^\bot$. We call the $C^*$-dynamical system $(B,\beta)$ constructed above the  \emph{natural reversible $J$-extension} of $(A,\alpha)$.

\end{defn} 
 
Let $(B, \beta)$ be a natural reversible $J$-extension of $(A,\alpha)$ and suppose that $A=C_0(X)$ is commutative. Then, in view of our construction,  $B$ is also commutative and thus  we may identify it with  $C_0(\X)$ where $\X$ is a locally compact Hausdorff space.  With this identification, $\beta$ is given  by the formula
$$
\beta(b)(\x)=\begin{cases}
b(\widetilde{\p}(\x)),  & \x \in \widetilde{\Delta},
\\
0        & \x \notin \widetilde{\Delta},
\end{cases} 
$$
where $\widetilde{\p}:\widetilde{\Delta}\to \widetilde{\p}(\widetilde{\Delta})$ is a homeomorphism,  $\widetilde{\Delta} \subseteq \X$ is open and $\widetilde{\p}(\widetilde{\Delta})\subseteq \X$ is clopen. 
The pair $(\X,\widetilde{\p})$ is uniquely determined by $(X,\p)$ and the closed set 
\begin{equation}\label{set Y}
Y=\{x\in X: a(x)=0 \textrm{ for all } a\in J\},
\end{equation}
which necessarily  contains $X\setminus \p(\Delta)$. Similarly as in \cite[Proposition 4.7]{kwa-rever}, cf. also  \cite[Theorem 3.5]{kwa-logist}, using the above construction of $(B,\beta)$ one can deduce the following description of $(\X,\widetilde{\p})$. 
 \begin{prop}\label{opis kosmiczakow}
 Up to conjugacy with a homeomorphism,  the above partial dynamical system $(\X,\widetilde{\p})$  can be described as follows:
$$
\X=\bigcup_{N=0}^{\infty}X_N\cup X_\infty
$$
where
$$
X_N=\{(x_0,x_1,...,x_N,0,...): x_n\in \Delta,\, \p(x_{n})=x_{n-1},\,\, n=1,...,N,\,x_N\in
Y\},
$$
$$
X_\infty=\{(x_0,x_1,...): x_n\in \Delta,\,
\p(x_{n})=x_{n-1},\, \,n\geqslant 1\}.
$$
The topology on $\X$ is the product one inherited from $\prod_{n\in \N} (X\cup \{0\})$ where $\{0\}$ is a clopen singleton and  $Y$ is given by \eqref{set Y}. The homeomorphism $\widetilde{\p}:\widetilde{\Delta}\to \widetilde{\p}(\widetilde{\Delta})$ is given by the formula
$$
\widetilde{\p}(x_0,x_1,...)=(\p(x_0),x_0,x_1,...),\qquad \TDelta=\{(x_0,x_1,...)\in \X: x_0 \in \Delta\}.
$$
\end{prop}
\begin{proof} We omit the proof as the assertion will follow from a much more general result we prove below, see Theorem \ref{proposition C-bundle B}. 
\end{proof}

\begin{defn}[cf. Definition 3.5 \cite{kwa-logist}]\label{reversible ext defn} Let  $Y$ be a closed subset of $X$ that contains $X\setminus \p(\Delta)$. We call the dynamical system $(\X,\widetilde{\p})$ described in the assertion of Proposition \ref{opis kosmiczakow}  the \emph{natural reversible $Y$-extension of  $(X,\p)$}.
\end{defn}
Note that for the natural reversible $Y$-extension $(\X,\widetilde{\p})$ of  $(X,\p)$,  the map $\Phi:\X\to X$ given by $\Phi(\x)=x_0$ is surjective and intertwines $\widetilde{\p}$ and $\p$. This justifies the name.
\subsection{Topological freeness and freeness} We turn to a discussion of certain conditions implying uniqueness property and gauge-invariance of all ideals in the crossed products. For reversible and extendible systems the relevant statements in \cite[Subsection 4.5]{kwa-rever} were  deduced from \cite[Theorem 2.20]{kwa-interact}. We will  extend them by applying general results from \cite{kwa-top} and facts presented in Appendix \ref{appendix section}.  
\begin{defn}\label{topolo disco} 
Let $\varphi$ be a partial homeomorphism of a topological (not necessarily Hausdorff) space $X$ with domain being an open set $\Delta\subseteq X$. We say that $\p$   is   {\em
topologically free} if  the set of its periodic points of any given period  $n>0$  has empty interior.   A set $V\subseteq X$ is \emph{invariant}    if 
$
\varphi(V\cap \Delta)= V\cap \varphi(\Delta). 
$
We say that $\varphi$ is (essentially) \emph{free}, if it is topologically free when restricted to any  closed invariant set.
\end{defn}

\begin{defn}\label{dual partial homeomorphism} 
Let $(A,\alpha)$ be a reversible $C^*$-dynamical system. Since $(\ker\alpha)^\bot$ is an  ideal in $A$ and $\alpha(A)=\alpha(A)A\alpha(A)$ is a hereditary subalgebra of $A$ we have the natural identifications: 
$$
\widehat{(\ker\alpha)^\bot}=\{\pi \in \widehat{A}: \pi((\ker\alpha)^\bot)\neq 0\},\qquad \widehat{\alpha(A)}=\{\pi \in \widehat{A}: \pi(\alpha(A))\neq 0\}. 
$$
Thus we treat $\widehat{\alpha(A)}$ and  $\widehat{(\ker\alpha)^\bot}$ as open subsets of $\widehat{A}$.
With these identifications the homeomorphism $\widehat{\alpha}: \widehat{\alpha(A)} \to  \widehat{(\ker\alpha)^\bot}$ dual to the isomorphism $
\alpha: (\ker\alpha)^\bot \to \alpha(A)$ becomes a partial homeomorphism  of the spectrum of $\widehat{A}$, cf. \cite{kwa-rever}. We refer to $\widehat{\alpha}$ as to the \emph{partial homeomorphism dual} to $(A,\alpha)$.\end{defn}

\begin{prop}\label{interactions?}
Let $(A,\alpha)$ be a reversible $C^*$-dynamical system.
\begin{itemize}
\item[(i)]
If $\widehat{\alpha}$ is  topologically free, then every  injective covariant representation $(\pi,U)$ of $(A,\alpha)$ give rise to a faithful representation of  $C^*(A,\alpha)$. 
\item[(ii)] If  $\widehat{\alpha}$  is free, then all ideals in $C^*(A,\alpha)$  are gauge-invariant; hence they are in one-to-one correspondence with invariant ideals in $(A,\alpha)$, cf. Corollary \ref{corollary for dolary}.
\end{itemize}
\end{prop}
\begin{proof} 
By Proposition \ref{identification of crossed products} and Lemma \ref{lemma 5.6},  Theorem \ref{theorem A.6} translates to the desired assertion.
\end{proof}
One can apply the above proposition to an arbitrary crossed product $C^*(A,\alpha;J)$ using the identification $
C^*(A,\alpha;J)=C^*(B,\beta) 
$  from Theorem \ref{rozszerzenie a repr. kowariantne2}, and the following lemma. 
\begin{lem}\label{lemma now trivial}
Let $(A,\alpha)$ be a $C^*$-dynamical system,  $J$ an ideal in $(\ker\alpha)^\bot$, and $(B,\beta)$ the natural reversible $J$-extension of $(A,\alpha)$.
\begin{itemize}
\item[(i)] For any  injective $J$-covariant representation $(\pi,U)$ of $(A,\alpha)$  the corresponding covariant representation $(\widetilde{\pi}, U)$ of  $(B,\beta)$ is injective if and only if $
J=\{a\in A: U^*U\pi(a)=\pi(a)\}$.
\item[(ii)] Relations \eqref{lattice isomorphisms relations} establish  a bijective correspondence between with $J$-pairs $(I,I')$ in $(A,\alpha)$ and invaraint ideals in $(B,\beta)$.
\end{itemize}
\end{lem}
\begin{proof} (ii). Since $
C^*(A,\alpha;J)=C^*(B,\beta) 
$ we get the assertion by applying Theorem \ref{gauge-invariant ideals thm} to $C^*(A,\alpha;J)$ and Corollary \ref{corollary for dolary} to $C^*(B,\beta)$.

(i). It  follows from  item (ii) and Lemma \ref{J-pairs from representations}.
\end{proof}
In practice, in order to use Proposition \ref{interactions?} and Lemma \ref{lemma now trivial},  one has to determine topological freeness and freeness of $\widehat{\beta}$ in terms of $(A,\alpha)$ and $J$.  This can be readily achieved  if $A$ is commutative. 
\begin{defn}[Definition  4.8 in \cite{kwa-rever}]\label{definition of topologicall freness outside}
Let $\varphi$ be a partial mapping of a locally compact Hausdorff space $X$ defined on an open set $\Delta\subseteq X$. We say that a periodic orbit $\OO=\{x, \varphi(x),..., \varphi^{n-1}(x)\}$ of a periodic point $x=\varphi^{n}(x)$    \emph{has an entrance} $y\in \Delta$  if $y\notin \OO$ and $\varphi(y)\in \OO$. We say $\varphi$ is \emph{topologically free outside a set $Y\subseteq X$} if the set of periodic points  whose orbits do not intersect $Y$ and have no entrances have empty interior. 
\end{defn}
\begin{lem}\label{equivalence of topological freedom}
Let  $(\X,\tphi)$ be the $Y$-extension of a partial dynamical system $(X,\varphi)$ where $Y$ is a closed set containing $X\setminus \p(\Delta)$, see Definition \ref{reversible ext defn}. Then
\begin{itemize}
\item[(i)] $\tphi$ is topologically free  if and only if   $\varphi$ is topologically free outside $Y$,
\item[(ii)] $\tphi$ is  free  if and only if   $\varphi$ is  free (has no periodic points).
\end{itemize}
\end{lem}
\begin{proof} Item (i) can be proved exactly as  \cite[Lemma 4.2]{kwa-rever}. Item (ii) is straightforward. 
\end{proof} 
One of the aims of the present paper is to obtain effective conditions implying the properties of crossed products described in Proposition \ref{interactions?} for a class of $C^*$-dynamical systems on $C_0(X)$-algebras. This is achieved in Theorems \ref{uniqueness theorem} and \ref{pure infiniteness theorem} below.

\subsection{Pure infinite crossed products for reversible $C^*$-dynamical systems}
In this subsection, we fix a reversible $C^*$-dynamical system $(A,\alpha)$. The property that we are about to introduce appears (without a name) in a number of proofs of pure infiniteness for crossed products. As we explain in more detail below, in the context of crossed products, this property is formally weaker than spectral freeness \cite{pp}, topological freenees, proper outerness \cite{Elliot} and aperiodicity \cite{KS}, but the general relationship between these notions is not  completely clear. 
\begin{defn}
Let $A$ be a $C^*$-subalgebra of a $C^*$-algebra $B$. We say that $A^+$ \emph{supports elements} of $B^+$ if for every if  for  every $b \in B^+ \setminus \{0\}$ there exists $a\in A^+$ such that $a\precsim b$. We say that that $A^+$ \emph{residually supports elements of $B^+$} if for every  ideal $I$ of $B$,   $q_I(A)^+$ supports elements of $q_I(B)$. 
\end{defn}
\begin{rem}\label{meaningless remark}
If $A^+$ is a filling family for $B$ in the sense of  \cite[Definition 4.2]{ks1} then $A^+$ residually supports elements of $B^+$  (it is not clear  whether the converse implication holds). Thus, if $A$ is commutative or seperable and $\alpha:A\to A$ is a residually properly outer authomorphism, then \cite[Theorem 3.8]{ks2} implies that $A^+$ residually supports elements of $C^*(A,\alpha)^+$. In \cite[Proposition 3.9]{pp} it is shown that $A^+$ supports elements of $B^+$ if and only if for  every $b \in B^+ \setminus \{0\}$ there is  $z\in B$ such that $zaz^*$ is a non-zero element of $A$. In particular, \cite[Lemma 3.2]{pp} implies that if $\alpha:A\to A$ is an automorphism and the corresponding $\Z$-action is spectrally free in the sense of \cite[Definition 1.3]{pp}, then $A^+$ residually supports elements of $C^*(A,\alpha)^+$.
\end{rem}

In connection with Remark \ref{meaningless remark} we  show that the notion of residual aperiodicity  introduced in \cite[Definition 8.19]{KS}, for (a semigroup version of) extendible reversible systems, implies that $A^+$ residually supports elements of $C^*(A,\alpha)^+$.
\begin{defn}\label{aperiodicity for corner systems definition} We  say that an extendible reversible $C^*$-dynamical  system $(A,\alpha)$ is \emph{aperiodic} if 
for  each $n>0$, each $a\in A$ and every hereditary
subalgebra $D$ of $A$ 
$$
\inf \{\|d a \alpha^n(d)\| : d\in D^+,\,\, \|d\|=1\}=0.
$$
We say that $(A,\alpha)$  is \emph{residually aperiodic} if the quotient system $(A/I,\alpha_I)$ is aperiodic  for every  invariant  ideal $I$ in $(A,\alpha)$.
\end{defn}

\begin{prop}\label{aperiodicity implies supportedness}
If  $(A,\alpha)$ is residually aperiodic then   $A^+$ residually supports elements of $C^*(A,\alpha)^+$.
\end{prop}
\begin{proof}
By \cite[Lemmas  8.18]{KS}, \cite[Corollary  4.7]{KS}  every ideal in  $C^*(A,\alpha)$ is generated by it intersection with $A$.
 Let $\I$ be an ideal in $C^*(A,\alpha)$. By  Corollary \ref{corollary for dolary}, we have the isomorphism $
C^*(A,\alpha)/\I \cong C^*(A/I,\alpha_I)$ where $I:=A\cap \I$ is an invariant ideal in $(A,\alpha)$. The system  $(A/I,\alpha_I)$ is reversible  by Lemma \ref{complemented kernel invariance lemma} ii).  Fix a positive element $b$ in $C^*(A,\alpha)/\I $. We may assume that $\|b\|=1$. Applying  to $(A/I,\alpha_I)$  \cite[Lemmas 4.2 and 8.18]{KS}, we may find a positive contraction $h \in A/I$ such that 
\begin{equation}\label{compression relations}
\|h\mathcal{E}(b)h -hbh\|\leq 1/4, \qquad \|h\mathcal{E}(b)h\| \geq \|\mathcal{E}(b)\|-1/4=3/4
\end{equation}
where $\mathcal{E}$ is the conditional expectation from $C^*(A/I,\alpha_I)$ onto $A/I$. Putting $a:=(h\mathcal{E}(b)h -1/2)_+\in A/I$ we have that $a\neq 0$ because $\|h \mathcal{E}(b)h\| > 1/2$. Moreover, by [27, Proposition 2.2], relations  $\|h \mathcal{E}(b)h\| > 1/2$ and $\|h\mathcal{E}(b)h -hbh\|\leq 1/4$ imply that  $a\precsim hb h$ relative to $
 C^*(A/I,\alpha_I)\cong C^*(A,\alpha)/\I $.
\end{proof}

Before we  prove the main result of this subsection we need two lemmas. 
\begin{lem}\label{this lemma makes no sense}
 If $A^+$ residually supports elements of $C^*(A,\alpha)^+$ then every ideal in $C^*(A,\alpha)$ is gauge-invariant.
\end{lem}
\begin{proof}
Let $\I$ be an ideal in $C^*(A,\alpha)$ and let $\langle I\rangle$  be the smallest ideal  in $C^*(A,\alpha)$ containing $I:=\I\cap A$. By Corollary \ref{corollary for dolary}, we may identify $C^*(A,\alpha)/\langle I\rangle$ with  $C^*(A/I,\alpha_I)$. We have a natural epimorphism $\Phi:C^*(A,\alpha)/\langle I\rangle\to  C^*(A,\alpha)/\I$ which is injective on $A/I$. For any   non-zero positive element $b$ in $C^*(A,\alpha)/\langle I\rangle$ there is a non-zero positive element $a$ in $A/I$ such that $a\precsim b$. Since $0\neq \Phi(a)\precsim\Phi(b)$, we conlude that $\Phi(b)\neq 0$. Thus $\ker\Phi=\{0\}$ and therefore $\I=\langle I\rangle$ is gauge-invariant. 
\end{proof}
\begin{lem}\label{lemma for other guys}
Let $A\subseteq B$ be  $C^*$-algebras and let $A$ be of real rank zero. The following conditions are equivalent
\begin{itemize}
\item[(i)] Every non-zero positive element in $A$ is properly infinite in $B$.
\item[(ii)] Every non-zero projection in $A$ is properly infinite in $B$.
\end{itemize}
\end{lem}
\begin{proof}
Implication (i)$\Rightarrow$(ii) is trivial. Assume that (ii) holds and let $a\in A$ be a non-zero positive element. By \cite[Theorem 2.6]{Brown-Ped} there is an approximate unit $\{p_\lambda: \lambda \in \Lambda\}$ in $\overline{aAa}$ consisting of projections. Thus, by \cite[Proposition 2.7(i)]{kr},  $p_\lambda\precsim a$  for all $\lambda$, in $A$ and all the more  in $B$. Applying  \cite[Lemma 3.17(ii)]{kr}  we see that  $\{p_\lambda: \lambda \in \Lambda\}\subseteq J(a):=\{x\in B: a\oplus |x|\precsim a\}$. Thus $B\{p_\lambda: \lambda \in \Lambda\}B \subseteq J(a)$ because $J(a)$ is an ideal, see \cite[Lemma 3.12(i)]{kr}. On the other hand $J(a)\subseteq BaB$ by \cite[Lemma 3.12(iii)]{kr} and since we clearly have $BaB\subseteq B\{p_\lambda: \lambda \in \Lambda\}B$ it follows that  $J(a)=B a B$. Hence \cite[Lemma 3.12(iv)]{kr} tells us that $a$ is properly infinite in $B$.
\end{proof}
\begin{rem}
The equivalence of (i) and (ii) in Lemma \ref{lemma for other guys} answers the question posed in the proof of \cite[Theorem 4.4]{gs}: it  shows that \cite[Theorem 4.4]{gs} can be deduced from  \cite[Theorem 4.2]{gs}. 
\end{rem}

\begin{prop}[pure infiniteness criterion]\label{pure infiniteness for reversible systems}
Let $(A,\alpha)$ be a reversible $C^*$-dynamical system such that  $A^+$ residually supports elements of $C^*(A,\alpha)^+$.  Suppose also that either  $A$ has the ideal property or that $A$ is separable and there  finitely many invariant ideals in  $(A,\alpha)$. The following statements are equivalent:
\begin{itemize}
\item[(i)] Every non-zero positive element in $A$ is properly infinite in $C^*(A,\alpha)$.
\item[(ii)] $C^*(A,\alpha)$ is purely infinite. 
\item[(iii)] $C^*(A,\alpha)$ is purely infinite and has the ideal property. 
\item[(iv)] Every non-zero hereditary $C^*$-subalgebra in any quotient $C^*(A,\alpha)$ contains an infinite projection.\end{itemize}
If $A$ is of real rank zero, then each of the above conditions is equivalent to
\begin{itemize}
\item[(i')] Every non-zero projection in $A$ is properly infinite in $C^*(A,\alpha)$.
\end{itemize}
In particular, if $A$ is purely infinite then $C^*(A,\alpha)$ is purely infinite and has the ideal property.
\end{prop}
\begin{proof}
Implications (iv)$\Leftrightarrow$(iii)$\Rightarrow$(ii)$\Rightarrow$(i) are general facts, see respectively \cite[Propositions 2.11]{Pas-Ror}, \cite[Proposition 4.7]{kr} and \cite[Theorem 4.16]{kr}. If $A$ is if real rank zero the equivalence (i)$\Leftrightarrow$(i') is ensured by Lemma \ref{lemma for other guys}. Thus it suffices to show that (i) implies (iii) or (iv). Let us then assume that every   element in $A^+\setminus\{0\}$ is properly infinite in $C^*(A,\alpha)$.

Suppose first that $A$ has the ideal property. We will show (iv). Let $\I$ be an ideal in $C^*(A,\alpha)$ and let $B$ be a non-zero hereditary $C^*$-subalgebra in the quotient $C^*(A,\alpha)/\I$.    Fix a non-zero positive element $b$ in $B$. Since $A^+$ residually supports elements of $C^*(A,\alpha)^+$ is non-zero positive element   $a$  in $q_\I(A)$ such that   $a \precsim b$. Note that  $a$ is properly infinite in $C^*(A,\alpha)/\I$ by \cite[Proposition 3.14]{kr}.  Since $A$ has the ideal property we can find  a projection $q\in A$ that belongs to the ideal in $A$  generated by the preimage of $a$ in $A$ but not to $I:=A\cap I$. Then $q+\I$ belongs to the ideal in $C^*(A,\alpha)/\I$ generated by $a$, whence $q+\I  \precsim a \precsim b$, by \cite[Proposition 3.5(ii)]{kr}. From the comment after \cite[Proposition 2.6]{kr} we can find $z\in C^*(A,\alpha)/\I$ such that $q+\I=z^*bz$. With $v:=b^{\frac{1}{2}}z$ it follows that $v^*v=q+\I$, whence $p:=vv^*=b^{\frac{1}{2}}zz^*b^{\frac{1}{2}}$ is a projection in $B$, which is equivalent to $q+\I$. By our assumption $q+\I$ and hence also $p$ is properly infinite.

Suppose now that $A$ is separable and there are finitely many, say $n$, invariant ideals in  $(A,\alpha)$.  By Lemma \ref{this lemma makes no sense} and Corollary \ref{corollary for dolary} they are in one-to-one correspondence with ideals in $C^*(A,\alpha)$. Hence by \cite[Proposition 2.11]{kr}, the conditions (ii) and (iii) are equivalent.  We will prove (ii). The proof goes by induction on $n$.

 Assume first that $n=2$ so that $C^*(A,\alpha)$ is simple. For any $b\in C^*(A,\alpha)^+\setminus\{0\}$ take $a\in A^+\setminus\{0\}$ such that $a\precsim b$. Then $b\in C^*(A,\alpha)aC^*(A,\alpha)=C^*(A,\alpha)$ and as $a$ is properly infinite we get $b\precsim a$ by \cite[Proposition 3.5]{kr}. Hence $b$ is properly infinite as it is Cuntz equivalent to $a$. Thus  $C^*(A,\alpha)$  purely infinite. 

Now suppose that our claim holds for any $k<n$. Let $\I$ be any non-trivial ideal in  $C^*(A,\alpha)$ and put $I=\I\cap A$. By Lemma \ref{complemented kernel invariance lemma} ii), the systems  $(A/I,\al_I)$ and $(I,\al|_{I})$ are reversible, and by Corollary \ref{corollary for dolary} and Proposition \ref{restrictions for reversible systems} we have $C^*(A/I,\al_I)\cong C^*(A,\alpha)/\I$ and  $ C^*(I,\al|_{I})\cong\I$. Clearly, the system $(A/I,\al_I)$ satisfies the assumptions of the assertion (a non-zero image of properly infinite element is properly infinite by \cite[Proposition 3.14]{kr}) and there are less than $n$ invariant ideals in $(A/I,\al_I)$. Hence $C^*(A/I,\al_I)$ is purely infinite. Similar, argument works for $C^*(I,\al|_{I})$; in particular note that if $a\precsim b$ for $b \in \I^+\setminus \{0\}$ and $a\in A^+\setminus \{0\}$, then $a\in I$. Also if $a\in I^+\setminus \{0\}$ is properly infinite in $C^*(A,\alpha)$, then it is properly infinite in $\I$, cf. \cite[Proposition 3.3]{kr}. Concluding, both  $\I$ and $C^*(A,\alpha)/\I$ are purely infinite, and since pure infiniteness is closed under extensions \cite[Theorem 4.19]{kr} we get that  $C^*(A,\alpha)$ is purely infinite.
\end{proof}
\begin{rem}\label{remark on strongly pure infiniteness}
We recall, see  \cite[Propositions 2.11, 2.14]{Pas-Ror}, that in the presence of the ideal property  pure infiniteness of a $C^*$-algebra is equivalent to strong pure infiniteness, weak  pure infiniteness, and many other notions of infiniteness appearing in the literature. Thus the list of equivalent conditions in Proposition \ref{pure infiniteness for reversible systems} can be considerably extended.
\end{rem}
\begin{rem}\label{funny remark}
In the case when there are finitely many invariant ideals in  $(A,\alpha)$,  we used separability of $A$ in the proof Proposition \ref{pure infiniteness for reversible systems} only to get the equivalence (ii)$\Leftrightarrow$(iii). Accordingly, in this case, the conditions (i) and  (ii) are equivalent  even for non-separable $C^*$-algebras.
\end{rem}
\begin{rem}
Certain properties  that imply condition (i) in Proposition \ref{pure infiniteness for reversible systems} were introduced in  \cite{KS}. In particular, Proposition \ref{pure infiniteness for reversible systems} can be readily used to obtain a  generalization of \cite[Theorem 8.22]{KS} so that it covers  not necessarily extendible systems  on $C^*$-algebras not necessarily possessing the ideal property. 
 \end{rem}

\section{Category of $C_0(X)$-algebras and $C_0(X)$-dynamical systems}\label{morphisms section}
In this section, we introduce morphisms of upper semicontinuous $C^*$-bundles which induce certain homomorphisms of $C_0(X)$-algebras.  We give several  characterizations  of such homomorphisms, and study basic properties of $C^*$-dynamical systems $(A,\alpha)$ where $A$ is a $C_0(X)$-algebra and $\alpha$ is induced by a morphism.  We show that the arising  category of $C_0(X)$-algebras has direct limits, and in some cases such limits exist in the subcategory of continuous $C_0(X)$-algebras.  
\subsection{Morphism of $C^*$-bundles and $C_0(X)$-dynamical systems}\label{subsection on C_0(X) dynamical systems}

Let $\A=\bigsqcup\limits_{x\in X} A(x)$ and $\B=\bigsqcup\limits_{y\in Y} B(y)$ be  upper semicontinuous $C^*$-bundles.
 We  wish to view  morphism between $C^*$-bundles as a  common generalization of proper mappings and $C^*$-homomorphisms. Mimicking   the definition of  morphisms of vector bundles, one can imagine such a morphism as a pair of continuous mappings $\al:\B\to \A$ and $\p:X \to Y$ such that the following diagram 
$$
\xymatrix{  \B   \ar[d]_{p}  \ar[r]^\al & \A \ar[d]^p& \\
  Y     & \ar[l]^{\p}   X   
 }
$$
commutes and   for each $x\in X$, $\al:B({\p(x)})\to A(x)$ is a homomorphism. Since some of these homomorphisms might be zero  we will allow $\p$ to be defined on an open subset $\Delta$ of $X$. 
\begin{defn}\label{morphism definition}
A \emph{morphism} (of upper semicontinuous $C^*$-bundles) from $\B$ to $\A$ is a pair $(\p,\{\al_{x}\}_{x\in \Delta})$ consisting of 
\begin{itemize}
\item[1)] a continuous proper map $\p:\Delta \to Y$ defined on an open set $\Delta\subseteq X$, and
\item[2)] a continuous bundle of homomorphisms $\{\al_{x}\}_{x\in \Delta}$ between the corresponding fibers, i.e.:
\begin{itemize}
\item[a)] for each $x\in \Delta$, $\al_x:B(\p(x))\to A(x)$ is a homomorphism; 
\item[b)] if $\{x_i\}_{i\in \Lambda}\subseteq \Delta$ and $\{b_i\}_{i\in \Lambda} \subseteq \B$ are nets such that $x_i\to x\in \Delta$,  $b_i \to b$ and $p(b_i)=\p(x_i)$, for  $i\in \Lambda$, then $\al_{x_i}(b_i)\to  \alpha_x(b)$.
\end{itemize}
\end{itemize}
\end{defn}
The above definition is born to work well with section algebras. 
\begin{prop}\label{characterization of induced endomorphisms}
Let $\p: \Delta \to Y$ be a proper continuous mapping where $\Delta\subseteq X$ is an open set. For each $x\in \Delta$ let  $\al_x:B(\p(x))\to A(x)$ be a homomorphism. The pair $(\p,\{\al_{x}\}_{x\in \Delta})$ is a  morphism from $\B$ to $\A$ if and only if   the formula 
\begin{equation}\label{endomorphism induced by a morphism}
\al(b)(x)=\begin{cases}
\al_x(b(\p(x)),  & x \in \Delta,
\\
0_x        & x \notin \Delta,
\end{cases}
\qquad b \in \Gamma_0(\B), \,\, x \in X,
\end{equation}  
 yields a well defined  homomorphism $\al:\Gamma_0(\B)\to \Gamma_0(\A)$  between the section $C^*$-algebras. 
\end{prop}

\begin{proof}
Suppose that $(\p,\{\al_{x}\}_{x\in \Delta})$ is a  morphism. Clearly,     it suffices to show that  the map \eqref{endomorphism induced by a morphism} is well defined, equivalently, that  for any $b\in \Gamma_0(\B)$ the mapping
\begin{equation}\label{the map to be shown to be continuous}
X \ni x \longmapsto \al(b)(x)\in  A(x) \subseteq \A
\end{equation}
is in $\Gamma_0(\A)$. Condition 2b) from Definition  \ref{morphism definition} readily implies that the map $\Delta \ni x \longmapsto \al(b)(x)\in  A(x) \subseteq \A$ is continuous  (consider elements $b_i:=b(\p(x_i))$). In  particular, $\Delta \ni x \longmapsto \|\al(b)(x)\|\in \R$ is upper semicontinuous. Thus for any $\varepsilon >0$  the set $\{x\in X: \|\al(b)(x)\| \geq \varepsilon\}=\{x\in \Delta: \|\al(b)(x)\| \geq \varepsilon\}$ is closed. Actually it is compact because 
$$
\{x\in X: \|\al(b)(x)\| \geq \varepsilon\} \subseteq  \{x\in \Delta: \|b(\p(x))\| \geq \varepsilon\}
$$
and  the latter set is compact as $\p$ is proper and $b$ vanishes at infinity. Thus the map \eqref{the map to be shown to be continuous} is vanishing at infinity. To conclude that $\alpha(b)\in \Gamma_0(A)$ we need to show that $\alpha(b)$ is continuous on the boundary $\partial \Delta$ of $\Delta$. But if $\{x_i\}_{i\in \Lambda}\subseteq X$ is a net convergent to $x_0\in \partial  \Delta$, then for every $\varepsilon >0$  the point $x_0$ belongs to the open set $\{x\in X: \|\al(b)(x)\| < \varepsilon\}$ and  hence $\alpha(b)(x_i)$ converges to $0$ by Lemma \ref{topology on bundles lemma} (consider $b_i=\alpha(b)(x_i)$ and $a\equiv 0$).

Conversely,  assume that   $\al:\Gamma_0(\B)\to \Gamma_0(\A)$ is a homomorphism satisfying  \eqref{endomorphism induced by a morphism}. We need to show condition 2b) in Definition  \ref{morphism definition}. Let  $\{x_i\}_{i\in \Lambda}\subseteq \Delta $ and $\{b_i\}_{i\in \Lambda} \subseteq \B$ be nets such that $x_i\to x \in \Delta$,  $b_i \to b$ and $p(b_i)=\p(x_i)$. Take arbitrary $\varepsilon >0$. By Lemma \ref{topology on bundles lemma}  there is $a\in \Gamma_0(\A)$ such that 
$
\|a(p(b))-b\|< \varepsilon $ and we  eventually have $
\|a(\p(x_i))-b_i\|< \varepsilon $. This implies that  $\|\alpha(a)(x)-\alpha_x(b)\|< \varepsilon$ and  we eventually have 
  $
\|\alpha(a)(x_i)-\alpha_{x_i}(b_i)\|< \varepsilon$.
 Since $\alpha(a)\in \Gamma_0(\A)$ we have $\alpha_{x_i}(b_i)\to \alpha_x(b)$ by Lemma \ref{topology on bundles lemma}. 
\end{proof}

\begin{defn}
To indicate that  a homomorphism $\al:\Gamma_0(\B)\to \Gamma_0(\A)$ is given by \eqref{endomorphism induced by a morphism} for a certain morphism $(\p,\{\al_{x}\}_{x\in \Delta})$ of upper semicontinuous $C^*$-bundles we will  say that $\alpha$ \emph{is induced by a morphism}.
\end{defn}

Let $A=\Gamma_0(\A)$ and  $B=\Gamma_0(\B)$. Note that for an induced homomorphism  $\alpha:B\to A$  the underlying mapping $\p:\Delta\to Y$ is uniquely determined by $\alpha$  on the set 
\begin{equation}\label{Delta zero set}
\Delta_0:=\{x\in X: \al_x\neq 0\}\subseteq \Delta,
\end{equation}
which coincides with $\Delta$  when all  endomorphisms $\alpha_x$, $x\in \Delta$, are non-zero. Sometimes we can assume that $\Delta=\Delta_0$ using the following lemma.
\begin{lem}\label{lemma to be applied}
Let $\alpha:B\to A$ be a homomorphism  induced by a morphism $(\p,\{\al_{x}\}_{x\in \Delta})$ from a   continuous $C_0(Y)$-algebra $B$ to a    continuous $C_0(X)$-algebra $A$ and let $\Delta_0$ be given by \eqref{Delta zero set}. Suppose  also that $B(\p(x))\neq \{0\}$, for $x\in \Delta$,  and that every $\al_x$, $x\in \Delta_0$, is injective. Then $\Delta_0$ is a clopen  in $\Delta$ and  in particular 
$(\p|_{\Delta_0},\{\al_{x}\}_{x\in \Delta_0})$ is a morphism that induces $\alpha$.
\end{lem} 
\begin{proof}
Since $\Delta_0=\bigcup_{b\in B} \{x\in X: \|\al(b)(x)\|>0 \}$ and $A$ is a continuous $C_0(X)$-algebra, we see that set $\Delta_0$ is open. 
Suppose that $x_0$ is a point in the boundary of $\Delta_0$ in $\Delta$. Take a net $\{x_i\}_i\subseteq \Delta_0$ converging to $x_0$ and an element $b\in B$ such that $\|b(\p(x_0))\|=1$. Since the homomorphism $\al_{x_i}$ are isometric, and the mappings $\Delta\ni x\mapsto \|\al_{x}(b(\p(x)))\|$ and $\Delta\ni x\mapsto \|b(\p(x))\|$ are continuous, we get
$$
\|\al_{x_0}(b(\p(x_0)))\|=\lim_i  \|\al_{x_i}(b(\p(x_i)))\|=\lim_i  \|b(\p(x_i))\|=1.
$$ 
Hence $\al_{x_0}\neq 0$, that is  $x_0\in \Delta_0$. Thus $\Delta_0$ is closed in $\Delta$ and therefore $\p|_{\Delta_0}:\Delta_0\to X$ is a proper map. Clearly, $\alpha$ satisfies \eqref{endomorphism induced by a morphism} with $\Delta_0$ in place of $\Delta$. Accordingly, $\alpha$ is induced by the morphism $(\p|_{\Delta_0},\{\al_{x}\}_{x\in \Delta_0})$ by Proposition \ref{characterization of induced endomorphisms}.
\end{proof}

 We have the following characterizations of homomorphism induced by morphisms phrased in terms of $C_0(X)$-algebras.
\begin{prop}\label{morphism induction criteria}
Let $A$ be  $C_0(X)$-algebra and $B$  a  $C_0(Y)$-algebra. For any homomorphism $\al:B\to A$ the following conditions are equivalent: 
\begin{itemize}
\item[(i)] $\al$ is induced by a morphism from $\B=\bigsqcup\limits_{y\in Y} B(y)$ to $\A=\bigsqcup\limits_{x\in X} A(x)$,
\item[(ii)] there is a homomorphism $\Phi:C_0(Y)\to C_0(X)$ such that 
$$
\al(f \cdot b)= \Phi(f)\cdot \al(b), \qquad  f\in C_0(Y),\,\, b\in B. 
$$
\end{itemize}
If additionally $B$ is unital and $A$ is  a continuous $C_0(X)$-algebra  then the above  conditions are equivalent to the following one:
\begin{itemize}
\item[(iii)] $\al$ maps $C_0(Y)$ `almost into' $C_0(X)$, that is
$$
\al(C_0(Y)\cdot 1)\subseteq C_0(X)\cdot \al(1).
$$
\end{itemize}
If the additional assumptions and condition (iii) are satisfied, then the corresponding morphism $(\p,\{\al_{x}\}_{x\in \Delta})$ can be chosen so that  $\Delta$ is compact and each $\alpha_x$, $x\in \Delta$, is non-zero.
\end{prop}
\begin{proof} (i) $\Rightarrow$ (ii). It suffices to put $\Phi(a):=a\circ \p$ for $a\in C_0(Y)$. 

 (ii) $\Rightarrow$ (i). Note that  $\Phi:C_0(Y) \to C_0(X)$ is given by the formula 
\begin{equation}\label{composition homomorphism formula}
\Phi(b)(x)=\begin{cases}b(\p(x)),  & x \in \Delta,
\\
0        & x \notin \Delta, 
\end{cases} \qquad b\in C_0(Y)
\end{equation}
where  $\p:\Delta \to Y$ is  a continuous proper mapping   defined on an open set $\Delta\subseteq X$.
Let $x \in \Delta$. We  define a homomorphism $\al_x:B(\p(x))\to A(x)$ as follows. For any $b_0\in B(\p(x))$  there is  $b\in B$ such that  $b(\p(x))=b_0$, and we claim that the element
\begin{equation}\label{definition of an endomorphism}
\al_x(b_0):=\al(b)(x)
\end{equation}
is  well defined (does not depend on the choice of $b$). Indeed,  let $\b, b\in B$ be  such that $\b(\p(x))=b(\p(x))=b_0$. Then $ b(\p(x))-\b(\p(x))=0$. Upper semicontinuity of the $C^*$-bundle $\B=\bigsqcup\limits_{y\in Y} B(y)$ imply that for every $\varepsilon>0$ there is an open neighbourhood $U$ of $\p(x)$ such that 
$$
\|b(y) - \b(y)\| < \varepsilon, \qquad \textrm{ for all } y\in U.
$$ 
Let us choose a function $h \in C_0(Y)$ such that $h(\p(x))=1$, $0\leq h\leq 1$ and $h(y)=0$ outside $U$. We get 
\begin{align*}
\|\al(b)(x)-\al(\b)(x)\|&=\|\big(\Phi(h)\al(b)-\Phi(h)\al(\b)\big)(x)\|
=\|\al(hb -h\b)(x)\|
\\
&\leq \|\al(hb -h\b)\|\leq\|hb -h\b\| \leq \varepsilon.
\end{align*}
This proves our claim. Now it is straightforward to see that \eqref{definition of an endomorphism} gives the desired homomorphism $\al_x:B(\p(x))\to A(x)$. Moreover,  for the above defined  pair $(\p,\{\al_{x}\}_{x\in \Delta})$ the formula \eqref{endomorphism induced by a morphism} holds. Hence in view of Proposition \ref{characterization of induced endomorphisms},  $\al$ is induced by a morphism.

Let us now assume that $B$ is  a unital  and $A$ is  a continuous $C_0(X)$-algebra.

(ii) $\Rightarrow$ (iii). It is  obvious. 

(iii)$\Rightarrow$ (ii).  Since $
\al(C_0(Y)\cdot 1)\subseteq C_0(X)\cdot \al(1)
$, for every $f\in C_0(Y)$ there exists $g\in C_0(X)$ such that  
 $$
 \al(f\cdot 1)(x)= g(x) \al(1)(x),\qquad x\in X.
 $$
Clearly, the function  $g$ is uniquely determined  by $f$ on the  set $\Delta:=\{x\in X: \al(B)(x)\neq 0\}=\{x\in X: \al(1)(x)\neq 0\}$. Since the mapping $X\ni x \to  \|\alpha(1)(x)\|\in \{0,1\}$ is continuous and vanishing at infinity, $\Delta$ is open and compact. Now it is straightforward to see that  the formula
   $
   \Phi(f)=g|_\Delta
   $ 
defines a  homomorphism  $\Phi:C_0(Y)\to C(\Delta)\subseteq C_0(X)$ satisfying condition (ii). 
\end{proof}

\begin{ex}\label{example quotient} Suppose that $q_I:A\to A/I$ is a quotient map and $A$ is a $C_0(X)$-algebra.  We may treat   $A/I$ as a  $C_0(V)$-algebra for any closed set $V$ containing $\sigma_A(\Prim(A/I))$, cf. Lemma \ref{lemma on ideals in C_0(X)-algebras}. Then  we have 
$$
q_I(f\cdot a)= f|_{V}\cdot q_I(a),\qquad f \in C_0(X), \,\, a\in A.
$$
Hence condition (ii) in Proposition \ref{morphism induction criteria} is satisfied. In particular, $q_I$ is induced by the morphism $(id,\{q_{I,x}\}_{x\in V})$ where 
$q_{I,x}:A(x)\to A(x)/I(x)$, $x\in V$, are the quotient maps.
\end{ex}
Let us consider a category of $C_0(X)$-algebras with morphisms being homomorphisms satisfying the equivalent conditions in Proposition \ref{morphism induction criteria}. In this paper, we are interested in properties of systems $(A,\alpha)$ where $A$ is an object and $\alpha$ is a morphism in this category.
\begin{defn}
We say that a $C^*$-dynamical system $(A,\alpha)$ is a \emph{$C_0(X)$-dynamical system}, if $A$ is a $C_0(X)$-algebra and $\alpha$ is induced by a morphism. If additionally $A$ is a continuous $C_0(X)$-algebra, we say that  $(A,\alpha)$ is a \emph{continuous $C_0(X)$-dynamical system}
\end{defn}
In section \ref{applications section}, we will study crossed products associated to continuous $C_0(X)$-dynamical systems introduced
in the following example.
\begin{ex}[Endomorphisms of $C^*$-algebras with Hausdorff primitive ideal space]\label{primitive Hausdorff example}
If  $A$ is a $C^*$-algebra and its primitive ideal space  $X:=\Prim(A)$ is Hausdorff, then using Dauns-Hofmann isomorphism we may naturally  treat $A$ as a continuous $C_0(X)$-algebra where the structure map $\sigma_A$ is identity, cf. \cite[2.2.2]{Blan-Kirch}. In particular, for $x\in X=\Prim(A)$  the fiber $A(x)=A/x$ is a simple (non-zero)  $C^*$-algebra. Thus if $\alpha:A\to A$ is an endomorphism induced by a morphism $(\p,\{\al_{x}\}_{x\in \Delta})$, then Lemma \ref{lemma to be applied} applies and we may assume that each  $\alpha_x$, $x\in \Delta$, is injective.  Moreover, if $A$ is unital then we may identify $C(X)$ with $Z(A)$ and by Proposition \ref{morphism induction criteria} an endomorphism $\al:A\to A$ is   induced by a morphism if and only if 
$
\al(Z(A))\subseteq Z(A)\alpha(1).
$
\end{ex}

For trivial $C^*$-bundles we have the following  description of endomorphisms induced by morphisms. 
 We equip  the set $\End (D)$ of all  endomorphisms of a $C^*$-algebra $D$ with the  topology of point-wise convergence. 
\begin{prop}\label{proposition for trivial bundles} Let    $\p: \Delta \to X$ be a proper continuous mapping defined on an open set $\Delta\subseteq X$ and let a continuous mapping $
\Delta \ni x \longrightarrow \al_{x}\in  \End (D)$ where  $D$ is a $C^*$-algebra. We treat $A:=C_0(X,D)$ as a  $C_0(X)$-algebra in an obvious way. The formula 
\begin{equation}\label{endomorphism on C_0(X,A) }
\al(a)(x):=\begin{cases}
\al_x(a(\p(x)),  & x \in \Delta,
\\
0        & x \notin \Delta,
\end{cases}
\qquad a \in C_0(X,D), \,\, x \in X.
\end{equation} 
defines an endomorphism of $A$ induced by a morphism. Every endomorphism of $A$ induced by a morphism arises in this way.  If $D$ is simple, or if $A$ is unital, then we can choose the corresponding morphism in such a way that each $\alpha_x$, $x\in \Delta$, is non-zero.
\end{prop}
\begin{proof}
The corresponding $C^*$-bundle $\A= \bigsqcup\limits_{x\in X} D$ can be identified with the product $X\times D$, together with its product topology. In this  case  condition 2b) from Definition  \ref{morphism definition} translates to the following: If  $\{x_i\}_{i\in \Lambda}\subseteq \Delta$ and $\{b_i\}_{i\in \Lambda} \subseteq D$ are nets such that $x_i\to x\in \Delta$ and  $b_i \to b\in D$, then $\al_{x_i}(b_i)\to  \alpha_x(b)$. The latter condition is equivalent to the continuity of the map $\Delta \ni x \longrightarrow \al_{x}\in  \End (D)$, which can be readily deduced from  the inequality:
$$
\|\al_{x_i}(b_i) -\alpha_x(b)\|\leq \|b_i - b\|+ \|\al_{x_i}(b) -\alpha_{x}(b)\|.
$$
Thus the  assertion  follows by Proposition \ref{characterization of induced endomorphisms}. The last remark follows by Lemma \ref{lemma to be applied} and the last part of  Proposition \ref{characterization of induced endomorphisms}.
\end{proof}
\subsection{Quotients and restrictions of $C_0(X)$-dynamical systems}\label{subsection quotient dynamical system}
Restrictions and  quotients of  $C_0(X)$-dynamical systems can be treated as $C_0(X)$-dynamical systems in the following sense.
 
\begin{prop}\label{taki tam sobie lemma}
Suppose that $\alpha:A\to A$ is an endomorphism induced by a morphism $(\p,\{\alpha_x\}_{x\in \Delta})$.  Let $I$ be a positively invariant ideal  in $(A,\alpha)$. Then  $(I,\alpha|_{I})$ and  $(A/I,\alpha_I)$ are naturally $C_0(X)$-dynamical systems where   $\alpha|_{I}$ is induced by $(\p,\{\al_{x}|_{I(\p(x))}\}_{x\in  \Delta})$ and  $\alpha_{I}$ is induced by $(\p,\{\alpha_{I,x}\}_{x\in  \Delta})$ where
\begin{equation}\label{quotient field of endomorphisms}
\alpha_{I,x}\big(a +I(\p(x))\big):= \alpha_x(a)+ I(x), \qquad a\in A(\p(x)),\,\, x\in  \Delta.
\end{equation}
\end{prop}

\begin{proof}
Note that positive invariance of $I$ implies that $\alpha_x(I(\p(x)))\subseteq I(x)$, $x\in \Delta$. In particular, \eqref{quotient field of endomorphisms} gives a  well defined homomorphism $\alpha_{I,x}$.  Now,   Proposition \ref{characterization of induced endomorphisms} readily implies that $(\p,\{\al_{x}|_{I(\p(x))}\}_{x\in  \Delta})$ is a morphism that induces $\alpha|_{I}$  and  that $(\p,\{\alpha_{I,x}\}_{x\in  \Delta})$ is a  morphism that induces $\alpha^I$ (cf. the description of the quotient $C^*$-bundle in Lemma  \ref{lemma on ideals in C_0(X)-algebras}).
\end{proof}
We note that even when the structure map $\mu_A: C_0(X)\to Z(M(A))$ is injective, the structure maps for $I$ and $A/I$ treated as $C_0(X)$-algebras as in the above proposition, will hardly ever be injective. Moreover, $A/I$ might not be a continuous  $C_0(X)$-algebra, even if $A$ is, cf. Lemma \ref{lemma on ideals in C_0(X)-algebras}. In certain situations,  these problems can be circumvented  by using the following  proposition.
\begin{prop}\label{taki tam sobie lemma2}
Suppose that $\alpha:A\to A$ is an endomorphism induced by a morphism $(\p,\{\alpha_x\}_{x\in \Delta})$ and $I$ is positively invariant ideal  in $(A,\alpha)$. Put
$$
V:=\overline{\sigma_A(\Prim(A/I))},\qquad U:=\sigma_A(\Prim(I))
$$
and treat  $A/I$ as a $C_0(V)$-algebra and $I$ as a $C_0(U)$-algebra.
\begin{itemize}
\item[(i)] If $\p(V\cap \Delta)\subseteq V$, which is automatic when for each $x\in \Delta$  the  range of $\alpha_x$ is a full subalgebra of $A(x)$, then 
the quotient endomorphism  $\alpha_I:A/I \to A/I$ is induced by the morphism $(\p|_{V\cap \Delta}, \{\alpha_{I,x}\}_{x\in V\cap \Delta})$, cf. \eqref{quotient field of endomorphisms}.  

\item[(ii)] If $U$ is open and $\p^{-1}(U)\subseteq U$, which is automatic when $A$ is a continuous $C_0(X)$-algebra and  each $\alpha_x$,  $x\in \Delta$, is injective, then the restricted $C^*$-dynamical system $(I,\alpha|_{I})$ is a naturally induced by a morphism 
$(\p|_{\p^{-1}(U)},\{\al_{x}|_{I(\p(x))}\}_{x\in  \p^{-1}(U)})$.
\end{itemize}
\end{prop}

\begin{proof}

(i). Suppose  that $\p(V\cap \Delta)\subseteq V$. Then the restriction $\p|_V:V\cap \Delta \to V$ is a proper map and $A$ can be naturally treated as a $C_0(V)$-algebra by Lemma  \ref{lemma on ideals in C_0(X)-algebras}. Hence the morphism $(\p,\{\alpha_{I,x}\}_{x\in  \Delta})$ from  Proposition \ref{taki tam sobie lemma} restricts to a morphism $(\p|_{V\cap \Delta}, \{\alpha_{I,x}\}_{x\in V\cap \Delta})$ that induces  $\alpha_I$. 

Now,  we show that $\p(V\cap \Delta)\subseteq V$, if for each $x\in \Delta$  the  range of $\alpha_x$ is a full subalgebra of $A(x)$. To this end, let 
$V_0=\sigma_A(\Prim(A/I))$ and recall that $x\in V_0$ if and only if $I(x)\neq A(x)$, see \eqref{J-quotient fibers}.  Let $x\in \Delta\cap V_0$. We claim  that $\p(x)\in V_0$. Indeed, assume on the contrary that  $I(\p(x))= A(\p(x))$. Then by positive invariance of $I$ we have  $\alpha_x(A(\p(x)))=\alpha_x(I(\p(x))\subseteq I(x)$. Since $\alpha_x(A(\p(x))$ is full in $A(x)$ we get 
$$
I(x)=A(x)I(x)A(x)\supseteq A(x)\alpha_x(A(\p(x)))A(x)=A(x).
$$
This  contradicts the fact that $x\in Y_0$, cf. \eqref{J-quotient fibers}.  
Accordingly,  $\p(V_0\cap \Delta)\subseteq V_0$. By continuity of $\p$ we get $\p(V\cap \Delta)\subseteq V$.

(ii). If $U$ is open and $\p^{-1}(U)\subseteq U$, then  $\p|_{\p^{-1}(U)}:\p^{-1}(U)\to U$ is a  proper map and $I$ is naturally a $C_0(U)$-algebra. Since $U=\{x\in X:  I(x)\neq \{0\}\}$  by \eqref{J-ideal fibers},  for any $a\in I$ and $x\notin\p^{-1}(U)$ we have $\alpha(a)(x)=0$. Thus $(\p|_{\p^{-1}(U)},\{\al_{x}|_{I(\p(x))}\}_{x\in  \p^{-1}(U)})$ is a morphism that induces   $(I,\alpha|_{I})$,  by  Proposition \ref{morphism induction criteria}. 

Now, suppose that   $A$ is a continuous $C_0(X)$-algebra and  each $\alpha_x$,  $x\in \Delta$, is injective. Then $U$ is open. By \eqref{J-ideal fibers}, for each $x\in \p^{-1}(U)$ we may find  an element $a\in I$ such that $a(\p(x))\neq 0$. Since $\alpha_x$ is injective, we have  $\alpha(a)(x)=\alpha_x(a(\p(x)))\neq 0$, and thus $x\in U$, again by \eqref{J-ideal fibers}. Hence $\p^{-1}(U)\subseteq U$.
\end{proof}

The above proposition implies that the quotient and the restriction of   continuous $C_0(X)$-dynamical systems described in Example \ref{primitive Hausdorff example} are naturally  $C^*$-dynamical systems of the same type, see Lemma \ref{lemma for restriction and quotient} below.

\subsection{Extendible morphisms and reversible $C_0(X)$-dynamical systems}
In the foregoing lemma we use the description of  multiplier algebras given in Proposition \ref{multiplier description}.
\begin{lem}\label{characterization of extendible induced endomorphisms}
Suppose that $A$ is a $C_0(X)$-algebra, $B$ is a $C_0(Y)$-algebra, and  $\al:B\to A$ is an extendible  homomorphism induced by a morphism $(\p,\{\al_{x}\}_{x\in \Delta})$. Each $\al_x:B(\p(x))\to A(x)$, $x\in \Delta$,  is extendible and  $\overline{\alpha}:M(B)\to M(A)$ is given by 
\begin{equation}\label{extended morphism2}
 \overline{\alpha}(m)(x)= \begin{cases}
\overline{\alpha}_x\big(m(\varphi(x))\big),  & x \in \Delta,
\\
0_x,  & x \notin \Delta,
\end{cases} \qquad m\in M(B),\,\, x\in X.
\end{equation}
If additionally  either $A$ has local units or  $A$  is locally trivial 
 then the set $\Delta_0$ defined in \eqref{Delta zero set}
 is closed in $X$. 
\end{lem}
\begin{proof}
Let $\{\mu_\lambda\}$ be an approximate unit in $B$ and let $x\in \Delta$. Then $\{\mu_\lambda(\varphi(x))\}$ is an approximate unit in $B(\varphi(x))$ and  $\alpha_x({\mu_\lambda}(\varphi(x)))=\alpha(\mu_\lambda)(x)$ converges strictly in $A(x)$. Hence the homomorphisms $\alpha_x$, $x\in \Delta$, are extendible. Recall that  $\overline{\alpha}$ is determined by the formula $\overline{\alpha}(m)a=\lim_\lambda \alpha(m \mu_\lambda)a$ where $a\in A$, $m\in M(B)$. It follows that for any $x\in \Delta$ we have
\begin{align*}
(\overline{\alpha}(m)a)(x)&=\lim_\lambda \alpha(m \mu_\lambda)(x)a(x)=\lim_\lambda \alpha_x\big((m \mu_\lambda)(\varphi(x))\big)a(x)
\\
&=\lim_\lambda \alpha_x\big(m(\varphi(x))\mu_\lambda(\varphi(x))\big)a(x)= \overline{\alpha}_x\big(m(\varphi(x))\big)a(x).
\end{align*}
Thus we get \eqref{extended morphism2}. 
\\
Now suppose that  $x_0$ is a point in the boundary of the set $\Delta_0=\{x\in X: \al(B)(x)\neq 0\}
$, but $x_0\notin \Delta_0$. If $A$ has local units, or if $A$  is locally trivial,  then we may choose  $a\in A$,  such that $\|a(x)\|=1$ for  every $x$ in an open neighbourhood $U$ of $x_0$. Then the compact set 
$\{x\in X: \|(\overline{\alpha}(1)a) (x)\|\geq  1/2\}$ contains $\Delta_0\cap U$ but does not contain $x_0$. This leads to a contradiction, since  $x_0$ is in the closure of  $\Delta_0\cap U$.
\end{proof}

Suppose now that $(A,\alpha)$ is a reversible $C^*$-dynamical system  where $A$ is a $C_0(X)$-algebra and $\alpha$ is induced by a morphism $(\p,\{\al_{x}\}_{x\in \Delta})$. 
In general all we can say about the kernel  of $\alpha$ and its annihilator  is that 
$$
\ker\al=\{a\in A: a(y)\in \bigcap_{x\in \p^{-1}(y)}\ker\al_x \textrm{ for all } y\in \p(\Delta)\}
$$
and $(\ker\al)^{\bot}$ is contained in
\begin{equation}\label{estimate set}
 \{ a\in A: \,  a|_{X\setminus \p(\Delta)}=0 \textrm{ and } 
  a(y)\in \Big(\bigcap_{x\in \p^{-1}(y)}\ker\al_x \Big)^{\bot}\textrm{ for } y\in \p(\Delta)\}.  
\end{equation}
Nevertheless, we have the following statement.
\begin{prop}\label{transfer induced by a transfer morphism}
Suppose that $(A,\alpha)$ is a reversible $C^*$-dynamical system where  $A$ is a $C_0(X)$-algebra and $\alpha$ is induced by a morphism $(\p,\{\al_{x}\}_{x\in \Delta})$.
\begin{itemize}
\item[(i)] If all of the endomorphisms $\alpha_x$, $x\in \Delta$, are injective then  $\varphi$ is injective.
\item[(ii)] If $\varphi$ is injective and $(A,\alpha)$ is extendible then  the unique regular transfer operator for $(A,\alpha)$ is determined by the formula 
\begin{equation}\label{transfer induced by a morphism}
\alpha_*(a)(x)= \begin{cases}
\alpha_{*,x}\big(a(\varphi^{-1}(x))\big),  & x \in \varphi(\Delta),
\\
0_x,  & x \notin \varphi(\Delta),
\end{cases} \qquad a\in A,\,\, x\in X,
\end{equation}
where for each $x\in \varphi(\Delta)$,  $\alpha_{*,x}:A(\varphi^{-1}(x))\to A(x)$ is 
a completely positive generalized inverse to $\alpha_{\varphi^{-1}(x)}:A(x)\to A(\varphi^{-1}(x))$. The mappings $\alpha_{*,x}$ have strictly continuous extensions $\overline{\alpha}_{*,x}$ and the strictly continuous extension $\overline{\alpha}_*$ of $\alpha_*$ is given by the formula
$$
\overline{\alpha}_*(a)(x)= \begin{cases}
\overline{\alpha}_{*,x}(a(\varphi^{-1}(x))),  & x \in \varphi(\Delta),
\\
0_x,  & x \notin \varphi(\Delta),
\end{cases} \qquad a\in M(A),\,\, x\in X,
$$
where we use the description of  multipliers given in Proposition \ref{multiplier description}.
\end{itemize}
\end{prop}
\begin{proof} (i). Injectivity of $\alpha_x$'s imply that  $
\ker\al=\{a\in A: a(x)=0\textrm{ for all } x\in \p(\Delta)\}
$ and therefore $
(\ker\al)^\bot\subseteq \{a\in A: a(x)=0\textrm{ for all } x\notin \p(\Delta)\}
$. Let $x, y\in \Delta$ be two different points. Take any $b\in \alpha(A)$ such that $b(x)\neq 0$ and any $h\in C_0(X)$ such that $h(x)=1$ and $h(y)=0$. Since 
$\alpha(A)=\alpha(A)A\alpha(A)$ we see that $c:=hb$ is in $\alpha(A)$. 
Obviously, $c(x)\neq 0$ and $c(y)=0$. Thus for  any $a\in  
A$ with $\alpha(a)=c$ we  have $\alpha_x(a(\varphi(x))\neq 0$ and $\alpha_y(a(\varphi(y))=0$. Injectivity of $\alpha_x$ and $\alpha_y$ implies that  $\varphi(x)\neq \varphi(y)$.  Hence $\varphi$ is injective.
\\
(ii).  Fix $x\in \Delta$ and $b_0 \in \overline{\alpha}_x(1_{\varphi(x)})A(x)\overline{\alpha}_x(1_{\varphi(x)})$ (we use the notation of Lemma \ref{characterization of extendible induced endomorphisms}). Take $b\in \alpha(A)=\overline{\alpha}(1)A\overline{\alpha}(1)$ such that $b(x)=b_0$. Let  $a\in (\ker\alpha)^\bot$ be the unique element such that $\alpha(a)=b$. Accordingly, $b_0=b(x)=\al_x(a(\varphi(x)))$ where $a(\varphi(x))\in (\ker\alpha_x)^\bot$ (here we use injectivity of $\varphi$ and  that $(\ker\alpha)^\bot$ is contained in the set \eqref{estimate set}).  It follows that the range of $\al_x:A(\p(x))\to A(x)$ is the corner $\overline{\alpha}_x(1_{\varphi(x)})A(x)\overline{\alpha}_x(1_{\varphi(x)})$ and  $\al_x:(\ker\alpha_x)^\bot\to \al_x(A(\p(x)))$ is an isomorphism. The latter fact implies that  $\ker\alpha_x$ is a complemented ideal in $A(\p(x))$. We  define the map
\begin{equation}\label{formula to invoke}
\alpha_{*,\varphi(x)}(a):=\alpha_{x}^{-1}\Big(\overline{\alpha}_x(1_{\varphi(x)}) a\overline{\alpha}_x(1_{\varphi(x)})\Big), \qquad a\in A(x),
\end{equation}
where $\alpha_{x}^{-1}$ is the inverse to the isomorphism 
$\al_x:(\ker\alpha_x)^\bot\to \overline{\alpha}_x(1_{\varphi(x)})A(x)\overline{\alpha}_x(1_{\varphi(x)})$. The maps $\alpha_{*,x}$ have  strictly continuous extensions which are given by \eqref{formula to invoke} with  $\alpha_{x}^{-1}$  replaced by the inverse to the strictly continuous isomorphism $\overline{\al}_x:M((\ker\alpha_x)^\bot)\to \overline{\alpha}_x(1_{\varphi(x)})M(A(x))\overline{\alpha}_x(1_{\varphi(x)})$, cf. Lemma \ref{characterization of extendible induced endomorphisms}. Now it is immediate to see that  the homomorphisms $\alpha_{*,x}$ and $\overline{\alpha}_{*,x}$, $x\in \varphi(\Delta)$,  fulfill the requirements of the assertion. 
\end{proof}
Injectivity of the map $\p$ in the second part of the above proposition is essential.
\begin{ex}\label{transfer not induced}
 Consider a reversible $C^*$-dynamical system  $(A,\alpha)$ where $A=\C^3$ and  $\alpha(a)=(a_2,0,a_3)$ for $a=(a_1,a_2,a_3)\in A$. Then the regular transfer operator $\alpha_*(a)=(0,a_1,a_3)$ for $(A,\alpha)$  is actually an endomorphism.  Treating $A$ as a $C(\{1,2\})$-algebra where  $a(1)=a_1\in \C$ and $a(2)=(a_2,a_3)\in \C^2$,  for $a\in A$, the endomorphism $\alpha$ is induced by the morphism $(\p,\{\alpha_1,\alpha_2\})$ where
 $$
 \p(1)=\p(2)=2, \qquad \alpha_1(a_2,a_3)=a_2, \,\,\, \alpha_2(a_2,a_3)=(0,a_3).
 $$
But $\alpha_*$ is  not induced by a morphism because the fiber $\alpha_*(a)(2)=(a_1,a_3)$ of $\alpha_*(a)$ depends on two fibers of $a$.
\end{ex}
\subsection{Direct limits} In this subsection, we show that a direct limit  of $C_0(X_n)$-algebras is naturally a $C_0(\X)$-algebra,  and  in certain cases  a continuous $C_0(\X)$-algebra. We will use this result, in subsection \ref{subsection 42}, to show that a reversible extension of a  $C_0(X)$-dynamical system is a $C_0(\X)$-dynamical system.

Let us consider  a direct sequence $B_0\stackrel{\al_{0}}{\longrightarrow} B_{1}  \stackrel{\al_{1}}{\longrightarrow}...
$, where for each $n\in \N$, $B_n$ is a $C_0(X_n)$-algebra and $\alpha_n:B_n\to B_{n+1}$ is a homomorphism induced by a morphism $(\p_n,\{\alpha_{x}\}_{x\in X_{n+1}})$ (we may assume that the sets $X_n$ are disjoint, so a homomorphism $\alpha_x$  is uniquely determined by the point $x\in X_{n+1}$).
We show that the $C^*$-algebraic direct limit $B:=\underrightarrow{\lim\,\,}\{B_{n},\al_{n}\}$ is naturally  a $C_0(\X)$-algebra over the  the topological inverse limit space 
$$
\X:=\underleftarrow{\lim\,\,}\{X_{n+1},\p_{n}\}.
$$
 To this end, we attach to each $\x=(x_0,x_1,...)\in \X$ the direct limit $C^*$-algebra 
$$
B(\x):=\underrightarrow{\lim\,\,}\{B_{n}(x_n),\al_{x_{n+1}}\}
$$
of the  direct sequence $B_{0}(x_0)\stackrel{\al_{x_1}}{\longrightarrow} B_{1}(x_1)
\stackrel{\al_{x_2}}{\longrightarrow} B_{2}(x_2)  \stackrel{\al_{x_3}}{\longrightarrow}...\, 
$. We let  $\phi_{\x,n}: B_n(x_n) \to B(\x)$ and $\phi_n:B_n \to B$ be  the natural homomorphisms:
$$
\phi_n(b_n)=[\underbrace{0,...,0}_{n},b_n,\al_{n+1}(b_n),...) ],
$$
$$
\phi_{\x,n}(b_n)=[\underbrace{0,...,0}_{n},b_n(x_n),\al_{x_{n+1}}(b_n(x_n)),...) ]
$$
where $b_n \in B_n$ and $\x\in \X$. The  following statement can be viewed as a generalization of \cite[Proposition 1.7]{HRW} to the non-unital case. In contrast to \cite{HRW} we prove it using the $C^*$-bundle approach. 
\begin{prop}\label{proposition for Hirshberg,Rordam,Winter}
Retain the above notation and  suppose additionally that  the mappings $\p_n$ are surjective, $n\in \N$.  There  is a unique topology on $\B=\bigsqcup_{\x\in \X} B(\x)$ making it an upper semicontinuous $C^*$-bundle  over $\X$ such that 
\begin{equation}\label{direct limit sections}
B \ni \phi_n(b_n) \longrightarrow \phi_{n}(b_n)(\x):=\phi_{\x,n}(b_{n}(x_n)), \qquad \x \in \X, b_n \in B_n,
\end{equation}
establishes the natural isomorphism from  $B=\underrightarrow{\lim\,\,}\{B_{n},\al_{n}\}$  onto $\Gamma_0(\B)$. 

If additionally, all the algebras $B_n$ are continuous $C_0(X_n)$-algebras and all the endomorphisms $\alpha_x$, $x\in X_{n+1}$, $n\in \N$, are injective, then
the $C^*$-bundle $\B=\bigsqcup_{\x\in \X} B(\x)$ is  continuous. 
\end{prop}
\begin{proof}
For $\x\in \X$ and $m > n$ we put
$$
\alpha_{\x, [n,m]}:=\alpha_{x_m}\circ   ...\circ \alpha_{x_{n+2}}\circ\alpha_{x_{n+1}} \quad\textrm{ and }\quad \alpha_{[n,m]}:=\alpha_{m-1}\circ   ...\circ \alpha_{n+1}\circ\alpha_{n}.
$$
These are  the bonding homomorphisms from $B(x_n)$ to $B(x_m)$ and  from $B_n$ to $B_m$, respectively. Let $\x \in \X, b_n \in B_n$. To check that the map \eqref{direct limit sections} is well defined assume that $\phi_{n}(b_n)=0$. Then for any $\varepsilon >0$ and sufficiently large $m$ we have $\|\alpha_{[n,m]}(b_n)\|<\varepsilon $, and all the more  $\|\alpha_{\x, [n,m]}(b_n(x_n))\| = \|\alpha_{[m,n]}(b_n)(x_m)\| < \varepsilon $. This implies that  $\phi_{\x,n}(b_{n}(x_n))=0$. Hence  \eqref{direct limit sections} is well defined and  clearly  it   yields a  surjective homomorphism from $B$ onto $B(\x)$. 
 \\
We show  that for a fixed $\phi_{n}(b_n) \in B$, the mapping 
\begin{equation}\label{norm map to check}
\X \ni \x \mapsto \|\phi_{n}(b_n)(\x)\| \in \C
\end{equation}
is upper semicontinuous.  Suppose that $\x \in \X$ is such that  $\|\phi_{n}(b_n)(\x)\| <K$. Then there is $m> n$ such that 
$
\|\alpha_{\x, [n,m]}(b_n(x_n))\| <K.
$
Since  $\alpha_{[m,n]}(b_n)(x_m)=\alpha_{\x, [n,m]}(b_n(x_n))$ and $X_m \ni x \to \|\alpha_{[m,n]}(b_n)(x)\|$ is upper semicontinuous,   there is an open  neighborhood $U$ of $x_m$ such that $\|\alpha_{[m,n]}(b_n)(x)\|< K$ for all $x\in U$. It follows that the set 
$$
\widetilde{U}:=\{\y=(y_0,y_1,...)\in \X:  y_m \in U\}
$$ 
  is an open neighborhood of $\x$ such that for $\y\in \widetilde{U}$ we have 
  $$\|\phi_{n}(b_n)(y)\|\leq
\|\alpha_{\y, [n,m]}(b_n(y_n))\| <K.
$$
 This proves the upper semicontinuity of \eqref{norm map to check}. 

We wish to show that  \eqref{norm map to check} vanishes at infinity.  Let $\varepsilon >0$.  By upper semicontinuity of \eqref{norm map to check} the set  $\{\x\in \X: \|\phi_{n}(b_n)(\x)\| \geq \varepsilon\}$ is closed, and clearly,  it is  a subset of 
$
\{\x\in \X: \|b_n(x_n)\| \geq \varepsilon\}
$.  However, the latter set is compact because the map $\X \ni \x \to x_n \in X_n$ is proper  and  $X_n \ni x\to \|b_n(x)\|$ is vanishing at the infinity. Hence $\{\x\in \X: \|\phi_{n}(b_n)(\x)\| \geq \varepsilon\}$  is compact as well.

Now, by Fell's theorem, see \cite[Theorem C.25]{Williams},  there is a unique topology on $\B$ such that $\eqref{direct limit sections}$ defines a surjective homomorphism from $B$ onto $\Gamma_0(\B)$.  We still need to show that this homomorphism is injective.
\\
 To this end, assume that $\phi_n(b_n)$ is non zero. Then there exists $\varepsilon >0$ such that   $\|\alpha_{[m,n]}(b_n)\|>\varepsilon $ for all $m>n$. Thus, for each $m>n$, the set 
$$
D_m:=\{x\in X_m: \|\alpha_{[m,n]}(b_n)(x)\| \geq \varepsilon\}
$$
is  nonempty,  and it is compact because  $X_m \ni x \to \|\alpha_{[m,n]}(b_n)(x)\|$  vanishes at infinity. Note that $\p_{m}(D_{m+1})\subseteq D_m$. Thus the sets 
$$
\widetilde{D}_m:=\{\x\in \X: x_m \in D_m\}
$$
form a decreasing sequence of compact nonempty sets (non-emptiness follows from surjectivity of the mappings $\p_m$). Hence there is $\x_0\in \bigcap_{m> n} \widetilde{D}_m$ and plainly 
$$
\|\phi_{n}(b_n)(\x_0)\|\geq \varepsilon >0.
$$
This finishes the proof of the first part of the assertion.

Assume now that for each $n\in \N$, $B_n$ is a continuous $C_0(X_n)$-algebra and all of the endomorphisms $\alpha_x$, $x\in X_n$, are injective. Then  $\|\phi_n(b_n)(\x)\|=\|b_n(x_n)\|$ for all $\x\in \X$, $b_n\in B_n$, $n\in \N$. Hence  mapping \eqref{norm map to check}, as a composition of two continuous mappings $\X \ni \x \to x_n \in X_n$ and $X_n \ni x_n \to \|b_n(x_n)\|$, is continuous.
\end{proof}
Injectivity of the endomorphisms $\alpha_x$, $x\in X_{n+1}$, in the second part of Proposition \ref{proposition for Hirshberg,Rordam,Winter} is essential.
\begin{ex}
Consider the stationary inductive limit given by the continuous $C_0(\N)$-algebra $A:=C_0(\N, \C^2)$ and the endomorphism $\alpha:A\to A$ induced by the morphism $(\p,\{\al_{n}\}_{n\in \N})$ where 
$$
\phi(0)=0,\quad \alpha_0=id, \quad \textrm{ and }\quad  \phi(n)=n-1,\quad \alpha_n(a,b)=(a,0), \textrm{ for }n>0.
$$
The resulting direct limit $B=\underrightarrow{\lim\,\,}\{A,\al\}$ can be viewed as a $C_0(\{-\infty\}\cup \Z)$-algebra with the obvious topology on  $\{-\infty\}\cup \Z$ and fibers $B_{-\infty}=\C^2$ and $B_{n}=\C$, $n\in \Z$.  The image of the  constant function $\N \ni n\rightarrow (a,b)\in \C^2$ (treated as an element of $A$) in the algebra $B$ corresponds to the section $f$ with $f(-\infty)=(a,b)$ and $f(n)=a$ for $n\in \Z$. If $|a|< |b|$, the function $\{-\infty\}\cup \Z \ni x\to \|f(x)\|$ is not lower semicontinuous at $-\infty$.
\end{ex}

\section{Crossed products of  $C_0(X)$-dynamical systems}\label{reversible section}

In this section, we fix  a $C_0(X)$-dynamical system $(A,\alpha)$ and denote by $(\p,\{\al_{x}\}_{x\in \Delta})$ a morphism that induces $\alpha$. We also fix an ideal $J$ in $(\ker\alpha)^\bot$ and make the following standing assumption:
\begin{itemize}
\item $\sigma_A:\Prim(A)\to X$ is surjective, equivalently $A(x)\neq \{0 \}$ for all $x\in X$.
\end{itemize} 
The above assumption  allows us to treat $C_0(X)$ as a subalgebra of $M(Z(A))$.  We will study  crossed products $C^*(A,\alpha;J)$ using the following tactic.  Firstly, we consider reversible systems. Then we show that the natural reversible $J$-extension $(B,\beta)$ of $(A,\alpha)$ is induced by a morphism, which will immediately lead us to general results.

\subsection{The case of a reversible system}\label{subsection 41}
In this subsection, we assume that  $(A,\alpha)$ is a reversible $C^*$-dynamical system. The dual partial homeomorphism $\hal:\widehat{ \al(A)}\to \widehat{(\ker\al)^{\bot}}$, cf. Definition \ref{dual partial homeomorphism}, factors through to the partial homeomorphism of the primitive ideal space $\Prim (A)$. We denote the latter mapping by   $\widecheck{\alpha}:\Prim(\al(A))\to \Prim((\ker\al)^{\bot})$. Thus we have $\widecheck{\alpha}(\ker \pi)=\ker(\pi\circ \alpha)$ for $\pi \in \widehat{ \al(A)}$. With the identifications 
$
\Prim(\alpha(A))=\{P \in \Prim(A): \alpha(A)\nsubseteq P\}$  and $\Prim((\ker\al)^{\bot})=\{P \in \Prim(A): (\ker\al)^{\bot}\nsubseteq P\}
$   we have
\begin{equation}\label{habeta as beta_*}
\widecheck{\alpha}(P)=\alpha^{-1}(P), \qquad P  \in \Prim(\al(A)).
\end{equation}
\begin{lem}\label{commutation of dual mappings}
With the above notation,  the following diagram 
$$
\begin{xy}\xymatrix{
      \Prim(\al(A)) \ar[d]_{\sigma_A} \ar[r]^{\widecheck{\alpha}} & \Prim((\ker\al)^{\bot}) \ar[d]^{\sigma_A} 
       \\
  \Delta \ar[r]^{\p} &   \p(\Delta)  
              }
               \end{xy}
$$
commutes. In particular, if $\p$ is free then $\widecheck{\alpha}$ is free, and if $A$ is a continuous $C_0(X)$-algebra and $\p$ is topologically free, then $\widecheck{\alpha}$ is also topologically free.
\end{lem} 
\begin{proof}
  Using, among other things, \eqref{J-ideal fibers} and \eqref{estimate set}  we see that
\begin{align*}
\sigma_A(\Prim(\al(A)))&=\sigma_A(\Prim(A\al(A)A))=\{x\in X: \big(A\al(A)A\big)(x)\neq 0\}
\\
&=\{x\in X: \alpha(A)(x)\neq 0\}\subseteq \Delta,
\\
\sigma_A(\Prim((\ker\al)^{\bot}))&=\{x\in (\ker\al)^{\bot}(x)\neq 0\}\subseteq \p(\Delta).
\end{align*}
Now let  $P\in \Prim(\al(A))$. Then $x:=\sigma_A(P)$ is  in $\Delta$. By \eqref{base map via primitive ideals},  $C_0(\Delta\setminus \{x\})\cdot A\subseteq P$. Applying  to this inclusion $\alpha^{-1}$ we get  
$
C_0\Big(\p(\Delta)\setminus \left\{\p(x)\right\}\Big)\cdot \alpha^{-1}(A)  \subseteq\alpha^{-1}(P)
$. This in view of \eqref{habeta as beta_*} and \eqref{base map via primitive ideals} means that $\sigma_{A}(\widecheck{\alpha}(P))=\p(x)=\sigma_A(P)$. The last part of the assertion is straightforward.
\end{proof}
Clearly, (topological) freeness of $\widecheck{\alpha}$ implies (topological) freeness of $\widehat{\alpha}$. Hence Lemma \ref{commutation of dual mappings} and Proposition \ref{interactions?} give us the following results.
\begin{cor}\label{takie sobie corollary}
 Suppose that $A$ is a continuous $C_0(X)$-algebra and $\p$ is topologically free. A  representation  of the crossed product $C^*(A,\alpha)$ is faithful if and only if it is faithful on $A$.
\end{cor}
\begin{cor}\label{takie sobie corollary2}
 If  $\p$ is free then every ideal in  $C^*(A,\alpha)$ is gauge-invariant.
\end{cor}
In order to use  our pure infiniteness criterion - Proposition \ref{pure infiniteness for reversible systems}, we need to show that freeness of $\p$ implies that $A^+$ residually supports elements of $C^*(A,\alpha)$. To this end we use  a technical device introduced in the following lemma, which will allow us to adapt the arguments of \cite{exel3} to our setting. Recall that any $C^*$-algebra $B$ is a $M(B)$-bimodule where $(m\cdot b):=mb$ and $(b\cdot m):=(m^*b^*)^*$, for $m\in M(B)$, $b\in B$.
\begin{lem}
The action of $h\in C_0(X)$ on $A$ as a multiplier of $A$ extends to the action on $C^*(A,\alpha)$ as a multiplier of $C^*(A,\alpha)$  which is uniquely determined by the formulas 
\begin{equation}\label{multiplier action extended}
h \cdot \left(au^n\right):=\left(h \cdot  a\right)u^n,\qquad  \left(a\alpha^n(b)u^n\right)\cdot h  := au^n (b\cdot h)=a\alpha^n(b\cdot h)u^n, 
\end{equation}
where $ a,b \in A, n\in \N$.
\end{lem}
\begin{proof}
Recall that  $A$ is a non-degenerate subalgebra of $C^*(A,\alpha)$. In other words, multiplication from the left defines a  non-degenerate homomorphism from $A$ into $M(C^*(A,\alpha))$. This homomorphisms extends uniquely to the homomorphism from $M(A)$ into $M(C^*(A,\alpha))$. Composing the latter with $\mu_A:C_0(X) \to Z (M(A))$ we  get a multiplier action of $C_0(X)$ on $C^*(A,\alpha)$ that clearly satisfies \eqref{multiplier action extended}. In view of   Proposition \ref{crossed product for reversible system} formulas \eqref{multiplier action extended} determine this action uniquely.  
\end{proof} 
In the following statements  we use the  $C_0(X)$-bimodule structure on $C^*(A,\alpha)$ described in the previous lemma (to increase readability  we will suppress the symbol `$\cdot$').
\begin{lem}[cf. Lemma 2.3 in \cite{exel3}]\label{rokhlin like lemma0}
Let $k>0$ and $a\in A\alpha^k(A)$. Suppose that $x_0\in X$  is not fixed by $\p^{k}$. For every $\varepsilon>0$ there is $h\in C_0(X)$ such that 
 $0\leq h\leq 1$, $h(x_0)=1$ and  $\|h(au^k)h\|\leq \varepsilon$.
\end{lem}
\begin{proof} The  proof of \cite[Lemma 2.3]{exel3}  is readily adapted to our case; it suffices to  replace the partial crossed product convolution formula with \eqref{multiplier action extended}.
\end{proof}
\begin{prop}[cf. Proposition 2.4 in \cite{exel3}]\label{rokhlin like lemma}
Suppose that either  $\p$ is topologically free and $A$ is a continuous $C_0(X)$-algebra, or
that  $\p$ is free.
Then for every $a\in C^*(A,\alpha)$ and every $\varepsilon>0$ there is $h\in C_0(X)$ such that 
\begin{itemize}
\item[(i)] $\|h\mathcal{E}(a)h\| \geq \|\mathcal{E}(a)\|- \varepsilon$,
\item[(ii)] $\|h\mathcal{E}(a)h - hah\| \leq \varepsilon$,
\item[(iii)] $ h\geq 0$ and $\|h\|=1$,
\end{itemize}
where $\mathcal{E}$ is the conditional expectation \eqref{conditional expectation formula}.
\end{prop}
\begin{proof}
 We adapt the proof of \cite[Proposition 2.4]{exel3}. A simple approximation argument  implies that we may assume  that $a$ is of the form
\eqref{general form of a guy in cross}.
 Then $\mathcal{E}(a)=a_0$. Let us consider the non-empty set $
V=\{x\in X: \|a_0(x)\| > \|a_0\|- \varepsilon\}
$ and notice that there exists $x_0\in V$ such that $x_0$ is not a fixed point for $\p^k$ for all $k=1,...,n$. Indeed, if   $\p$ is free, existence of such $x_0$ is obvious. If $A$ is a continuous $C_0(X)$-algebra then $V$ is open  and the existence of $x_0$ is guaranteed by topological freeness of $\p$. Applying Lemma \ref{rokhlin like lemma0} we see that for each $k=\pm 1,...,\pm  n$ there exists $h_k\in C_0(X)$ such that
$$
h_k(x_0)=1, \qquad \|h_k(a_ku^{|k|})h_k\| \leq \frac{\varepsilon}{2n}, \quad \textrm{and } 0\leq h_k \leq 1.
$$
Let $h:=\prod_{k=\pm 1,...,\pm  n}h_k$. Then (iii) is immediate, and (i)  holds because $\|ha_0 h\| \geq \|a_0(x_0)\|> \|a_0\|- \varepsilon$.
For (ii), we have
$$
\|ha_0 h - h a h\| \leq \sum_{k=\pm 1,...,\pm  n} \|h (a_ku^{|k|})h\| \leq \sum_{k=\pm 1,...,\pm  n} \|h_k (a_ku^{|k|})h_k\| < \varepsilon.
$$
\end{proof}

\begin{prop}\label{freness implies supportedness}
If  $\p$ is free then   $A^+$ residually supports elements of $C^*(A,\alpha)^+$.
\end{prop}
\begin{proof}
 Let $\I$ be an ideal in $C^*(A,\alpha)$. By  Corollaries \ref{takie sobie corollary2} and \ref{corollary for dolary}, we have the isomorphism $
C^*(A,\alpha)/\I \cong C^*(A/I,\alpha_I)$ where $I:=A\cap \I$ is an invariant ideal in $(A,\alpha)$. The system  $(A/I,\alpha_I)$ is reversible  by Lemma \ref{complemented kernel invariance lemma} ii). By Proposition \ref{taki tam sobie lemma},   $(A/I,\alpha_I)$ is induced by the morphism $(\p, \{\alpha_{I,x}\}_{x\in  \Delta})$. Fix a positive element $b$ in $C^*(A,\alpha)/\I $. Without loss of generality we may assume that  $\|b\|=1$. Applying Proposition \ref{rokhlin like lemma} to $(A/I,\alpha_I)$ ,   we may find a positive contraction $h \in M(A/I)$ such that \eqref{compression relations} holds. Now the last part of the proof of Proposition \ref{aperiodicity implies supportedness} shows that $a:=(h\mathcal{E}(b)h -1/2)_+\in A/I$ is non-zero element such that  $a \precsim b$ relative to $C^*(A,\alpha)/\I \cong C^*(A/I,\alpha_I)$.
\end{proof}
\begin{cor}\label{pure infiniteness for reversible systems cor}
Suppose that $\p$ is free.
\begin{itemize}
\item[(i)] If $A$ is has the ideal property and is  purely infinite, then the same holds $C^*(A,\alpha)$.
\item[(ii)]  If there are finitely many invariant ideals in  $(A,\alpha)$ and $A$ is purely infinite, then $C^*(A,\alpha)$ is purely infinite.
\end{itemize}
\end{cor}
\begin{proof}
The assertion follows from 
 Proposition \ref{freness implies supportedness} and  Proposition \ref{pure infiniteness for reversible systems} modulo  Remark \ref{funny remark}.
\end{proof}

\subsection{Reversible extension of a $C^*$-dynamical system  induced by a morphism}\label{subsection 42}
Let us now get back to the case of a not necessarily reversible $C_0(X)$-dynamical system $(A,\alpha)$. 
Let $(B,\beta)$ denote the reversible $J$-extension of $(A,\alpha)$, cf. Definition \ref{reversible extension definition}. We put 
$$
Y=\overline{\sigma_A(\Prim(A/J))}.
$$
In view of our standing assumption, equality \eqref{J-quotient fibers} and the fact that $J$ is contained in the set \eqref{estimate set} we see  that $Y$ contains $X\setminus \p(\Delta)$.  
We denote by $(\X,\widetilde{\p})$ the reversible $Y$-extension of the partial dynamical system $(X,\p)$, see Definition \ref{reversible ext defn}. Our aim is to use Proposition \ref{proposition for Hirshberg,Rordam,Winter} to describe  $B$ as a $C_0(\X)$-algebra and $\beta$ as an endomorphism  induced by a morphism $(\widetilde{\p}, \{\beta_{\x}\}_{\x\in \Delta})$ for a certain field of homomorphisms $\beta_{\x}$, $\x\in \Delta$.

We start by fixing indispensable notation. Let $\Delta_n:=\p^{-n}(\Delta)$ be  the domain of $\p^n$, $n\in \N$. For $x\in \Delta_n$  we put  
$$
\al_{(x,n)}:=\al_x \circ \al_{\p(x)}\circ ... \circ \al_{\p^{n-1}(x)}, \qquad n >0,
$$
and  $\al_{(x,0)}:=id$. To each $n\in \N$ and  $x\in X$   belonging to the domain of $\varphi^n$, we associate the hereditary subalgebra in  $A(x)$ generated by the range of $ \al_{(x,n)}$: 
\begin{equation}\label{strange subalgebras}
A_n(x):=\alpha_{(x,n)}(A(\p^n(x)))A(x) \alpha_{(x,n)}(A(\p^n(x))).
\end{equation}
We construct  fibre $C^*$-algebras $B(\x)$ as follows. If $\x=(x_0,x_1,...)\in X_{\infty}$,  we let
$$
B(\x):=\underrightarrow{\lim\,\,}\{A_n(x_n),\al_{x_{n+1}}\}$$
to be  the inductive limit of the sequence  $A_0(x_0)\stackrel{\al_{x_1}}{\longrightarrow} A_1(x_1)
\stackrel{\al_{x_2}}{\longrightarrow} A_2(x_3)  \stackrel{\al_{x_3}}{\longrightarrow}...\,
$. 
 If $\x=(x_0,x_1,...,x_N, 0,...)\in X_{N}$,  we simply put
$$
B(\x)=A_N(x_N)/J(x_N).
$$
In other words, $B(\x)=q_{x_N}(A_N(x_N))$  where $(id,\{q_x\}_{x\in Y})$ is the morphism that induces the quotient map $q:A\to A/J$,  see Example \ref{example quotient}. 

We will represent the dense $*$-subalgebra $\bigcup_{n\in\N} B_n$ of $B$ as an algebra of sections of $\B=\bigsqcup_{\x\in \X} B(\x)$. For every $a=(a_0+J)\oplus ... \oplus (a_{n-1}+J)\oplus a_n \in B_n$ we define the section $\pi(\phi_n(a))$ of $\B$ by the formula 
$$
\pi(\phi_n(a))(\x)=
\begin{cases}
a_N(x_N) + J(x_N), & \x\in X_N,\,\, N \leq n, \\
\al_{(x_{N}, N-n)}(a_n(x_n)) + J(x_N), & \x\in X_N,\,\, N > n, \\
[\underbrace{0,...,0}_{n},a_n(x_n),\al_{x_{n+1}}(a_n(x_n)),\al_{(x_{n+2},2)}(a(x_n)),...],  & \x\in X_{\infty}.
\end{cases}
$$
 We define endomorphisms $\beta_{\x}:B(\widetilde{\p}(\x))\to B(\x)$,  $\x\in \TDelta$, as follows. We put 
$$
\beta_{\x}[a_0,a_1,a_2,...]:=[a_1,a_2,...], \qquad  \text{if } \x\in  X_\infty\cap \TDelta,
$$
and if $\x\in X_N\cap \TDelta$ we let $\beta_{\x}$ be the inclusion map corresponding to the inclusion $B(\tp(\x))=q_{x_N}(A_{N+1}(x_N))\subseteq B(\x)=q_{x_N}(A_N(x_N))$.

If the system $(A,\alpha)$ happens to be extendible, then by Lemma \ref{characterization of extendible induced endomorphisms}  the homomorphisms $\alpha_x$, and hence also $\al_{(x,n)}$, are extendible. In this case we put 
$$
 p_{(x,0)}:=1_x, \,\, x\in X,\quad\text{ and }\quad  p_{(x,n)}:=\overline{\al}_{(x,n)}(1_{\varphi^{n}(x)}) \quad \text{ for }n>0, \,\, x\in \p^{-n}(\Delta).
$$
With this notation the algebras $A_n(x)$ are corners $p_{(x,n)}A(x) p_{(x,n)}$. We define  positive linear maps $\beta_{*,\tp(\x)}: B(\x) \to B(\widetilde{\p}(\x))$, $\x\in \TDelta$, as follows. For $X_\infty\cap \TDelta$ we set 
$$
\beta_{*,\tp(\x)}[a_0,a_1,a_2,...]:=[0,p_{(x_0,1)} \,a_0\, p_{(x_0,1)},\,p_{(x_1,2)}\,a_1\,p_{(x_1,2)},\, ...].
$$
If $\x\in X_N \cap \TDelta$, we  put  $\beta_{*,\tp(\x)}(q_{x_N}(a)):=q_{x_N}(p_{(x_N,N+1)}\,a \,p_{(x_N,N+1)})$.

\begin{thm}\label{proposition C-bundle B}
Retain the above notation. There is a unique topology on $\B=\bigsqcup_{\x\in \X} B(\x)$ making it into an upper semicontinuous $C^*$-bundle  over $\X$ such that $\pi$ establishes the isomorphism 
$B\cong \Gamma_0(\B).$
Identifying  $B$ with the  algebra of continuous sections of $\B$ we have
$$
\beta(a)(\x)=\begin{cases} 
\beta_{\x}\big(a\big(\widetilde{\p}(\x)\big)\big), & \x\in \widetilde{\Delta},\\
0, & \x\notin \widetilde{\Delta},
\end{cases}
$$
and if $(A,\alpha)$ is extendible, then $(B,\beta)$ is extendible and
$$
 \beta_*(a)(\x)=\begin{cases} 
\beta_{*,\x}\big(a\big(\widetilde{\p}^{-1}(\x)\big)\big), & \x\in \widetilde{\varphi}(\widetilde{\Delta}),\\
0, & \x\notin \widetilde{\varphi}(\widetilde{\Delta}),
\end{cases} 
$$
where $\beta_*$ is  the unique regular transfer operator for $(B,\beta)$. 
Moreover, if $A$ is a continuous $C_0(X)$-algebra and either 
\begin{itemize}
\item[(i)]  $J$ is a complemented ideal, every  ideal $J(x)$ is trivial (i.e. either $\{0\}$ or $A(x)$),  and every homomorphism $\al_x$ is injective,  or
\item[(ii)]  $\sigma_A$ is injective, i.e.  $X\cong \Prim(A)$, cf. Example \ref{primitive Hausdorff example},
\end{itemize}
then $\B=\bigsqcup_{\x\in \X} B(\x)$ is a continuous $C^*$-bundle. 
\end{thm}
\begin{proof} 
Notice that $B_n$, $n\in \N$, is naturally a $C_0(Y_n)$-algebra with 
$$
Y_n:=Y  \sqcup   Y\cap \Delta_1 \sqcup ... \sqcup Y\cap \Delta_{n-1} \sqcup \Delta_n
$$
where $\sqcup$ denotes the disjoint sum of topological spaces.   Moreover, the bonding homomorphism $\alpha_n:B_n \to B_{n+1}$ is induced by a morphism. Indeed, consider the map $\varphi_n:Y_{n+1}\to Y_n$  given by the diagram
$$
\begin{xy}
\xymatrix@C=3pt{
      **[r]  Y_n & = &  Y &  \sqcup & ... &  \sqcup &
      Y\cap \Delta_{n-1}& \sqcup & \Delta_n      \\
       Y_{n+1} \ar[u]^{\varphi_n}& = &  Y \ar[u]^{id}& \sqcup & ... &  \sqcup&   Y\cap \Delta_{n-1} \ar[u]^{id}&\sqcup &   Y\cap \Delta_{n} \ar[u]^{id}& \sqcup  & \Delta_{n+1}.\ar[llu]^{\varphi}
        }
  \end{xy}
  $$
Let $y\in Y_{n+1}$. Define $\alpha_{x,n}:B_n(\varphi_{n}(y))\to B_{n+1}(y)$ to be identity if  $y\in Y\cap \Delta_{k}$ belongs to the $k$-th summand of $X_{n+1}$,    $k\leq n-1$, to be  $q_y$ if $y\in Y\cap \Delta_{n}$ belongs to the $n$-th summand of $X_{n+1}$, and to be   $\alpha_y$ if  $y\in \Delta_{n+1}$ belongs to the last summand of $Y_{n+1}$. Then $\alpha_n:B_n \to B_{n+1}$ is induced by 
$(\varphi_n, \{\alpha_{x,n}\}_{x\in Y_{n+1}})$.

Since $ \varphi(\Delta)\cup Y=X$, the mappings $\varphi_n:Y_{n+1}\to Y_n$ are surjective and we may apply  Proposition \ref{proposition for Hirshberg,Rordam,Winter} to the inductive system $\{B_n,\alpha_n\}_{n\in \N}$. The arising direct limit $C^*$-bundle can be identified with  the one described above. Indeed, we have a natural homeomorphism $\Phi: \X\to \underleftarrow{\lim\,\,}\{Y_{n+1},\varphi_{n}\}$. Namely, for  $\x=(x_0,x_1,x_2,...)\in X_\infty$ we define  $\Phi(\x)=(x_0,x_1,x_2,...)\in \underleftarrow{\lim\,\,}\{Y_{n+1},\varphi_{n}\}$  where in the latter we treat $x_n\in \Delta_n$ as the point in last direct summand of $Y_n$ for all $n\in \N$. For   $\x=(x_0,x_1,...,x_N,0,0,...)\in X_N$ we define $\Phi(\x)=(x_0,x_1, ...,x_N,x_N,x_N,...)\in \underleftarrow{\lim\,\,}\{Y_{n+1},\varphi_{n}\}$ where in the latter we treat $x_n \in \Delta_n$ as the point in the last direct summand of $Y_n$, for $n\leq N$, and $x_N\in Y\cap \Delta_{N}$ as the point in $N$-the direct summand in $Y_n$ for $n\geq N$.  For $\x\in X_\infty$, the algebras $B(\x)$ and $B(\Phi(\x))$ are naturally isomorphic because they  arise as  direct limits of  direct sequences that can be naturally identified. For   $\x=(x_0,x_1,...,x_N,0,0,...)\in X_N$, $B(\Phi(\x))$ is naturally isomorphic to $B_N(x_N)$ where $x_N\in Y_N$ lies in the last direct summand of $Y$. Thus $B(\Phi(\x))$ is naturally isomorphic to $A_N(x_N)/J(x_N)$. Hence we may identify the corresponding fibers of bundles. Then we get   $
\pi(\phi_{n}(a))(\x)=\phi_{n}(a)(\Phi(\x))$ for any $a\in B_n$ and $\x\in \X$.  This proves the first part of the assertion.

Let $a=a_0+J\oplus ... \oplus a_{n-1}+J\oplus a_n \in B_n$ and $\x\in \X$. Note that $\pi(\phi_{n}(\beta_n(a)))(\x)$ is equal to
$$
\begin{cases}
a_{N+1}(x_N) + J(x_N), & \x\in X_N,\,\, N+1\leq n, \\
\al_{(x_{N}, N+1-n)}(a_{n}(x_{n-1})) + J(x_N), & \x\in X_N,\,\, N+1> n, \\
[\underbrace{0,...,0}_{n},\al_{x_{n}}(a_n(x_{n-1})),\al_{(x_{n+1},1)}(a(x_{n-1})),...],  & \x\in X_{\infty}.
\end{cases}
$$
Suppose first that $\x=(x_0,...,x_n,...)\notin \widetilde{\Delta}$. Since $x_0\in X\setminus \Delta$ we either have  $x_N\in \Delta_N\setminus\Delta_{N+1}$, when $\x\in X_N$ for $N<n$, or $x_n\in \Delta_n\setminus\Delta_{n+1}$, otherwise. Thus  $\pi(\beta_n(a))(\x)=0$ because for any $x\notin \Delta_k$, using that $a_k\in \alpha^k(A)A\alpha^k(A)$,  we get  $a_k(x)=0$.   On the other hand, for $\x\in \widetilde{\Delta}$ one  sees that   
$$
\pi(\phi_n(\beta_n(a)))(\x)=\beta_{\x}(\pi(\phi_n(a))(\widetilde{\varphi}(\x))).
$$
  Hence, in view of Proposition \ref{characterization of induced endomorphisms},  $\beta:B\to B$ is induced by the morphism $(\widetilde{\varphi}, \{\beta_{\x}\}_{\x\in \widetilde{\Delta}})$.

If $(A,\alpha)$ is extendible then $(B,\beta)$ is extendible by \cite[Proposition 2.4]{kwa-rever}. Moreover, invoking Proposition \ref{transfer induced by a transfer morphism} and formula \eqref{formula to invoke} one concludes that the corresponding transfer operator $\beta_*$ satisfies the formula described in the assertion with the mappings 
 $$\beta_{*,\widetilde{\varphi}(\x)}(b):=\beta_{\x}^{-1}\Big(\overline{\beta}_{\x}(1_{\widetilde{\varphi}(\x)}) a\overline{\beta}_{\x}(1_{\widetilde{\varphi}(\x)})\Big), \qquad b\in B(\x), \,\, \x \in \widetilde{\Delta},
$$
where $\beta_{\x}^{-1}$ is the inverse to the isomorphism $\beta_{\x}:(\ker\beta_{\x})^\bot\to \overline{\beta}_{\x}(1_{\widetilde{\varphi}(\x)}) B(\x)\overline{\beta}_{\x}(1_{\widetilde{\varphi}(\x)})$.  We leave it to the reader to check  that these maps coincide with the maps  we have previously described. This proves the second part of the assertion.

For the last part of the assertion it suffices to apply the second part of  Proposition \ref{proposition for Hirshberg,Rordam,Winter}. Indeed, if all  ideals $J(x)$ are trivial and all  homomorphisms $\al_x$ are injective, then all  homomorphisms $\alpha_{x,n}$  are injective. Moreover, if $J$ is complemented or $\sigma_A$ is injective, then $B_n$, $n\in \N$, is  a continuous $C_0(Y_n)$-algebra by the last part of Lemma \ref{lemma on ideals in C_0(X)-algebras}.
\end{proof}
\begin{rem}
The morphism $(\widetilde{\p}, \{\beta_{\x}\}_{\x\in \Delta})$ constructed above can be considered canonical. In particular,  $\widetilde{\p}$ is always a partial homeomorphism and all the homomorphisms $\beta_{\x}$ are injective. Thus even when the initial system $(A,\alpha)$ is already reversible and $J=(\ker\alpha)^\bot$, so that we have $(A,\alpha)=(B,\beta)$, the morphism $(\p,\{\alpha_{x}\}_{x\in\Delta})$ may differ from  $(\widetilde{\p}, \{\beta_{\x}\}_{\x\in \Delta})$. For instance, for the reversible dynamical system $(A,\alpha)$  described in Example \ref{transfer not induced} we obtain (omitting zero fibers) that $B=A=\C^3$ is naturally a $C_0(\{1,2,3\})$-algebra and $\beta=\alpha$ is induced by the morphism $(\p,\{\alpha_{1}, \alpha_3\})$ where 
$
\p(1)=2$, $\p(3)=3$, $\alpha_{1}=\alpha_3=id$.

\end{rem}

\subsection{General results}\label{subsection 43}
Now, we put together the results of the previous subsections. For the first statement recall Definition \ref{definition of topologicall freness outside}.

 \begin{thm}\label{uniqueness theorem} Let $(A,\alpha)$ be a continuous $C_0(X)$-dynamical system. Suppose that   
$\p$ is topologically free outside the set $Y=\overline{\sigma_A(\Prim(A/J))}$ and either (i) or (ii) in Theorem \ref{proposition C-bundle B} holds. 
Every injective representation $(\pi,U)$ of $(A,\alpha)$ such that $
J=\{a\in A: U^*U\pi(a)=\pi(a)\}$ give rise to a faithful representation $\pi\rtimes U$ of $C^*(A,\alpha;J)$.

\end{thm}
\begin{proof}
By Theorem \ref{proposition C-bundle B} the reversible $J$-extension $(B,\beta)$ of $(A,\alpha)$ is induced by a morphism based on the reversible $Y$-extension $(\X,\tp)$ of $(X,\p)$. Moreover, $B$ is a continuous $C_0(\X)$-algebra and $\tp$ is topologically free by  Lemma \ref{equivalence of topological freedom}. Hence the assertion follows from Corollary \ref{takie sobie corollary} applied to $(B,\beta)$. 
 \end{proof}
The first part of the following result also implies the uniqueness  property described in Theorem \ref{uniqueness theorem}.  It does not require  continuity  of the $C_0(X)$-algebra $A$,   still it requires freeness of $\p$ which is a much stronger condition than topological freeness.

\begin{thm}\label{pure infiniteness theorem}
Suppose  that  $\p$ is  free. Then all ideals in $C^*(A,\alpha;J)$ are gauge-invariant and, in particular, they are in one-to-one correspondence with $J$-pairs for $(A,\alpha)$. Moreover, 
\begin{itemize}
\item[(i)] If $A$  has the ideal property and is  purely infinite, then the same holds $C^*(A,\alpha;J)$.
\item[(ii)]  If there are finitely many $J$-pairs for  $(A,\alpha)$ and $A$ is purely infinite, then $C^*(A,\alpha;J)$ has finitely many ideals and is purely infinite.
\end{itemize}
\end{thm}
\begin{proof}  Let $(B,\beta)$ be the natural $J$-extension of $(A,\alpha)$.
We show that  $(B,\beta)$   satisfies the assumptions of  Corollary \ref{pure infiniteness for reversible systems cor}, when treated us induced by the morphism $(\widetilde{\p}, \{\beta_{\x}\}_{\x\in \TDelta})$ described in Theorem \ref{proposition C-bundle B}. Plainly,  $\widetilde{\p}$ is free, cf. Lemma \ref{equivalence of topological freedom}. Hence the first part of the assertion follows by  Corollary  \ref{takie sobie corollary2}, applied to $(B,\beta)$.

Suppose that $A$ is purely infinite. Since  pure infiniteness is  preserved under taking direct sums, quotients, hereditary subalgebras and direct limits, see \cite[Propositions 4.3, 4.17 and 4.18]{kr}, we conclude that $B$ is purely infinite.

(i). It is easy to see that the ideal property is  preserved under taking direct sums and  quotients. It is also preserved when passing to direct limits  \cite[Proposition 2.2]{Pasnicu} and  in the presence of pure infiniteness it also  passes to hereditary subalgebras, see \cite[Proposition 2.10]{Pas-Ror}. Thus $B$ is purely infinite and has the ideal property. Accordingly, Corollary \ref{pure infiniteness for reversible systems cor} (i) applies.

(ii). By Lemma \ref{lemma now trivial} (ii) there are finitely many invariant ideals in $(B,\beta)$. Hence  Corollary \ref{pure infiniteness for reversible systems cor} (ii) applies  to $(B,\beta)$.
\end{proof}

\section{Crossed products of  $C^*$-algebras with Hausdorff primitive ideal space}\label{applications section}

In this section, we fix a $C^*$-algebra with a Hausdorff primitive ideal space $X=\Prim(A)$ and consider a $C_0(X)$-dynamical system $(A,\alpha)$ described in  Example \ref{primitive Hausdorff example}. Let $(\p,\{\al_{x}\}_{x\in \Delta})$ be the morphism determining $\alpha$. By Lemma \ref{lemma to be applied}, without loss of generality we may assume that  every $\alpha_x$, $x\in \Delta$, is nonzero (and thus injective), so that $(\p,\{\al_{x}\}_{x\in \Delta})$ is uniquely determined by $\alpha$. Thus, we  make the following standing assumptions:
\begin{itemize}
\item $X=\Prim(A)$ is a Hausdorff space, $\sigma_A=id$, and every  $\alpha_x$, $x\in \Delta$, is non-zero.
\end{itemize}
In particular,  we have a bijective correspondence 
\begin{equation}\label{from topology to algebra and back}
X \supseteq V \longmapsto I_V:=\{a\in A:  a(x)=0 \text{ for all } x\in  V\} \triangleleft A
\end{equation}
between closed  subsets of $X$ and ideals in $A$. We use it to describe ideal structure of the crossed product $C^*(A,\alpha;J)$ in terms of the dynamical system $(X,\p)$. We also give some criteria for $C^*(A,\alpha;J)$  to be purely infinite or a Kirchberg algebra. We finish this section by describing the $K$-theory of all ideals and quotients in $C^*(A,\alpha;J)$ when $A=C_0(X,D)$ with $D$ a simple $C^*$-algebra.

\subsection{Ideal structure of $C^*(A,\alpha;J)$}
In this subsection, we generalize  results proved in the commutative case (i.e., when $D=\C)$ in \cite[Subsection 4.6]{kwa-rever}. Let us fix an ideal $J$ in $(\ker\alpha)^\bot$. Since  $(\ker\alpha)^\bot= I_{\overline{X\setminus \p(\Delta)}}$ we have
$$
J=I_Y \textrm{ where }Y \textrm{ is a closed subset of }X \textrm{ such that } Y \cup \p(\Delta) =X.
$$
We have the following dual version of Definitions \ref{invariance definition} and \ref{J-pairs definition}.
\begin{defn}[Definition 4.9 in \cite{kwa-rever}] \label{invariants definition} A closed set   $V\subseteq X$ is \emph{positively invariant} under $\varphi$ if  $\varphi(V\cap \Delta)\subseteq V$, and $V$ is \emph{$Y$-negatively invariant} if $V\subseteq Y\cup\varphi(V\cap \Delta)$. If $V$ is both positively and $Y$-negatively invariant, we  call  it   \emph{$Y$-invariant}.
 We say that $V$ is \emph{invariant} if it is  $\overline{X\setminus \p(\Delta)}$-invariant.  A pair $(V,V')$ of closed subsets of $X$ satisfying 
$$
V \textrm{ is positively } \varphi\textrm{-invariant,}\quad  V'\subseteq Y \quad \textrm{and}\quad  V'\cup \varphi(V\cap \Delta)=V
$$  
is called a $Y$-pair for $(X,\varphi)$. 
\end{defn}

\begin{lem}\label{proposition on restriction and quotient in trivial case} An ideal $I_V$ in $A$ is  positively (resp. $J$-)invariant if and only if $V$ is positively (resp. $Y$-)invariant). A pair $(I_V,I_{V'})$ of ideals in $A$ is a $J$-pair for $(A,\alpha)$ if and only if $(V,V')$ is a $Y$-pair for $(X,\p)$.

\end{lem}
\begin{proof} We recall that $\p:\Delta\to X$ is necessarily a closed map. In particular, if $V$ is closed then $\varphi(V\cap \Delta)$ is also closed. 
Since  the endomorphisms $\alpha_x$, $x\in \Delta$, are injective, one readily sees that  $\alpha^{-1}(I_V)=I_{\varphi(V\cap \Delta)}$. Using this observation we  get the following equivalences 
$$
\alpha(I_V)\subseteq I_V \,\,\Longleftrightarrow \,\, I_V \subseteq \alpha^{-1}(I_V) \,\,\Longleftrightarrow \,\,\varphi(V\cap \Delta)\subseteq V,
$$ 
$$
J\cap \alpha^{-1}(I_V)\subseteq I_V \,\,\Longleftrightarrow \,\, I_{Y\cup \varphi(V\cap \Delta)} \subseteq I_V \,\,\Longleftrightarrow \,\,V\subseteq Y\cup\varphi(V\cap \Delta),
$$ 
This proves the initial part of the assertion. Similarly as above, we get 
$$
I_{V'}\cap \alpha^{-1}(I_V)=I_V \,\,\Longleftrightarrow \,\, V'\cup \varphi(V\cap \Delta)=V.
$$ 
Since $I_{V'}\subseteq J$ if and only if $V'\subseteq Y$, this completes the proof. 
\end{proof}
We note that   the class $C^*$-dynamical systems satisfying our standing assumptions is closed under taking  quotients and restrictions. 

\begin{lem}\label{lemma for restriction and quotient}
If $V$ is a positively invariant closed set, then  the quotient endomorphism $\alpha_{I_V}$  is  induced by 
the  morphism $(\p|_{\Delta\cap V}, \{\al_{x}\}_{x\in \Delta\cap V})$ where we treat $A/I_V$ as a $C_0(V)$-algebra,  and the restricted endomorphism $\alpha|_{I_V}$  is   induced by 
the  morphism $(\p|_{\Delta\setminus \p^{-1}(V)}, \{\al_{x}\}_{x\in \Delta\setminus \p^{-1}(V)})$ where we treat $I_V$ as a $C_0(X\setminus V)$-algebra.
\end{lem}
\begin{proof} It readily follows from   Proposition \ref{taki tam sobie lemma2}.
\end{proof}

Now we describe the dual topological version of the system introduced in Definition \ref{complementing the kernel}, it's action is schematically presented in \cite[Figure 1]{kwa-logist}. 
\begin{defn}\label{complementing the kernel2}
We define a partial dynamical system $(X^Y,\p^Y)$ by  putting 
$$
X^Y:=
\p(\Delta) \sqcup Y, \qquad \Delta^Y:=(\p(\Delta) \cap \Delta) \sqcup (Y\cap \Delta) \subseteq X^Y
$$
and letting  $\p^Y:\Delta^Y\to X^Y$ to map a point $x$ from  $\Delta^Y$ to the point $\p(x)$ lying in the first disjoint summand $\p(\Delta)$ of $X^Y$.
\end{defn}
\begin{lem}\label{lemma without a proof}
Let $(A^J, \alpha^J)$ be the $C^*$-dynamical system described in  Definition \ref{complementing the kernel}. We may assume the identification $\Prim(A^J)=X^Y$, and  treating $A^J$ as a $C_0(X^Y)$-algebra the endomorphism $\alpha^J$ is induced by the morphism   $(\p^Y, \{\alpha_x\}_{x\in \Delta^Y})$. Moreover, if $(V,V')$ is $Y$-pair for $(X,\p)$,  then
\begin{equation}\label{stupid set description}
(V,V')^Y:=(\p(\Delta)\cap V)\sqcup V'\subseteq  X^Y
\end{equation}
is a closed invariant set in  $(X^Y,\p^Y)$ that corresponds to a positively invariant ideal in $(A^J, \alpha^J)$ given by
\begin{equation}\label{crazy ideal}
I_{(V,V')^Y}=\{a\in A^J: a(x)=0 \text{ for } x\in (V,V')^Y\}.
\end{equation}
\end{lem}
\begin{proof}
 In particular, assuming the identification $\Prim(A^J)=X^Y$, cf.   Lemma \ref{lemma without a proof}, for the corresponding $Y$-pair  $(V,V')$ we have $(I_V,I_{V'})^J=q_{\ker\alpha }(I_V) \oplus  q_{J}(I_{V'})$. 
Thus the assertion follows by Proposition \ref{general gauge-invariant ideals description}. 
\end{proof}

The following proposition generalizes   \cite[Proposition 4.9]{kwa-rever} (proved in the commutative case)  and in addition it describes up to Morita-Rieffel equivalence all ideals   in $C^*(A,\alpha;J)$.

\begin{prop}\label{proposition on ideals in commutative}
If $\p$ is free then all ideals in $C^*(A,\alpha;J)$ are gauge-invariant. In general, we have a  bijective correspondence between gauge-invariant  ideals $\II$ in  $C^*(A,\alpha;J)$ and $Y$-pairs $(V,V')$ for $(X,\varphi)$
established by relations
$$
I_V=A\cap \II, \qquad I_{V'}=\{a\in A: (1-u^*u) a \in \II\}.
$$
Moreover, for the corresponding objects we have an isomorphism 
\begin{equation}\label{isomorphism of the quotient for trivial bundles}
 C^*(A,\alpha;J)/\II \cong  C^*(A/I_V,\alpha_{I_V}; q_{I_V'}(I_V')), 
\end{equation}
and  if $V'=V\cap Y$ (equivalently if $\II$ is generated by $I_V$), then $\II$ is Morita-Rieffel equivalent
 to $C^*(I_V,\alpha|_{I_V};I_{V\cup Y}) $. In general,  $\II$ is Morita-Rieffel equivalent to
$$
C^*(I_{(V,V')^Y},\alpha^J|_{I_{(V,V')^Y}, D)})
$$ 
where $I_{(V,V')^Y}$ is given by \eqref{crazy ideal} and $\alpha^J:A^J\to A^J$ is  induced by the morphism $(\p^Y, \{\alpha_x\}_{x\in \Delta^Y})$.
\end{prop}
\begin{proof}
By Theorem \ref{pure infiniteness theorem}  every ideal in $C^*(A,\alpha;J)$ is gauge-invariant. Hence  the  assertion follows  from Theorem \ref{gauge-invariant ideals thm} and Lemmas \ref{proposition on restriction and quotient in trivial case} and \ref{lemma without a proof}.
\end{proof}

\begin{cor}\label{corollary from corollary for dolary}
Suppose that $\p$ is free and that $\p(\Delta)$ is open in $X$ (equivalently  $\ker\alpha$ is a complemented  ideal in $A$). We have a bijective correspondence between    ideals $\II$ in $C^*(A,\alpha)$ and invariant sets $V$ for $(X,\varphi)$
established by the relation
$
I_V= A\cap \II.
$
Moreover, for every ideal $\II$ and the corresponding invariant set $V$ we have an isomorphism 
$$
C^*(A,\alpha)/\II \cong C^*(A/I_V,\alpha_{I_V})
$$
and $\II$ is Morita-Rieffel equivalent
 to $C^*(I_V,\alpha|_{I_V}) $.
\end{cor}
\begin{proof}
In the proof of Proposition \ref{proposition on ideals in commutative}, instead of Theorem \ref{gauge-invariant ideals thm} we may  apply Corollary \ref{corollary for dolary}. 
\end{proof}

\subsection{Pure infiniteness and simplicity}
For separable $C^*$-algebras we have the following result concerning permanence of  pure infiniteness and the ideal property. 
\begin{prop}\label{pure infiniteness for Hausdorrf case}
Suppose that $A$ is separable and purely infinite and assume that $\p$ is free. If either $X$ is totally disconnected or there are finitely many $Y$-pairs for $(X,\p)$, then $C^*(A,\alpha;J)$ is purely infinite and has the ideal property. 
\end{prop}
\begin{proof}
If $X$ is totally disconnected, then $A$ has the ideal property by  \cite[Proposition 2.11]{kr}. Thus $C^*(A,\alpha;J)$ is purely infinite and has the ideal property by Theorem \ref{pure infiniteness theorem} (i). If there are finitely many $Y$-pairs for $(X,\p)$, then $C^*(A,\alpha;J)$ is purely infinite and has finitely many ideals by Theorem \ref{pure infiniteness theorem} (ii). Hence  \cite[Proposition 2.11]{kr} implies that $C^*(A,\alpha;J)$ has the ideal property.
\end{proof}

The following characterization of simplicity of $C^*(A,\alpha)$, cf. Remark \ref{remark on simplicity}, is a far  reaching generalization of \cite[Theorem 4.4]{kwa-rever}, proved in the case $A$ is commutative. 

\begin{prop}\label{simplicity result}
The  crossed product $C^*(A,\alpha)$ is simple if and only one of the two possible cases hold:
\begin{itemize}
\item[(i)] $X$ is discrete and $(X,\varphi)$ is (up to conjugacy) either a truncated shift on $\{1,...,n\}$, one sided shift on $\N$, or a two-sided shift on $\Z$,
\item[(ii)] $X$ is not discrete   and $\varphi:X\to X$ is a  surjection such that  there are no non-trivial closed subsets $V$ of $X$ satisfying  $\varphi(V)=V$. 
\end{itemize}
\end{prop}
\begin{proof}
If (i)  or (ii) holds then  there are no non-trivial closed invariant set $V$   in $(X,\varphi)$ and $\p$ is free. Hence $C^*(A,\alpha)$ is simple by Corollary \ref{corollary from corollary for dolary}.

Conversely, suppose that  $C^*(A,\alpha)$ is simple. By Proposition  \ref{proposition on ideals in commutative} there are no non-trivial closed invariant sets in $(X,\varphi)$. The argument in  the proof of \cite[Theorem 4.4]{kwa-rever} shows that either (i) or (ii) holds. In particular, if $\varphi:X\to X$ is a surjection then a closed set $V$ is invariant in $(X,\varphi)$ if and only if  $\varphi(V)=V$.

\end{proof}

\begin{cor}\label{kirchberg classification}
Suppose that  $A$ is  purely infinite,  nuclear and separable. The crossed product $C^*(A,\alpha)$ is a Kirchberg $C^*$-algebra if and only if one of the conditions (i) or (ii) in Proposition \ref{simplicity result} is satisfied.

If $C^*(A,\alpha)$ is a Kirchberg $C^*$-algebra, and  $A$ satisfies the UCT then $C^*(A,\alpha)$ satisfies the UCT.
\end{cor}
\begin{proof}
Plainly, $C^*(A,\alpha;J)$ is separable, and it is nuclear by Proposition \ref{permanence properties prop} (ii). Proposition \ref{pure infiniteness for Hausdorrf case} implies that  $C^*(A,\alpha;J)$ is purely infinite if one of the conditions in Proposition \ref{simplicity result} holds. Hence the assertion follows from Proposition \ref{simplicity result}.

Now suppose that $C^*(A,\alpha)$ is a Kirchberg $C^*$-algebra and  $A$ satisfies the UCT. If  $\p$ is surjective then $(\ker\alpha)^\bot=A$ and if $\p$ is not surjective then $X\setminus\{\p(\Delta)\}=\{x_0\}$, cf.   Proposition \ref{simplicity result} (i), and $A= (\ker\alpha)^\bot\oplus A(x_0)$.  In both cases $(\ker\alpha)^\bot$ satisfies the UCT and thus $C^*(A,\alpha)$ satisfies the UCT by Proposition \ref{permanence properties prop} (iii).
\end{proof}
\subsection{$K$-theory in the case of a trivial bundle}
In this subsection, we assume that the  associated $C^*$-bundle $\A$ is trivial. In other words, we assume that $A=C_0(X,D)$ where $D$ is a simple $C^*$-algebra. By Proposition \ref{proposition for trivial bundles},  $\alpha$ is given by the formula \eqref{endomorphism on C_0(X,A) } where 
$\varphi:\Delta \to X$ is proper continuous map defined on an open subset $\Delta \subseteq X$, and $\Delta \ni x \longrightarrow \al_{x}\in  \End (D)\setminus\{0\}$ is a continuous map. Actually, we make the following  standing assumptions:
\begin{itemize}
\item $A=C_0(X,D)$ where $D$ is a simple $C^*$-algebra,
\item $X$ is totally disconnected, $G:=K_0(D)$ is torsion free and $K_1(D)=\{0\}$.
\end{itemize} 
We treat $G$ as a discrete group and denote by $C_0(X,G)$ the set of  continuous functions $f:X\to G$ such that $f^{-1}(G\setminus\{0\})$ is compact. In other words, any $f\in C_0(X,G)$ is of the form $f=\sum_{i=1}^n\chi_{X_i}\tau_i$ where $X_i$'s are compact and open subsets of $X$ and $\tau_i\in G$.
 We consider  $C_0(X,G)$  an abelian group with the group operation defined pointwise. We also put $C_0(\emptyset,G):=\{0\}$.
\begin{lem}\label{locally constant lemma}
For each $\tau\in G$, the function $\Delta \ni x\mapsto K_0(\alpha_{x})(\tau) \in G$ is continuous.
\end{lem}
\begin{proof} For any projection $p$ in $M_n(D)$, the function $x\mapsto \alpha_x(p) \in M_n(D)$ is continuous. Hence the function $x\mapsto [\alpha_x(p)]_0\in G$ is locally constant. This  implies the assertion. \end{proof}
\begin{defn}\label{K-invariant} 
 Let $\delta_\alpha$ be a group homomorphism $\delta_\alpha:C_0(X,G)\to C_0(X,G)$ given by 
$$
\delta_\alpha(f)(x)=
\begin{cases}
f(x)- K_0(\alpha_{x})(f(\p(x))),  & x \in \Delta
\\
0        & x \notin \Delta.
\end{cases}
$$
Note that  $\delta_\alpha$ is well defined by Lemma \ref{locally constant lemma}.
We define $\delta_\alpha^Y:C_0(X\setminus Y,G)\to C_0(X,G)$ to be the restriction of  $\delta_\alpha$.  We  put
$$
K_0(X,\p,\{\al_{x}\}_{x\in \Delta};Y):=\coker (\delta_\alpha^Y),
\qquad 
K_1(X,\p,\{\al_{x}\}_{x\in \Delta};Y):=\ker (\delta_\alpha^Y).
$$
If $Y=X\setminus \p(\Delta)$ then we write $K_i(X,\p,\{\al_{x}\}_{x\in \Delta}):=K_i(X,\p,\{\al_{x}\}_{x\in \Delta};Y)$, $i=0,1$.
\end{defn}
\begin{prop}\label{K-theory computation}
We have the following isomorphism
\begin{equation}\label{isomorphisms of K-groups for trivial bundles}
K_*(C^*(C_0(X,D),\alpha;J))\cong K_*(X,\p,\{\al_{x}\}_{x\in \Delta};Y).
\end{equation}
\end{prop}
\begin{proof}
Since $G=K_0(D)$ is torsion free, $K_1(C_0(X))=0$ and $K_1(D)=0$, by K\"unneth formulas, see for instance \cite[Proposition 2.11]{Schochet}, we get 
$$
K_0(C_0(X,D))\cong K_0(C_0(X))\otimes G, \qquad K_0(C_0(X,D))=\{0\},
$$ 
where  the isomorphism $\Psi:K_0(C_0(X,D))\to K_0(C_0(X))\otimes K_0(D)$ is determined by the natural identifications
$$
  M_{r}(C_0(X)\otimes D)=C_0(X)\otimes M_r(D) ,\qquad r\in \N.
$$
It is well known, that the maps $\textrm{Proj}(M_r(C_0(X))) \ni p \mapsto \textrm{Tr}\circ p \in C_0(X,\Z)$ determine the isomorphism $K_0(C_0(X))\cong C_0(X,\Z)$, cf. \cite[Exercise 3.4]{Rordam_book}. The formula
$$
C_0(X,\Z)\otimes G \ni f\otimes \tau \longmapsto f_{\tau}  \in C_0(X, G), \textrm{ where } f_{\tau}(x):=f(x)\tau, \,x\in X,
$$
determines an isomorphism $\Phi:C_0(X,\Z)\otimes G  \to C_0(X, G)$, and to see it is enough to note that any element in $C_0(X,\Z)\otimes G$ can be presented as a sum of the form $\sum_{i=1}^n\chi_{X_i} \otimes \tau_i$ where $X_i$'s are compact-open and pairwise disjoint subsets of $X$  and $\tau_i\in G$.

Composing  the aforementioned isomorphisms we conclude that we have the isomorphism
$$K_0(C_0(X,D))\cong C_0(X,G)
$$
whose inverse sends a function $f=\sum_{i=1}^n\chi_{X_i}[p_i]_0 \in C_0(X,G)$, with $X_i$  compact-open and disjoint, and $p_i\in \textrm{Proj}(\otimes M_{r}(D))$, to the element $[\sum_{i=1}^n\chi_{X_i} p_i]_0 \in K_0(C_0(X,D))$. We recall that $A=C_0(X)\otimes D$ and $J=C_0(X\setminus Y)\otimes D$. The above analysis shows that  the following  diagram
$$
\begin{xy}
\xymatrix{
   K_0( J )      \ar[r]^{K_0(\iota)- K_0(\alpha|_{J})}   \ar[d]    &  K_0(A)   \ar[d]   
      \\
  K_0(C_0(X\setminus Y,D))     \ar[r]^{\delta_\alpha}   & K_0(C_0(X,D))} 
  \end{xy} ,
 $$ 
where the vertical arrows are isomorphisms, commutes. Since  $K_1(A)=K_1(J)=0$, by Proposition \ref{Voicu-Pimsner for interacts}, we get  
	$$
	K_0(C^*(A,\alpha;J))\cong \coker \left(K_0(\iota)- K_0(\alpha|_{J})\right)\cong K_0(X,\p,\{\al_{x}\}_{x\in \Delta};Y),
$$ 
$$
K_1(C^*(A,\alpha;J))\cong \ker (K_0(\iota)- K_0(\alpha|_{J}))\cong K_1(X,\p,\{\al_{x}\}_{x\in \Delta};Y).
$$
\end{proof}
\begin{rem}
We note that neither Definition \ref{K-invariant} nor the proof of  Theorem \ref{K-theory computation} makes use of the assumption that  $D$ is simple.
\end{rem}
We are ready to give formulas for $K$-theory of all gauge-invariant ideals and  corresponding quotients in $C^*(C_0(X,D),\alpha;J)$.
\begin{thm}\label{corollary complicated}
If $\p$ is free then all ideals in $C^*(C_0(X,D),\alpha;J)$ are gauge-invariant. In general, if $\II$ is a gauge-invariant ideal in $C^*(C_0(X,D),\alpha;J)$ and $(V,V')$ is the corresponding $Y$-pair for $(X,\p)$, as described in Proposition \ref{proposition on ideals in commutative}, then
\begin{equation}\label{isomorphisms of K-groups for ideals}
K_*(C^*(C_0(X,D),\alpha;J)/\II)\cong K_*(V,\p|_{\Delta\cap V}, \{\al_{x}\}_{x\in \Delta\cap V};V')
\end{equation}
and
\begin{equation}\label{isomorphisms of K-groups for ideals2}
K_*(\II)\cong K_*(U,\p^Y|_{\Delta\setminus (\p^Y)^{-1}(U)\}}, \{\al_{x}\}_{x\in \Delta^Y\setminus(\p^Y)^{-1}(U)}),
\end{equation}
where $(X^Y,\p^Y)$ is the system described in Definition \ref{complementing the kernel2} and $U:=X^Y\setminus (V,V')^Y$ where $(V,V')^Y$ is given by \eqref{stupid set description}. 
If  $V'=V\cap Y$, then 
\begin{equation}\label{isomorphisms of K-groups for ideals3}
K_*(\II)\cong K_*(X\setminus V, \p|_{\Delta\setminus \p^{-1}(V)}, \{\al_{x}\}_{x\in \Delta\setminus \p^{-1}(V)}; X\setminus (V\cup Y)).
\end{equation}
\end{thm}
\begin{proof}
We get  \eqref{isomorphisms of K-groups for ideals} by combining isomorphisms \eqref{isomorphism of the quotient for trivial bundles} and \eqref{isomorphisms of K-groups for trivial bundles}. Since $\I$ is Morita-Rieffel equivalent to 
$
C^*(C_0(U, D),\alpha^J|_{C_0(U, D)})
$ see Proposition \ref{proposition on ideals in commutative}, we get  \eqref{isomorphisms of K-groups for ideals2} by Theorem \ref{K-theory computation} and \cite[Proposition B.5]{katsura}. Similarly, in view of the last part of Proposition \ref{proposition on ideals in commutative}, we obtain  \eqref{isomorphisms of K-groups for ideals3}.
\end{proof}

\begin{cor}\label{complemented kernel K-theory} Suppose that $\varphi$ is free and $\p(\Delta)$ is open in $X$. For any ideal $\II$ in $C^*(C_0(X,D),\alpha)$ we have
$$
K_*(C^*(C_0(X,D),\alpha)/\II)\cong K_*(V,\p|_{\Delta\cap V}, \{\al_{x}\}_{x\in \Delta\cap V})
$$
and
$$
K_*(\II)\cong K_*(X\setminus V, \p|_{\Delta\setminus \p^{-1}(V)}, \{\al_{x}\}_{x\in \Delta\setminus \p^{-1}(V)})
$$
where $I_V=C_0(X,D)\cap \II$.
\end{cor}
\begin{proof} In the proof of Theorem \ref{corollary complicated} apply Corollary \ref{corollary from corollary for dolary} instead of Proposition \ref{proposition on ideals in commutative}.
\end{proof}

We illustrate the above results by  indicating how our construction  can be used to produce non-simple classifiable  $C^*$-algebras  from simple ones, only by adding an appropriate `dynamical ingredient'.  Starting from an arbitrary Kirchberg algebra we construct a classifiable  $C^*$-algebra with a non-Hausdorff primitive ideal space with two points. Such algebras were the first to be considered in classification  of  non-simple infinite $C^*$-algebras \cite{Rordam2}, \cite{bonkat}.

\begin{ex}\label{example for Jamie}
Let $D$ be a Kirchberg algebra that satisfies the UCT and let $X:=\mathcal{C}\cup \{x_0\}$  be a disjoint sum of the Cantor set $\mathcal{C}$ and a clopen singleton $\{x_0\}$. Then $A=C(X,D)$ is a nuclear, separable $C^*$-algebra satisfying the UCT. Suppose that $\p:X\to X$ is such that  $\p(X)=\mathcal{C}$ and $\p:\mathcal{C}\to \mathcal{C}$ is a minimal homeomorphism. Then $\mathcal{C}$ is the only non-trivial closed set invariant in $(X,\p)$. Hence $C^*(A,\alpha)$ has the only one non-trivial ideal $\II$  (in particular, $\Prim(C^*(A,\alpha))$  has two elements and is non-Hausdorff). Note that $I_\mathcal{C}\cong D$ and $\alpha_{I_\mathcal{C}}=0$. Hence $C^*(I_\mathcal{C},\alpha_{I_\mathcal{C}})\cong D$. By Theorem \ref{gauge-invariant ideals thm},  we conclude that  $\I$ is Morita-Rieffel equivalent to $D$, and thus $K_*(\I)=K_*(D)$. The latter fact can be  deduced using our formulas for $K$-theory: since $X\setminus \mathcal{C}=\{x_0\}$ and $\p|_{X\setminus \p^{-1}(\mathcal{C})}=\p|_{\emptyset}$ is the empty map, the domain of the corresponding group homomorphism $\delta_{\alpha|_{I_{\mathcal{C}}}}^{\{x_0\}}$ is the zero group $C_0(\emptyset,G)=\{0\}$ and the codomain is  $C_0(\{x_0\},G)\cong G$.   Moreover, both $C^*(A,\alpha)$ and  $C^*(A,\alpha)/\I\cong C^*(I_\mathcal{C},\alpha_{I_\mathcal{C}})$ satisfy the UCT, see Proposition \ref{permanence properties prop}. Note that  $\II$ is not a complemented ideal in $A$ and hence has no unit.  Concluding, cf. Proposition \ref{pure infiniteness for Hausdorrf case}, we get 
\begin{quotation}
	$C^*(A,\alpha)$   is a strongly purely infinite, nuclear and separable $C^*$-algebra with only one non-trivial ideal $\II$, which is a unique non-unital Kirchberg algebra, satisfying the UCT, with the same $K$-theory as $D$. The non-trivial quotient $C^*(A,\alpha)/\I$ is stably isomorphic to  the unique stable Kirchberg algebra satisfying the UCT with the $K$-theory equal to 
	$
K_*(\mathcal{C},\p|_{\mathcal{C}}, \{\al_{x}\}_{x\in \mathcal{C}})	
	$.
\end{quotation}
Let us comment  on the groups $K_*(C^*(A,\alpha)/\I)\cong 
K_*(\mathcal{C},\p|_{\mathcal{C}}, \{\al_{x}\}_{x\in \mathcal{C}})	
	$.	In the case $G=\Z$ (that is, for instance, when $D=\OO_\infty$) and $K_0(\alpha_x)=id$ for every $x\in \mathcal{C}$, these groups coincide with 
$K$-groups of crossed products studied by Putnam in \cite{Putnam}. In particular, by \cite[Theorem 6.2]{HPS}, $K_0(C^*(A,\alpha)/\I)$ might be any group which can be equipped with a structure of a simple dimension group. This indicates that allowing  the endomorphisms  $K_0(\alpha_x)$ to be non-trivial or $G$ not to be equal to $\Z$, we have a lot of flexibility for constructing systems with different $K_0(C^*(A,\alpha)/\I)$.
\\
Turning to $K_1(C^*(A,\alpha)/\I)\cong 
K_1(\mathcal{C},\p|_{\mathcal{C}}, \{\al_{x}\}_{x\in \mathcal{C}})$, we	note that $f \in K_1(\mathcal{C},\p|_{\mathcal{C}}, \{\al_{x}\}_{x\in \mathcal{C}})	$ if and only if 
\begin{equation}\label{equation for K_1}
f(x)=K_0(\alpha_x)(f(\p(x))), \qquad \text{for every }x\in \mathcal{C}.
\end{equation}
Thus by minimality of $\p:\mathcal{C}\to \mathcal{C}$, we see that $f\in K_1(\mathcal{C},\p|_{\mathcal{C}}, \{\al_{x}\}_{x\in \mathcal{C}})$ is uniquely determined by its value in a fixed point ($0\in \mathcal{C}$, for instance). Therefore we have
$$
K_1(C^*(A,\alpha)/\I)\cong  \{g\in G: \textrm{ there is }f\in C_0(\mathcal{C},G) \textrm{ such that } f(0)=g  \textrm{ and \eqref{equation for K_1} holds} \} .
$$ 
If $K_0(\alpha_x)=id$ for every $x\in \mathcal{C}$, then $K_1(C^*(A,\alpha)/\I)\cong G$. If $K_0(\alpha_x)=0$ for at least one point $x\in \mathcal{C}$, then $K_1(C^*(A,\alpha)/\I)=\{0\}$  (by Lemma \ref{locally constant lemma} and minimality of the system). For the particular case when $G=\Z$,    we have $K_0(\alpha_{x})(\tau)=m_x \cdot \tau$ for all $\tau \in \Z$,  where $m_x\in \Z$ is fixed   for every $x\in \mathcal{C}$ . If at least one of the numbers $m_x$ is different than $\pm 1$ then $K_1(C^*(A,\alpha)/\I)=\{0\}$. Indeed, then there is a non-empty open set $U\subseteq  \mathcal{C}$ and $m\in \Z\setminus \{\pm 1\}$ such that $K_0(\alpha_{x})(\tau)=m \cdot \tau$ for all $x$ in $U$. We may assume that $m\neq 0$. By minimality, the orbit of $0$  visits $U$ infinitely many times. Thus for any $f\in K_1(\mathcal{C},\p|_{\mathcal{C}}, \{\al_{x}\}_{x\in \mathcal{C}})$ the integer $g:=f(0)$ 	is divisible by any power of $m$. This implies that $g=0$.
\end{ex}

Finally, we include a simple example showing explicitly  the dependence of  $K_0(C^*(A,\alpha))$ on the choice of endomorphisms $\alpha_x$, $x\in X$.
\begin{ex}[$K$-theory for finite minimal  systems]\label{example for Me} Let $(X,\p)$ be given by the relations:
 $X=\{1,2,...,n\}$, $n>1$, $\Delta=X\setminus\{1\}$,  $\p(i)={i-1}$, for $i=2,...,n$. 
Let $Y=X\setminus \p(\Delta)=\{n\}$. Assume also that $G=\Z$. Then there are integers $m_2,...,m_{n}$ such that $K_0(\alpha_i)(k)= m_i k$, for all  $i=2,...,n$ and $k\in \Z$. In particular,  identifying $C_0(X,\Z)$  and $C_0(X\setminus Y,\Z)$ with $\Z^n$ and $\Z^{n-1}$ respectively, we see that $\delta_\alpha^Y:\Z^{n-1}\to \Z^n$ is given by the formula
$$
\delta_\alpha^Y(k_1,...,k_{n-1})=(0,k_2- m_2 k_1,...,k_{n-1}-m_{n-1}k_{n-2}, -m_{n}k_{n-1}).
$$
A moment of thought yields that
$$
K_1(C^*(A,\alpha))\cong \ker (\delta_\alpha^Y)\cong \begin{cases} \Z & \text{if }m_i=0\text{ for some }i,
\\
\{0\} & \text{otherwise}.
\end{cases}
$$
The reader may check that $\coker (\delta_\alpha^Y)\cong \Z^2 $ if  $m_i=0$  for some $i$, and  if  the numbers $m_i$ are non-zero then the map $(g_1,...,g_n)\mapsto (g_1,\sum_{i=2}^n(\prod_{j=i+1}m_j^n)g_i)$ factors through to the isomorphism 
$
\coker (\delta_\alpha^Y)\cong \Z\oplus \Z/(m_2 m_3 ... m_n)\Z
$.  Thus we get
$$
K_0(C^*(A,\alpha))\cong 
\Z\oplus \Z/(m_2 m_3 ... m_n)\Z.
$$
\end{ex}

\appendix

\section{Relative Cuntz-Pimsner algebras}\label{appendix section}

A $C^*$-correspondence over a $C^*$-algebra $A$ is a right Hilbert $A$-module $E$ with a left action $\phi_E:A\to \L(E)$ of $A$ on $E$ via adjointable operators. We let $J(E):=\phi_E^{-1}(\K(E))$  to be the ideal in $A$ consisting of  elements that act from the left on $E$ as generalized compact operators. 
For any ideal $J$ in $J(E)$ the  relative Cuntz-Pimnser algebra $\OO(J,E)$ is constructed as a quotient of  the $C^*$-algebra generated by  Fock representation of $E$, see \cite[Definition 2.18]{ms} or \cite[Definition 4.9]{kwa-leb}. The $C^*$-algebra $\OO(J,E)$ is universal with respect to appropriately defined representations  of $E$, see \cite[Remark 1.4]{fmr} or \cite[Proposition 4.10]{kwa-leb}. Namely, a representation $(\pi,\pi_E)$ of a $C^*$-correspondence $E$   consists of a representation $\pi:A\to \B(H)$ in a Hilbert space $H$ and a linear map $\pi_E:E\to \B(H)$ such that
 $$
 \pi_E(ax\cdot b)=\pi(a)\pi_E(x)\pi(b),\quad \pi_E(x)^*\pi_E(y)=\pi(\langle x, y\rangle_A),\quad  a,b\in A, x\in E. 
 $$
Then $\pi_E$ is automatically bounded. If $\pi$ is faithful, then $\pi_E$ is isometric and we say that $(\pi,\pi_E)$ is \emph{injective}.
The $C^*$-subalgebra $\K(E)\subseteq \mathcal{L}(E) $ of \emph{generalized compact operators}  is the closed linear span  of  the operators $\Theta_{x,y}$ where $\Theta_{x,y} (z)=x \langle y,z\rangle_A$ for $x, y, z\in E$. 
 Any  representation $(\pi,\pi_E)$ of $E$ induces a homomorphism
 $(\pi,\pi_E)^{(1)}: \K(E) \to \B(H)$ which satisfies
 $$
 (\pi,\pi_E)^{(1)}(\Theta_{x,y})=\pi_E(x)\pi_E(y)^*, \qquad (\pi,\pi_E)^{(1)}(T)\pi_E(x)=\pi_E(Tx)
 $$
 for $x, y \in E$ and $T\in \K(E)$. The set 
 $$
 I_{(\pi,\pi_E)}:=\{a\in J(E): (\pi,\pi_E)^{(1)}(\phi_E(a))=\pi(a)\}
 $$
 is an ideal in $J(E)$. We call $I_{(\pi,\pi_E)}$ the \emph{ideal of covariance} for $(\pi,\pi_E)$.  For any    ideal $J$ contained in $J(E)$ a representation $(\pi,\pi_E)$ of $E$ is said to be  \emph{$J$-covariant} if
$J\subseteq I_{(\pi,\pi_E)}$. Note that if $(\pi,\pi_E)$ is injective, then we have 
$$
I_{(\pi,\pi_E)}=\{a\in A: \pi(a)\in (\pi,\pi_E)^{(1)}(\K(E))\} \subseteq (\ker\phi_E)^\bot
$$
cf.  \cite[Page 143]{katsura}. 
\begin{prop}\label{existence of relative Cuntz-Pismener algebras} Let $E$ be a $C^*$-correspondence over $A$ and let $J$ be an ideal in $J(E)$. Then there is a $J$-covariant representation  $(\iota,\iota_E)$  of $E$ such that 
\begin{itemize}
\item[i)] $\OO(J,E)$ is generated as a $C^*$-algebra by $\iota(A)\cup \iota_E(E)$,
\item[ii)] for any $J$-covariant representation $(\pi,\pi_E)$ of  $E$ there is a homomorphism $\pi\rtimes_J \pi_E$ of  $\OO(J,E)$ such that $(\pi\rtimes_J \pi_E)\circ\iota = \pi$ and $(\pi\rtimes_J \pi_E)\circ\iota_E = \pi_E$.
\end{itemize} 
Moreover, the representation $(\iota,\iota_E)$ is injective if and only if $J\subseteq (\ker\phi_E)^\bot$.  
\end{prop}
\begin{proof}
The first part of the assertion is \cite[Proposition 1.3]{fmr}. The second part follows from \cite[Proposition 2.21]{ms} and  \cite[Proposition 3.3]{katsura}.
\end{proof}
It follows from the above proposition that  $\OO(J,E)$  is equipped with a gauge circle action which acts as identity on the image of $A$ in $\OO(J,E)$.
We say that a representation $(\pi,\pi_E)$ of $E$ \emph{admits a gauge action} if there exists  a group homomorphism $\beta:\mathbb{T}\to \Aut (C^*(\pi(A)\cup \pi_E))$ such that $\beta_z(\pi(a))=\pi(a)$ and  $\beta_z(\pi_E(x))= z\pi_E(x)$ for all $a\in A$, $x\in E$ and $z\in \mathbb{T}$.
\begin{prop}[Corollary 11.8 in \cite{ka3}]\label{gauge-invariance theorem}
Let us assume that $J$ is an ideal $J(E)\cap(\ker\phi_E)^\bot$. For any injective $J$-covariant representation $(\pi,\pi_E)$ the homomorphism $\pi\rtimes_J \pi_E$ of $\OO(J,E)$ is injective if and only if $I_{(\pi,\pi_E)}=J$ and  $(\pi,\pi_E)$  admits a gauge action.
\end{prop}
 Katsura, in \cite{ka3}, described ideals in $\OO(J,E)$ that are invariant under the gauge action in the following way. For any ideal $I$ in $A$ we define two another ideals 
$$
E(I):=\clsp\{\langle x,a\cdot y\rangle_A\in A: a\in I,\,\, x,y\in E  \}, 
$$ 
$$
E^{-1}(I):=\{a\in A: \langle x,a\cdot y\rangle_A \in I \textrm{ for all } x,y\in E  \}. 
$$
If $E(I)\subseteq I$, then the ideal  $I$ is said to be \emph{positively invariant}, \cite[Definition 4.8]{ka3}. 
For any positively invariant ideal $I$ we have a naturally defined quotient $C^*$-correspondence $E_I=E/EI$ over $A/I$. Denoting by $q_I:A\to A/I$ the quotient map one puts
$$
J_E(I):=\{a\in A: \phi_{E_I}(q_I(a))\in \K(E_I), \, \, aE^{-1}(I) \subseteq I\}.
$$   
\begin{defn}[Definition 5.6 in \cite{ka3}]Let $E$ be a $C^*$-correspondence over a $C^*$-algebra $A$. A \emph{$T$-pair} of
$E$ is a pair $(I,I')$ of ideals $I$, $I'$ of $A$ such that $I$ is positively invariant and
$I\subseteq I'\subseteq  J_E(I)$.
\end{defn}
Exploiting the results of \cite{ka3} we get the following theorem.
\begin{thm}\label{thm for referee}
Let $E$ be a $C^*$-correspondence over a $C^*$-algebra $A$ and $J$ be
an ideal of $A$ contained in $(\ker\phi_E)^\bot \cap J(E)$. Then relations 
 \begin{equation}\label{lattice isomorphisms relations A}
I=A\cap \I, \qquad I'=A\cap (\I+ EE^*) 
\end{equation}
establish a  bijective correspondence between 
$T$-pairs $(I,I')$ for $E$ with $J\subseteq I'$ and  gauge-invariant ideals $\I$ in $\OO(J,E)$.
Moreover, for objects satisfying \eqref{lattice isomorphisms relations A} we have 
$$
\OO(J,E)/ \I\cong \OO(q_I(I'),E_I),
$$
and if  $\I$ is generated (as an ideal) by $I$ then $\I$ is Morita-Rieffel equivalent to $\OO(J\cap I,IE)$.
\end{thm}
\begin{proof}
The first part of the assertion follows from \cite[Proposition 11.9]{ka3}. Now, let $(I,I')$ and $\I$ be the corresponding objects satisfying \eqref{lattice isomorphisms relations A}  and let $q:\OO(J,E)\to \OO(J,E)/ \I$ be the quotient map. Put 
$$
\pi(a+ I):=q(a), \qquad \pi_{E_I}(x+ IE):=q(x), \qquad a\in A,\, x\in E.
$$
Since $I=A\cap \I$, this yields a well defined representation $(\pi,\pi_{E_I})$ of $(I,E_I)$.  Since $I' \subseteq (\I+ EE^*)$ we have $q_I(I')\subseteq  I_{(\pi,\pi_{E_I})}$.   Thus $(\pi,\pi_{E_I})$ gives rise to a  surjection $\OO(q_I(I'),E_I)\to  \OO(J,E)/ \I$. To see it is an isomorphism note that $(\pi,\pi_{E_I})$ admits a gauge action, because $\I$ is gauge-invariant. Moreover, $I'=A\cap (\I+ EE^*)$ implies that $a\in I'$ if and only if  $a+
\I\in \I+ EE^*$, for any $a\in A$. Thus  we get
\begin{align*}
\{q_I(a)\in A/I: \pi(q_I(a))\in  (\pi,\pi_E)^{(1)}(\K(E_I)\} &= \{a+I\in A/I: q(a)\in q(EE^*)\}
\\
&= \{a+I\in A/I: a \in I'\}=q_I(I').
 \end{align*}
Hence by \cite[Corollary 11.8]{ka3} we get  $ \OO(q_I(I'),E_I)\cong \OO(J,E)/ \I$.

Suppose now that $\I$ is generated (as an ideal) by $I$.   The embeddings of $I$ and $IE$ into $\OO(J,E)$ give rise to a faithful representation $(\pi,\pi_{IE})$ of $(I,IE)$ in $\OO(J,E)$. Clearly, $(\pi,\pi_{IE})$ admits a gauge action and we have
\begin{align*}
I_{(\pi,\pi_{IE})}&=\{a\in I: a \in (IE)(IE)^*\}=\{a\in I: a \in EE^*\}
\\
&=\{a\in A: a \in EE^*\}\cap I=J\cap I.
\end{align*}
Hence by \cite[Corollary 11.8]{ka3} we see that the  $C^*$-subalgebra $B$  of $\OO(J,E)$ generated by $I$ and $IE$ is isomorphic to $\OO(J\cap I,IE)$. It is not difficult to see, cf. the proof of \cite[Proposition 9.3]{ka3}, that $B=I \OO(J,E)I$ is  the hereditary subalgebra of $\OO(J,E)$ generated by $I$. Hence  $B\cong \OO(J\cap I,IE)$ is Morita-Rieffel equivalent to the ideal $\I$ generated by $I$.
\end{proof}
We recall  Katsura's version of  the Pimsner-Voiculescu exact sequence for a  $C^*$-correspondence $E$.   We consider the linking algebra $D_E=\K(E\oplus A)$  in the following matrix representation
$$
D_E=\left(\begin{matrix} 
 \K(E) & E
\\
\widetilde{E} & A
\end{matrix}\right),
$$
where $\widetilde{E}=\K(E,A)$ is the dual Hilbert bimodule of $E\cong \K(A,E)$. Let $\iota:J \to A$, $\iota_{11}: \K(E)\to D_E$ and $\iota_{22}: A\to D_E$  be  inclusion maps;   $\iota_{11}(a)=\left(\begin{matrix} 
a& 0
\\
0 & 0
\end{matrix}\right)$, $\iota_{22}(a)=\left(\begin{matrix} 
0& 0
\\
0 & a
\end{matrix}\right)$. 
By \cite[Proposition B.3]{katsura}, $K_i(\iota_{22}):K_i(A)\to K_i(D_E)$, $i=0,1$,	 are isomorphisms.

\begin{thm}[Theorem 8.6 in \cite{katsura}]\label{Pimsner-Voiculescu for correspondences} Within the above notation, the following  sequence is exact: 
\begin{equation}\label{Katsura Pimsner voiculescu sequence}
\begin{xy}
\xymatrix{
      K_0(J) \ar[rrr]_{K_0(\iota)-K_0(\iota_{22})^{-1}\circ K_0(\iota_{11}\circ \phi_E|_{J})} & \qquad & & K_0(A) \ar[rrr]_{K_0(i_A)\,\,\,\,\,\,\,\,\,}  & \qquad  & &   \ar[d]  K_0(\OO(J,E) )
             \\
   K_1(\OO(J,E)) \ar[u]  & & &  K_1(A)  \ar[lll]^{\,\,\,\,\,\,\,\,\,\,\,\,K_1(i_A)}  & &  & \ar[lll]^{\,\,\,K_1(\iota)-K_1(\iota_{22})^{-1}\circ K_1(\iota_{11}\circ \phi_E|_{J})}  K_1(J)
              } 
  \end{xy} 
\end{equation}
\end{thm}

By a \emph{Hilbert $A$-$B$ bimodule} we mean  $E$ which is both a left Hilbert  $A$-module and a right Hilbert  $B$-module  with respective inner products  $\langle \cdot,\cdot \rangle_{B}$ and  ${_{A}\langle} \cdot,\cdot \rangle$ satisfying the  condition:
$
 x \cdot \langle y ,z \rangle_B = {_A\langle} x , y  \rangle \cdot z$, for all $x,y,z\in E$. If additionally, $B=\langle B,B \rangle_{B}$ and  $A={_{A}\langle} A,A \rangle$, we say that $A$ and $B$ are \emph{Morita-Rieffel}  \emph{equivalent} (or \emph{strongly Morita equivalent}). A Hilbert $A$-$A$ bimodule is also called a \emph{Hilbert  bimodule over $A$}.  If $E$ is a Hilbert bimodule over $A$ then it is also a $C^*$-correspondence and Katsura's algebra $\OO_E$ associated to $E$ coincides with the $C^*$-algebra associated to $E$ in \cite{aee}, see \cite[Proposition 3.7]{katsura1}.
 
Suppose that $E$ is a Hilbert bimodule over $A$. Then $E$ induces a \emph{partial homeomorphism} $\widehat{E}$ of $\widehat{A}$ \emph{dual to $E$}, see \cite[Definition 1.1]{kwa-top}.
More specifically, $\langle E ,E \rangle_A$ and ${_A\langle} E , E  \rangle$ are ideals in $A$ and $\widehat{E}:\widehat{\langle E ,E \rangle_A} \to \widehat{{_A\langle} E , E  \rangle}$ is a homeomorphism, which factors through the induced representation functor $E\dashind$. The latter  is defined as follows: if  $\pi:A\to \B(H)$ is  a representation, then $E\dashind(\pi):A\to \B(E\otimes_\pi H)$ is a representation where the Hilbert  space $E\otimes_\pi H$ is generated by simple tensors $x\otimes_\pi h$, $x\in E$, $h\in H$, satisfying
$
\langle x_1\otimes_\pi h_1, x_2\otimes_\pi h_2 \rangle = \langle h_1,\pi(\langle x_1, x_2 \rangle_{\A})h_2\rangle,
$
and
$$
E\dashind(\pi)(a)  (x\otimes_\pi h) = (a x)\otimes_\pi h, \qquad a\in A.
$$ 
 By \cite[Theorems 2.2 and 2.5]{kwa-top} we have the following result.
\begin{thm}\label{theorem A.6} Let $\widehat{E}$ be the partial homeomorphism of  $\widehat{A}$ associated to a Hilbert bimodule $E$. If $\widehat{E}$ is topologically free, then every non-zero ideal in $\OO_E$ has a non-zero intersection of $A$. 
If  $\widehat{E}$ is  free, then every  ideal  in $\OO_E$ is gauge-invariant. 
\end{thm} 
\subsection*{$C^*$-correspondences associated to $C^*$-dynamical systems}
Let us now fix $C^*$-dynamical system $(A,\alpha)$. We associate to $(A,\alpha)$  the $C^*$-correspondence given by 
$$
E_\alpha:=\alpha(A)A, \qquad a \cdot x := \alpha(a)x,\quad x\cdot a := xa,\quad  \langle x,y\rangle_A:= x^*y,
$$
where $a\in A$, $x,y\in E$. Clearly, we have $\ker\phi_{E_\alpha}=\ker\alpha$.
\begin{lem}\label{another lemma without a proof}
We have $J(E_\alpha)=A$ and the map $\K(E_\alpha)\ni \Theta_{x,y}\mapsto xy^* \in \alpha(A)A\alpha(A)$ yields an isomorphism of $C^*$-algebras.
\end{lem}
\begin{proof}
The proof is straightforward and thus left to the reader.
\end{proof}
\begin{prop}\label{identification of crossed products} For any ideal $J$ in $(\ker\alpha)^{\bot}$
there is  a natural isomorphism $C^*(A,\alpha;J)\cong \OO(J,E_\alpha)$. More precisely, the relation 
$$
\pi_{E_\alpha}(x)=U^*\pi(x), \qquad x\in \alpha(A)A,
$$
yields a one-to-one correspondence between representations $(\pi, U)$ of $(A,\alpha)$ and representations $(\pi,\pi_{E_\alpha})$ of ${E_\alpha}$ where $\pi:A\to B(H)$ is a non-degenerate representation. For the corresponding representations we have   
$
I_{(\pi,\pi_{E_\alpha})}=I_{(\pi, U)}$.
\end{prop}
\begin{proof}
 By
\cite[Proposition 3.26]{kwa-exel} crossed product by $\alpha$ treated as a completely positive map coincides with the crossed product considered in the present paper (note that an operator $S$ in \cite{kwa-exel} plays the role of $U^*$). By \cite[Lemma 3.25]{kwa-exel} the GNS $C^*$-correspondence associated to $\alpha$ (treated as a completely positive map) is naturally isomorphic to $E_\alpha$. Thus the assertion follows from \cite[Propositions 3.10]{kwa-exel}.
\end{proof}

Some of the following facts were stated without  proof in \cite[Appendix A]{kwa-rever}. 

\begin{prop}\label{positively invariant ideals  and correspondences}
An ideal $I$  in $A$ is positively invariant for $E_\alpha$ if and only if $I$ is positively invariant in $(A,\alpha)$.
Moreover, if $I$ is positively invariant, then we have  natural identifications:
$$
IE_\alpha=E_{\alpha|_I},\qquad (E_\alpha)_I=E_{\alpha_I}.
$$
\end{prop}
\begin{proof}
Clearly, $\alpha(I)\subseteq I$ if and only if $E_\alpha(I)=A\alpha(I)A \subseteq I$, which proves the first part of the assertion. If $I$ is positively invariant then $\alpha(I)A=\alpha(I)IA=\alpha(I)I$, which allows us the identification $IE_\alpha=E_{\alpha|_I}$.  The natural algebraic isomorphism $E_{\alpha_I}=\alpha_I(A/I)A/I=q_I(\alpha(A)A)\cong  \alpha(A)A/\alpha(A)I=(E_\alpha)_I$ intertwines the operations of $C^*$-correspondences. Hence it is an isomorphism that allows us the identification $(E_\alpha)_I=E_{\alpha_I} $.
\end{proof}

\begin{prop}\label{J-pairs and T-pairs}
Suppose that  $I$ and $I'$ are ideals in $A$. 
 $$ (I,I') \text{ is a }T\text{-pair for } E_\alpha\text{ with } J\subseteq I' \,\, \Longleftrightarrow\,\,  (I,I') \text{ is a }J\text{-pair for } (A,\alpha).
$$
 \end{prop}
\begin{proof}
We have $E^{-1}_{\alpha}(I)=\{a\in A: x^* \alpha(a) y\in I \textrm{ for all } x,y\in \alpha(A)A  \}=\alpha^{-1}(I)$.  By Proposition \ref{positively invariant ideals  and correspondences} we may identify $(E_\alpha)_I$ with $E_{\alpha_I}$. Hence $J((E_\alpha)_I)=A/I$ and we get 
$
J_{E_\alpha}(I)=\{a\in A: a\alpha^{-1}(I)\subseteq I\}$
.
 Thus if we assume that $I$ is positively invariant, cf. Proposition \ref{positively invariant ideals  and correspondences}, then  we get
$$
I \subseteq I'  \subseteq J_{E_\alpha}(I) \,\, \Longleftrightarrow\,\, I'\cap \alpha^{-1}(I)=I.
$$
This  implies the assertion.
\end{proof}

We recall, see \cite[Subsection 3.3]{katsura1} or \cite[Proposition 1.11]{kwa-doplicher}, that a $C^*$-correspondence $E$ over $A$ is a Hilbert bimodule over $A$ if and only if  $\phi_E:(\ker\phi_E)^\bot\cap J(E) \to \K(E)$ is an isomorphism, and then  ${_{A}\langle} x, y \rangle=\phi_E^{-1}(\Theta_{x,y})$. 
\begin{prop}\label{reversible correspondence prop} The $C^*$-correspondence $E_\alpha$ is a Hilbert bimodule over $A$ if and only if $(A,\alpha)$ is reversible and then
$$
{_{A}\langle} x, y \rangle=\alpha^{-1}(xy^*)
$$
where $\alpha^{-1}$ is the inverse to the isomorphism $\alpha:(\ker\alpha)^\bot \to \alpha(A)$. 
\end{prop}
\begin{proof}
Clearly, we have $(\ker\phi_{E_\alpha})^\bot=(\ker\alpha)^\bot$. Thus Lemma \ref{another lemma without a proof} implies that $\phi_{E_\alpha}:(\ker\phi_{E_\alpha})^\bot\cap J(E_\alpha)=(\ker\alpha)^\bot \to \K(E)\cong \alpha(A)A\alpha(A)$ is an isomorphism if and only if the system $(A,\alpha)$ is reversible. 
\end{proof}
  Let us now consider a reversible $C^*$-dynamical system $(A,\alpha)$ and the corresponding Hilbert bimodule  $E_\alpha$. Clearly, ${_A\langle} E_\alpha, E_\alpha\rangle=(\ker\alpha)^{\bot}$ and $\langle E_\alpha, E_\alpha\rangle_A=A\alpha(A)A$. Under the standard identifications we have $
\widehat{\alpha(A)}=\{[\pi]\in \widehat{A}: \pi(\alpha(A))\neq 0\}=\widehat{\langle E, E\rangle_A}
$.  The partial homeomorphism dual to $E_\alpha$ can be identified with the one described in Definition \ref{dual partial homeomorphism}:
  
  \begin{lem}\label{lemma 5.6} Let $(A,\alpha)$ be a reversible $C^*$-dynamical system.  The 
homeomorphisms $\widehat{\alpha}: \widehat{\alpha(A)}\to \widehat{(\ker\alpha)^{\bot}}$ 
and $\widehat{E_\alpha}:\widehat{\langle E_\alpha, E_\alpha\rangle_A}\to \widehat{{_A\langle} E_\alpha, E_\alpha\rangle}$ coincide.
\end{lem}
\begin{proof}
Let $\pi:A\to \B(H)$ be   an irreducible representation such that  $\pi(\alpha(A))\neq 0$.  Then $\widehat{\alpha}([\pi])$ is the equivalence class of the representation $\pi\circ\alpha:A \to \B(\pi(\alpha(A))H)$. Since $\pi(\alpha(A))H=\pi(\alpha(A)A)H$ and 
$
\|\sum_{i} a_i \otimes_\pi h_i\|^2=\|\sum_{i,j} \langle h_i, \pi( a_i^*a_j)h_j\rangle\|= \|\sum_{i} \pi(a_i)h_i\|^2,
$
  for $a_i \in E_\alpha=\alpha(A)A$, $h_i\in H$,  $i=1,...,n$, we see  that $a\otimes_\pi h \mapsto \pi(a)h $  yields a unitary operator $U:E_\alpha\otimes_{\pi} H \to \pi(\alpha(A))H$. Furthermore, for $a\in A$, $b\in \alpha(A)$ and $h\in H$ we have
$$
 [E_\alpha\dashind(\pi)(a) U^*] \pi(b)h= E_\alpha\dashind(\pi)(a)\,  b \otimes_\pi h= (\alpha(a)b)\otimes_\pi h=  [U^* (\pi\circ \alpha)(a)] \pi(b)h.
$$
Hence $U$  intertwines $E_\alpha\dashind$ and  $\pi\circ\alpha$. This proves that  $\widehat{E_\alpha}=\widehat{\alpha}$.
\end{proof}

We get the exact sequence for crossed products by endomorphisms, see Proposition \ref{Voicu-Pimsner for interacts}, by using \eqref{Katsura Pimsner voiculescu sequence} and the following lemma.
\begin{lem}\label{probably non-standard proposition}
Let  $J$ be an ideal in $(\ker\alpha)^\bot)$. With the  notation preceding Theorem \ref{Pimsner-Voiculescu for correspondences} applied to $E_\alpha$ we have
$
K_i(\iota_{22}\circ \alpha|_{J})=K_i(\iota_{11}\circ \phi_{E_\alpha}|_{J})$, $i=1,2$.
\end{lem}
\begin{proof} 
For brevity, we put $E:=E_\alpha$. Let $\flat:E \to  \E$ be the canonical antilinear isomorphism, and let $\alpha^{-1}$ be the inverse to the isomorphism $\alpha:(\ker\alpha)^\bot\to \alpha((\ker\alpha)^\bot)$.  Plainly,   the map
$$
M_2(\alpha(J))\ni \left(\begin{matrix} 
a_{11} & a_{12}
\\
a_{21} & a_{22}
\end{matrix}\right) \stackrel{\Phi}{\longmapsto} \left(\begin{matrix} 
\phi_E(\alpha^{-1}(a_{11})) &  a_{12}
\\
\flat(a_{21}^*) & a_{22}
\end{matrix}\right) \in D_E,
$$
is a homomorphism of $C^*$-algebras. The following diagram commutes:
$$
\begin{xy}
\xymatrix{
       &  J   \ar[dl]_{\iota_{11}\circ \alpha}   \ar[dr]^{\iota_{11}\circ \phi_E} &   
      \\
      M_2(\alpha(J)) \ar[rr]^{\Phi}&  & D_{E}}
  \end{xy}. 
 $$ 
Therefore
$$
K_i(\iota_{11}\circ \phi_E|_{J})=K_i(\Phi\circ \iota_{11}\circ \alpha|_{J}), \qquad i=0,1.
$$
Recall that  for any $C^*$-algebra $B$ the homomorphisms $\iota_{ii}:B\to M_2(B)$,  $i=1,2$, induce the same mappings on $K$-groups.
Thus  $K_i(\iota_{11}\circ \alpha|_{J})= K_i(\iota_{22}\circ \alpha|_{J})$, $i=1,2$. By the form of $\Phi$ we see that  $\Phi\circ\iota_{22}\circ \alpha=\iota_{22}\circ \alpha$ on $J$. Concluding,  for $i=0,1$ we get
 $$
K_i(\iota_{11}\circ \phi_E|_{J})=K_i(\Phi\circ \iota_{11}\circ \alpha|_{J})= K_i(\Phi\circ \iota_{22}\circ \alpha|_{J})=K_i(\iota_{22}\circ \alpha|_{J}).
 $$
\end{proof}

\section*{Acknowledgments}

The author would like to  thank the  referee for her/his valuable comments and suggestions  that resulted in  rewriting of the previous version of the manuscript and including $K$-theory considerations. The author also thanks Jamie Gabe for his   comments and an additional  source of motivation for the present work. This research was supported by the 7th
European Community Framework Programme FP7-PEOPLE-2013-IEF: Marie-Curie Action: Project `OperaDynaDual' number 621724; and in part by NCN;  grant number  DEC-2011/01/D/ST1/04112.


\begin{thebibliography}{99}

  \bibitem{adji}
 S. Adji,  Crossed products of $C^*$-algebras by semigroups of endomorphisms, Ph.D. thesis, University of Newcastle, 1995.

 \bibitem{Ant-Bakht-Leb}  A. B. Antonevich, V. I. Bakhtin, A. V. Lebedev, \emph{Crossed product of $C^*$-algebra by an endomorphism,  
  coefficient algebras and transfer operators},  Sb. Math. \textbf{202}(9) (2011), 1253--1283.


 \bibitem{aee}  B. Abadie, S. Eilers, R. Exel, \emph{Morita equivalence for crossed products by Hilbert $C^*$-bimodules}.
Trans. Amer. Math. Soc. \textbf{350}(8) (1998),  3043--3054. 


\bibitem{Anton_Lebed}  A. B. Antonevich, A. V. Lebedev,  Functional differential equations: I. $C^*$-theory,
   Longman Scientific \& Technical,
 Harlow, Essex, England, 1994.
  
	


	
\bibitem{BFK}	M.  A.  Bastos,  C.  A.  Fernandes,  Y.  I.  Karlovich,  \emph{A $C^*$-algebra of singular integral operators    with
shifts admitting distinct fixed points}, J. Math. Anal. Appl. \textbf{413} (2014), 502--524.

 \bibitem{Blan-Kirch} E. Blanchard, E. Kirchberg, \emph{Non-simple purely infinite $C^*$-algebras: the Hausdorff case}, J. 
Funct. Anal. 207 (2004), 461--513.




 \bibitem{bonkat}
A. Bonkat, Bivariante K-Theorie fur Kategorien projektiver Systeme von C*-Algebren,
Ph.D. Thesis, Westf. Wilhelms-Universitat Munster, 2002.


 \bibitem{Brown-Ped} 
L. G. Brown, G. K. Pedersen, \emph{$C^*$-algebras of real rank zero}, J.
Funct. Anal. \textbf{99} (1991), 131--149.


\bibitem{cuntz} J. Cuntz, \emph{Simple $C^*$-algebras generated by isometries}, Comm. Math. Phys., \textbf{57} (1977),  173--185.






 

\bibitem{Elliot} G. A. Elliot, \emph{Some simple $C^*$-algebras constructed as crossed products with discrete outer automorphism groups}, Publ. Res. Inst. Math. Sci. \textbf{16} (1980), no. 1, 299--311. 

\bibitem{Elliot2} G. A.  Elliot, \emph{On the classification of  $C^*$-algebras of real rank zero}, J. Reine Angew. Math.  \textbf{443} (1993),  179--219. 

\bibitem{ER} G. A. Elliot, M. R\o rdam,  \emph{Classification of certain infinite simple $C^*$-algebras, II}, Comment. Math. Helv. \textbf{70} (1995), 615--638. 


 \bibitem{exel2} R. Exel,  \emph{A new look at the crossed-product of a
$C^*$-algebra by an endomorphism}, Ergodic Theory Dynam. Systems,  \textbf{23} (2003),  1733--1750.


\bibitem{exel3} R. Exel,  M. Laca, J. Quigg, \emph{Partial dynamical systems and $C^*$-algebras generated by partial isometries},
  J. Operator Theory, \textbf{47} (2002), 169--186.


\bibitem{fmr} N. J. Fowler, P. S, Muhly,  I. Raeburn, \emph{Representations of Cuntz-Pimsner algebras}, Indiana Univ. Math. J.  \textbf{52} (2003) 569--605.

\bibitem{gs} T. Giordano, A. Sierakowski, \emph{Purely infinite partial crossed products},  J. Funct. Anal.  \textbf{266}(9)  (2014), 5733--5764.


\bibitem{HPS} R. Herman, I. F. Putnam, C. F. Skau  \emph{Ordered Bratteli diagrams, dimension groups and topological dynamics}, Internat. J. Math. \textbf{3} (1992), 827--864.

\bibitem{HRW} I. Hirshberg, M. R\o rdam, W. Winter, \emph{$C_0(X)$-algebras, stability and strongly self-absorbing $C^*$-algebras}, Math. Ann. \textbf{339} (2007), 695--732.




\bibitem{Jeong} 	Ja A Jeong, \emph{Purely infinite simple $C^*$-crossed products}, Proc. Amer. Math. Soc. \textbf{123} (1995), 3075--3078.

\bibitem{Joli-Robert} P. Jolissaint and G. Robertson, \emph{Simple purely infinite $C^*$-algebras and $n$-filling actions}, J. Funct. Anal. \textbf{175} (2000),  197--213. 



\bibitem{kadison_ringrose} R.V. Kadison, J.R. Ringrose, Fundamentals of the Theory of Operator Algebras. Vol. 2. Advanced Theory, Academic
Press, 1986.

\bibitem{karlovich} Yu. I. Karlovich,  \emph{A local-trajcetory method and iosmorphism theorems for non-local $C^*$-algebras}, In Operator Theory: Advances and Applications \textbf{170}, 137-166, Birkhauser, Basel, 2006.

\bibitem{katsura1}
T. Katsura, \emph{A construction of $C^*$-algebras from $C^*$-correspondences.}  Contemp. Math. vol. 335, pp. 173--182, Amer. Math. Soc., Providence (2003).

\bibitem{katsura} T. Katsura, \emph{On $C^*$-algebras associated with $C^*$-correspondences}, J. Funct. Anal.,   \textbf{217}(2) (2004),  366--401.


\bibitem{ka3} T. Katsura, \emph{Ideal structure of $C^*$-algebras associated with $C^*$-correspondences},
Pacific J. Math. \textbf{230} (2007), 107--145.

\bibitem{kirchberg}
E. Kirchberg, \emph{Das nicht-kommutative Michael-Auswahlprinzip und die Klassifikation
nicht-einfacher Algebren}, $C^*$-Algebras (Munster, 1999), Springer, Berlin, 2000, 92--141.

\bibitem{kr} E. Kirchberg, M. R\o rdam, \emph{Non-simple purely infinite $C^*$-algebras}, Amer. J.
Math. \textbf{122} (2000), 637--666.

\bibitem{ks1} E. Kirchberg and A. Sierakowski, \emph{Filling families and strong pure infiniteness}, preprint, arXive:1503.08519v1.
\bibitem{ks2} E. Kirchberg and A. Sierakowski, \emph{Strong pure infiniteness of crossed products}, preprint, arXive:1312.5195v2.


\bibitem{kwa-logist} B. K. Kwa\'sniewski, \emph{$C^*$-algebras associated with reversible extensions of logistic maps}, Sb. Math. \textbf{203}(10) (2012), 1448--1489.

\bibitem{kwa-doplicher} B. K. Kwa\'sniewski,  \emph{$C^*$-algebras generalizing both relative Cuntz-Pimsner and Doplicher-Roberts algebras},   Trans. Amer. Math. Soc. {\bf 365} (2013), 1809--1873. 

 \bibitem{kwa-top} B. K. Kwa\'sniewski,  \emph{Topological freeness for  Hilbert bimodules}.  Israel J. Math. \textbf{199}(2) (2014),  641--650.


\bibitem{kwa-rever} B. K. Kwa\'sniewski,  \emph{Ideal structure of crossed products by endomorphisms via reversible extensions of $C^*$-dynamical systems},  Internat. J. Math. 26 (2015), no. 3, 1550022 [45 pages]. 




\bibitem{kwa-interact} B. K. Kwa\'{s}niewski, \emph{Crossed products  by interactions and graph algebras},  Integral Equations Operator Theory \textbf{80}(3) (2014),  415-451.




\bibitem{kwa-exel} B. K. Kwa\'{s}niewski,  \emph{Exel's crossed products and crossed products 
by completely positive maps}, to appear in Houston J. Math., arXiv:1404.4929.

\bibitem{kwa-ext} B. K. Kwa\'sniewski, \emph{Extensions of $C^*$-dynamical systems  to systems  with  complete transfer operators}, Math. Notes,  \textbf{98}(3) (2015), 419--428.

 \bibitem{kwa-leb} B. K. Kwa\'sniewski, A. V. Lebedev, \emph{Crossed products by endomorphisms and reduction of relations in relative Cuntz-Pimsner algebras},
J. Funct. Anal. \textbf{264}(8) (2013),  1806-1847.

\bibitem{KS} B. K. Kwa\'sniewski, W. Szyma\'nski, \emph{Pure infiniteness and ideal structure of $C^*$-algebras associated to Fell bundles}, preprint, arxiv: 1505.05202 v1.
 

\bibitem{Laca-Spiel} M. Laca and J. S. Spielberg,
\emph{Purely infinite  $C^*$-algebras from boundary actions of discrete groups}, J. reine angew. Math. \textbf{480} (1996),  125--139.



\bibitem{LR}
J. Lindiarni, I. Raeburn, \emph{Partial-isometric crossed products by semigroups of endomorphisms}, J. Operator
Theory, \textbf{52} (2004), 61--87.

\bibitem{meyer_nest}
R. Meyer, R. Nest, \emph{$C^*$-algebras over topological spaces: filtrated K-theory}, Canad. J. Math. \textbf{64}(2) (2012), 368--408.

\bibitem{ms} P. S. Muhly and B. Solel, \emph{Tensor algebras over $C^*$-correspondences (representations, dilations, and $C^*$-envelopes)}, J. Funct. Anal. {\bf 158} (1998), 389--457.

\bibitem{Ortega-Pardo} E. Ortega, E. Pardo, \emph{Purely infinite crossed products by endomorphisms}, J. Math. Anal. Appl. \textbf{412} (2014) 466--477.

\bibitem{Paschke2} W. L. Paschke,  \emph{The crossed product of a $C^*$-algebra by an endomorphism}, Proc. Amer. math. Soc. \textbf{80} (1980), 113--118.



\bibitem{Pasnicu} C. Pasnicu, \emph{On the AH algebras with the ideal property}, J. Operator Theory \textbf{43}(2) (2000), 389--407.




\bibitem{pp} C. Pasnicu and N. C. Phillips, \emph{Crossed products by spectrally free actions}, J. Funct. Anal. \textbf{269} (2015),  915--967.

\bibitem{Pas-Ror} C. Pasnicu, M. R\o rdam, \emph{Purely infinite $C^*$-algebras of real rank zero}, J.
Reine Angew. Math. \textbf{613} (2007), 51-73.


\bibitem{Putnam} I. Putnam,  \emph{The $C^*$-algebras associated with minimal homeomorphisms of the Cantor set}, Pacific J. Math. \textbf{136} (1989), 329--353.


\bibitem{Rordam1}
 M. R\o rdam,  \emph{Classification of certain infinite simple $C^*$-algebras}. J. Funct. Anal. \textbf{131} (1995), 415--458.
 
 \bibitem{Rordam2}
M. R\o rdam, \emph{Classification of extensions of certain $C^*$-algebras by their six term exact
sequences in $K$-theory}, Math. Ann. \textbf{308}(1) (1997),  93--117.


\bibitem{Rordam_book}
M. R\o rdam, F. Larsen, N. J. Laustsen, An introduction to $K$-theory for $C^*$-algebras, London Math. Society Student Texts 49, Cambridge University Press, 2000.



\bibitem{rordam_sier} M. R\o rdam, A. Sierakowski, \emph{Purely infinite $C^*$-algebras arising from crossed products.} Ergodic Theory Dynam. Systems \textbf{32} (2012), 273--293. 


\bibitem{Schochet}
 C. Schochet,  \emph{Topological methods for -algebras. II. Geometry resolutions and the K\"{u}nneth formula}. Pacific J. Math. 98 (1982), no. 2, 443--458.

\bibitem{Stacey}  P. J. Stacey, \emph{Crossed products of $C^*$-algebras by
 $^*$-endomorphisms}, J. Aust. Math. Soc.  {\bfseries 54}
 (1993), 204--212.
 
 
 
 
 \bibitem{Williams} D. P. Williams, Crossed Products of $C^*$-Algebras, Math. Surveys and Monographs v. 134, Amer. Math. Soc., Providence, Rh.I., 2007.
\end{thebibliography}
\end{document}